\documentclass[10pt]{amsart}

\usepackage{amsmath, amscd, amsthm} 
\usepackage{amssymb, amsfonts, latexsym} 
\usepackage[cmtip,all]{xy}
\usepackage{graphicx}
\usepackage{booktabs}
\usepackage{mathtools}
\usepackage{enumerate}
\usepackage{bbm}
\usepackage{color}
\usepackage[all]{xy} 
\usepackage[margin=1.1in]{geometry}
\usepackage{palatino}
\usepackage{mathrsfs}
\usepackage{booktabs}
\usepackage{aliascnt}
\usepackage[svgnames]{xcolor} 
\usepackage{hyperref}
\usepackage{slashbox}
\definecolor{dark-red}{rgb}{0.4,0.15,0.15}
\hypersetup{
    colorlinks, linkcolor=dark-red,
    citecolor=DarkBlue, urlcolor=MediumBlue}
\usepackage{mathrsfs}
\usepackage[
textwidth=3cm,
textsize=small,
colorinlistoftodos]
{todonotes}
\usepackage{tikz}
\usepackage{tikz-cd}
\usepackage{array}
\usepackage{cleveref}
\usepackage{tabularx}
\usepackage{caption}
\usepackage{array,multirow}
\usepackage{pdflscape}

\usepackage[T1]{fontenc}

\numberwithin{equation}{section} %
\theoremstyle{plain}
\theoremstyle{definition}
 
\newaliascnt{theorem}{equation}  
\newtheorem{theorem}[theorem]{Theorem}  
\aliascntresetthe{theorem}

\newaliascnt{proposition}{equation}  
\newtheorem{proposition}[proposition]{Proposition}
\aliascntresetthe{proposition}

\newaliascnt{lemma}{equation}  
\newtheorem{lemma}[lemma]{Lemma}
\aliascntresetthe{lemma}

\newaliascnt{corollary}{equation}  
\newtheorem{corollary}[corollary]{Corollary}
\aliascntresetthe{corollary}

\newaliascnt{claim}{equation}  

\aliascntresetthe{claim}

\newaliascnt{convention}{equation}  
\newtheorem{convention}[convention]{Convention}
\aliascntresetthe{convention}

\newaliascnt{conjecture}{equation}  

\aliascntresetthe{conjecture}

\newaliascnt{question}{equation}  

\aliascntresetthe{question}

\newaliascnt{defn}{equation}  
\newtheorem{defn}[defn]{Definition}
\aliascntresetthe{defn}

\newaliascnt{example}{equation}  
\newtheorem{example}[example]{Example}
\aliascntresetthe{example}

\theoremstyle{remark}
\newaliascnt{remark}{equation}  
\newtheorem{remark}[remark]{Remark}
\aliascntresetthe{remark}

\newcommand{\aref}[1]{\autoref{#1}}

\newcommand{\KGL}{\mathbf{KGL}}
\newcommand{\MGL}{\mathbf{MGL}}

\newcommand{\KW}{\mathbf{KW}}

\newcommand{\KQ}{\mathbf{KQ}}

\newcommand{\KO}{\mathbf{KO}}

\newcommand{\MZ}{\mathbf{MZ}}

\newcommand{\Sm}{\mathbf{Sm}}
\newcommand{\SH}{\mathbf{SH}}

\newcommand{\USp}{\mathbf{USp}}
\newcommand{\KSp}{\mathbf{KSp}}
\newcommand{\U}{\mathbf{U}}

\newcommand{\Char}{\operatorname{char}}
\newcommand{\Sq}{\mathsf{Sq}}
\newcommand{\QQ}{\mathsf{Q}}
\newcommand{\spec}{\operatorname{Spec}}
\newcommand{\One}{\mathbf{1}}

\newcommand{\et}{\text{\'et}}
\newcommand{\AAA}{{\mathcal A}}
\newcommand{\id}{\operatorname{id}}
\newcommand{\im}{\operatorname{im}}

\newcommand{\tor}{\text{tor}}
\newcommand{\coker}{\operatorname{coker}}

\newcommand{\holim}{\operatorname{holim}}
\newcommand{\Spec}{\operatorname{Spec}}
\newcommand{\SC}{\mathsf{sc}}
\newcommand{\Pic}{\operatorname{Pic}}
\newcommand{\Br}{\operatorname{Br}}
\newcommand{\Div}{\operatorname{Div}}
\newcommand{\rk}{\operatorname{rk}}
\newcommand{\Cliff}{\operatorname{Cl}}

\newcommand{\eff}{\mathrm{eff}}
\newcommand{\tensor}{\otimes}
\newcommand{\GW}{\mathbf{GW}}
\newcommand{\OFS}{{\OO}_{F,\mathcal{S}}}
\newcommand{\OF}{{\OO}_{F}}
\newcommand{\ts}{t_{\mathcal{S}}}
\newcommand{\sls}{s_{\mathcal{S}}}
\newcommand{\wt}[1]{\widetilde{#1}}
\newcommand{\E}{\mathsf{E}}

\newcommand{\f}{\mathsf{f}}
\newcommand{\s}{\mathsf{s}}
\newcommand{\dd}{\mathsf{d}}
\newcommand{\cd}{\mathsf{cd}}
\newcommand{\vcd}{\mathsf{vcd}}
\newcommand{\Mil}{\mathsf{M}}
\newcommand{\pr}{\mathsf{pr}}
\newcommand{\inc}{\mathsf{inc}}

\newcommand{\Gm}{\mathbf{G}_m}

\newcommand{\N}{\mathbb{N}}
\newcommand{\Z}{\mathbb{Z}}
\newcommand{\Q}{\mathbb{Q}}
\newcommand{\C}{\mathbb{C}}
\newcommand{\R}{\mathbb{R}}
\newcommand{\FF}{\mathbb{F}}
\newcommand{\OO}{{\mathcal O}}
\newcommand{\KMW}{K^{\mathsf{MW}}}
\newcommand{\GL}{\mathbf{GL}}

\begin{document}
\title{Hermitian $K$-theory, Dedekind $\zeta$-functions, and quadratic forms over rings of integers in number fields}
\subjclass[2010]{Primary: 	11R42, 14F42, 19E15, 19F27}
\keywords{Motivic homotopy theory, slice filtration, motivic cohomology, algebraic $K$-theory, hermitian $K$-theory, higher Witt-theory, 
quadratic forms over rings of integers, special values of Dedekind $\zeta$-functions of number fields}

\author{Jonas Irgens Kylling}
\address{Department of Mathematics, University of Oslo, Norway}
\email{jonasik@math.uio.no}
\author{Oliver R\"ondigs}
\address{Mathematisches Institut, Universit\"at Osnabr\"uck, Germany}
\email{oroendig@mathematik.uni-osnabrueck.de}
\author{Paul Arne {\O}stv{\ae}r}
\address{Department of Mathematics, University of Oslo, Norway}
\email{paularne@math.uio.no}
\begin{abstract}
We employ the slice spectral sequence, the motivic Steenrod algebra, 
and Voevodsky's solutions of the Milnor and Bloch-Kato conjectures to calculate the hermitian $K$-groups of rings of integers in number fields.
Moreover, 
we relate the orders of these groups to special values of Dedekind $\zeta$-functions for totally real abelian number fields.
Our methods apply more readily to the examples of algebraic $K$-theory and higher Witt-theory, 
and give a complete set of invariants for quadratic forms over rings of integers in number fields.
\end{abstract}
\maketitle

\section{Introduction}
\label{section:introduction}

The themes explored in this paper are Karoubi's hermitian $K$-theory \cite{Karoubi}, 
Lichtenbaum's conjectures on special values of $\zeta$-functions \cite{Lichtenbaum:values1}, \cite{Lichtenbaum:values2}, 
Milnor's conjecture on quadratic forms \cite{Milnor}, \cite{MilnorHusemoller} extended to arithmetic Dedekind domains \cite{Hahn},
and Voevodsky's slice filtration \cite{Voevodsky:open}. 
By explicit slice spectral sequence calculations we identify the hermitian $K$-groups of rings of integers in number fields in terms of motivic cohomology groups.
Voevodsky's proof of Milnor's conjecture on Galois cohomology \cite{Voevodsky:Z/2} combined with Wiles's proof of the main conjecture in Iwasawa theory \cite{Wiles} 
allow us to relate the orders of hermitian $K$-groups to special values of $\zeta$-functions for totally real abelian number fields.
While these beautiful links between number theory and homotopy theory are traditionally expressed in terms of algebraic $K$-theory, 
recent calculations of universal motivic invariants have brought hermitian $K$-theory into focus \cite{April1}. 
One expects that motivic homotopy theory has more to offer in this direction since hermitian $K$-theory is in a precise sense closer to the motivic sphere than algebraic $K$-theory.
A shadow of this is witnessed by the Betti realization functor sending algebraic $K$-theory to topological unitary $K$-theory and hermitian $K$-theory to topological orthogonal $K$-theory.
Via the $J$-homomorphism and the Adams conjecture, 
the latter $8$-periodic theory gives rise to cyclic summands in the stable homotopy groups of the topological sphere whose orders are related to Bernoulli numbers \cite{MR0482758}, 
\cite[Chapters 1, 5]{Ravenel:green}.
Our calculation of the hermitian $K$-groups represents a decisive first step in a quest to establish a motivic analogue of this result over number fields.

Suppose $F$ is a number field with $r_{1}$ (resp.~$r_{2}$) number of real (resp.~pairs of complex conjugate) embeddings into the complex numbers $\C$.  
Let $\mathcal{S}$ be a (not necessarily finite) set of places in $F$ containing the archimedean and dyadic ones.
We denote the ring of $\mathcal{S}$-integers in $F$ by $\OO_{F,\mathcal{S}}$.
Classically, 
the zeroth hermitian $K$-group $\KQ_{0}(\OO_{F,\mathcal{S}})$ is the Grothendieck-Witt ring of symmetric bilinear forms on $\OO_{F,\mathcal{S}}$ \cite{MilnorHusemoller}.
When $\mathcal{S}$ is finite, 
then $\KQ_{n}(\OO_{F,\mathcal{S}})$ is a finitely generated abelian group for all $n\geq 0$. 
Its odd torsion subgroup is the invariant part of the odd torsion subgroup of the algebraic $K$-group $\KGL_{n}(\OO_{F,\mathcal{S}})$ for the involution $M\mapsto {^{t}M}^{-1}$ on $\GL(\OO_{F,\mathcal{S}})$, 
see \cite[Propositions 3.2, 3.13]{BKO}.
Our main focus is on the two-primary subgroup of $\KQ_{n}(\OO_{F,\mathcal{S}})$.
We also identify its $\ell$-primary subgroup, 
for $\ell$ any odd prime number, 
by making use of the Rost-Voevodsky solution of the Bloch-Kato conjecture \cite{Voevodsky:Z/l}. 

In the first main result of this paper we identify the mod $2$ hermitian $K$-groups ${\KQ_{n}(\OO_{F,\mathcal{S}};\Z/2)}$ up to extensions of motivic cohomology groups of Dedekind domains as defined in 
\cite{MR2103541}, \cite{MR1811558}, \cite{Spitzweck}.
Our method of proof reveals for $n\geq 1$ the existence of an $8$-fold periodicity isomorphism 
\begin{equation}
\label{equation:8foldway}
{\KQ_{n}(\OO_{F,\mathcal{S}};\Z/2)}\cong {\KQ_{n+8}(\OO_{F,\mathcal{S}};\Z/2)}. 
\end{equation}
Moreover, 
for all $k\geq 0$, 
we show the vanishing result 
\begin{equation}
\label{equation:vanishingKQ5}
{\KQ_{8k+5}(\OO_{F,\mathcal{S}};\Z/2)}
=
0.
\end{equation}
By the universal coefficient short exact sequence 
\[
0
\to
{\KQ_{n}(\OO_{F,\mathcal{S}})}/2
\to
{\KQ_{n}(\OO_{F,\mathcal{S}};\Z/2)}
\to
{_{2}}{\KQ_{n-1}(\OO_{F,\mathcal{S}})}
\to
0,
\]
the vanishing in \eqref{equation:vanishingKQ5} implies the multiplication by $2$ map on the abelian group ${\KQ_{n}(\OO_{F,\mathcal{S}})}$ is injective when $n=8k+4$ and surjective when $n=8k+5$.

To state our main result for ${\KQ_{n}(\OO_{F,\mathcal{S}};\Z/2)}$,
let $h^{p,q}$ (resp.~$h^{p,q}_{+}$) denote the degree $p$ and weight $q$ mod $2$ (resp.~positive mod $2$) motivic cohomology of $\OO_{F,\mathcal{S}}$ (see \aref{appendixB}).
As usual $\rho$ is the class of $-1$ in $h^{1,1}$.
Let $h^{p,q}/\rho^{i}$ denote the cokernel of $\rho^{i}\colon h^{p-i,q-i} \to h^{p,q}$ and $\ker(\rho^{i}_{p,q})$ denote the kernel of $\rho^{i}\colon h^{p,q} \to h^{p+i,q+i}$.
We denote the Picard group of $\OO_{F,\mathcal{S}}$ by $\Pic(\OO_{F,\mathcal{S}})$, 
and its Brauer group by $\Br(\OO_{F,\mathcal{S}})$.
If $A$ is an abelian group, we let ${_{2}}A$ denote its subgroup of elements of exponent $2$ and $\rk_{2}A$ its $2$-rank.

\begin{theorem}
\label{theorem:mod2hermitiankgroupsintroduction}
Let $\OFS$ be the ring of $\mathcal{S}$-integers in a number field $F$.
The mod $2$ hermitian $K$-groups of $\OO_{F,\mathcal{S}}$ are computed up to extensions by the following filtrations of length $l$. 
\vspace{-0.1in}
\begin{table}[bht]
\begin{center}
\begin{tabular}
[c]{p{0.4in}|p{0.1in}|p{4.0in}}\hline
$n\geq 0$ & $l$ & ${\KQ_{n}(\OO_{F,\mathcal{S}};\Z/2)}$ \\ \hline
$8k$         & $3$ & $\f_{0}/\f_{1}=h^{0,4k}$, $\f_{1}/\f_{2}=h^{1,4k+1}\oplus \ker(\rho_{2,4k+1})$, $\f_{2}=h^{2,4k+2}/\rho$ \\
$8k+1$     & $2$ & $\f_{0}/\f_{1}=h^{0,4k+1}\oplus \ker(\rho_{1,4k+1})$, $\f_{1}=h^{1,4k+2}\oplus \ker(\rho_{2,4k+2})$ \\
$8k+2$     & $3$ & $\f_{0}/\f_{1}=h^{0,4k+1}$, $\f_{1}/\f_{2}=h^{0,4k+2}\oplus h^{1,4k+2}$, $\f_{2}=h^{2,4k+3}$ \\
$8k+3$     & $2$ & $\f_{0}/\f_{1}=h^{0,4k+2}$, $\f_{1}=h^{1,4k+3}$ \\
$8k+4$     & $2$ & $\f_{0}/\f_{1}=h^{0,4k+3}$, $\f_{1}=h^{4,4k+4}$ \\
$8k+5$     & $0$ & $0$ \\
$8k+6$     & $2$ & $\f_{0}/\f_{1}=\ker(\rho_{2,4k+4})$, $\f_{1}=h^{4,4k+5}/\rho^{3}$ \\
$8k+7$     & $2$ & $\f_{0}/\f_{1}=\ker(\rho^{2}_{1,4k+4})$, $\f_{1}=\ker(\rho_{2,4k+5})$\\ \hline
\end{tabular}
\end{center}
\caption{The mod $2$ hermitian $K$-groups of $\OO_{F,\mathcal{S}}$}
\label{table:KQ2groupsOF}
\end{table}
\end{theorem}
\vspace{-0.35in}

\begin{remark}
In \aref{table:KQ2groupsOF}:
$h^{0,q}\cong\Z/2$,
$\ker(\rho_{2,1})=h^{2,1}\cong\Pic(\OO_{F,\mathcal{S}})/2$,
$h^{1,q}\cong\OO_{F,\mathcal{S}}^{\times}/2\oplus {_{2}}\Pic(\OO_{F,\mathcal{S}})$, 
$h^{2,q}\cong\Pic(\OO_{F,\mathcal{S}})/2\oplus {_{2}}\Br(\OO_{F,\mathcal{S}})$ for $q>1$, 
$h^{4,q}/\rho^{3}\cong h^{3,q-1}/\rho^{2}\cong (\Z/2)^{t_{\mathcal{S}}^{+}-t_{\mathcal{S}}}$,
$\ker(\rho_{2,q}) \cong \im (h^{2,q}_{+}\to h^{2,q})$ for $q>1$,
$\ker(\rho^{2}_{1,q}) \cong \im (h^{1,q}_{+}\to h^{1,q})$, 
and
$\ker(\rho_{1,q}) \subseteq \im (h^{1,q}_{+}\to h^{1,q})$. 
Here $t_{\mathcal{S}}^{+}$ is the 2-rank of the narrow Picard group $\Pic_{+}(\OO_{F,\mathcal{S}})$ and $t_{\mathcal{S}}=\rk_{2}\Pic(\OO_{F,\mathcal{S}})$.
The $2$-rank of $\ker(\rho^{2}_{1,q})$ (resp.~$\ker(\rho_{2,q})$) equals $r_{2}+s_{\mathcal{S}}+t_{\mathcal{S}}^{+}$ (resp.~$s_{\mathcal{S}}+t_{\mathcal{S}}-1$), 
where $s_{\mathcal{S}}$ is the number of finite primes in $\mathcal{S}$.
This determines the abelian group ${\KQ_{n}(\OO_{F,\mathcal{S}};\Z/2)}$ up to extensions, 
e.g., 
there is a short exact sequence
\begin{equation*}
\label{equation:sesKQ}
0 \to h^{1,4k+2}\oplus h^{2,4k+2} \to {\KQ_{8k+1}(\OO_{F,\mathcal{S}};\Z/2)} \to h^{0,4k+1}\oplus \ker(\rho_{1,4k+1})\to 0.
\end{equation*}
\end{remark}

\begin{remark}
More generally, 
\aref{theorem:mod2hermitiankgroups} identifies the homotopy groups $\pi_{p,q}\KQ/2$ for the motivic spectrum $\KQ$ representing hermitian $K$-theory over $\OFS$.
When $q\equiv 0,1,2,3\bmod 4$ and $n=p-2q$, 
$\pi_{p,q}\KQ$ coincides with Karoubi's hermitian $K$-groups $\KQ_{n}=\KO_{n}$, $\USp_{n}$, $\KSp_{n}$, and $\U_{n}$ \cite{Karoubi}. 

Calculating the cup-product map $\rho\colon h^{1,q} \to h^{2,q+1}$ is a challenging arithmetic problem.
(See also \aref{lem:rhoh11}.)
While $\rho^{n}\neq 0$ for all $n\geq 1$ if $F$ admits a real embedding, 
it is unknown when $\rho^{2}\neq 0$ over totally imaginary quadratic number fields.
\end{remark}

Moreover,
for every integer $n\geq 1$ we calculate the mod $2^{n}$ algebraic $K$-groups, hermitian $K$-groups, and higher Witt-groups of $\OO_{F,\mathcal{S}}$.
When $\mathcal{S}$ is finite we use this to identify the corresponding $2$-adically completed groups.
With this in hand we are ready to discuss the analytic aspect of these $K$-groups in connection with Dedekind $\zeta$-functions.
For complex numbers $s$ with real part $\mathrm{Re}(s)>1$, recall that 
\begin{equation}
\label{equation:DedekindF}
\zeta_{F}(s)
:=
\sum_{\mathscr{I}\neq 0} (\# \OO_{F}/\mathscr{I})^{-s},
\end{equation}
where the summation is over all nonzero ideals $\mathscr{I}\subseteq\OO_{F}$ and $\# \OO_{F}/\mathscr{I}$ is the absolute norm.
The sum in \eqref{equation:DedekindF} diverges for $s=1$, 
but work of Hecke shows $\zeta_{F}$ admits a meromorphic extension to all of $\C$ which is holomorphic except for a simple pole at $s=1$ with residue expressed by the analytic class number formula
\cite[Chapter VII, \S5]{MR1697859}.

The functional equation relating the values of $\zeta_{F}$ at $s$ and $1-s$ implies the multiplicity ${d}_{q}$ of the zero of $\zeta_{F}$ at $1-q$ ($0<q\in\N$) is equal to $r_{1}+r_{2}-1$ if $q=1$, 
$r_{2}$ if $q$ is even, 
and $r_{1}+r_{2}$ is $q>1$ is odd.
In \cite{borel}, 
Borel showed that ${d}_{q}$ coincides with the rank of $K_{2q-1}(\OO_{F})$ for $q>0$.
Note that ${d}_{q}=0$ occurs only if $F$ is totally real and $q$ even, when $\zeta_{F}(1-q)\in\Q$ by the Siegel-Klingen theorem \cite{MR0133304}
(in this case the Borel regular map is trivial and $\zeta_{F}(1-q)$ equals the leading term $\zeta^{\ast}_{F}(1-q)$).
 
The Birch-Tate conjecture relates the special value $\zeta_{F}(-1)\neq 0$ to algebraic $K$-groups by the formula
\begin{equation}
\label{equation:BTconjecture}
\zeta_{F}(-1)
=
\pm
\frac{\# K_{2}(\OO_{F})}{w_{2}(F)}.
\end{equation}
Here $w_{q}(F)$ is the largest natural number $N$ such that the absolute Galois group of $F$ acts trivially on $\mu_{N}^{\otimes q}$, 
i.e., 
the order of the {\et}ale cohomology group $H^{0}_{\et}(F;\Q/\Z(q))$.
Wiles's proof of the main conjecture in Iwasawa theory \cite{Wiles},
which identifies the characteristic power series of certain inverse limits of $p$-class groups with $p$-adic $L$-series,
implies \eqref{equation:BTconjecture} for totally real abelian number fields via the formula
\begin{equation}
\label{equation:mainconjecture}
\zeta_{F}(1-q)
=
\pm
\frac{\#\prod_{p}H^{2}_{\et}(\OO_{F}[\frac{1}{p}];\Z_{p}(q))}{w_{q}(F)}.
\end{equation}
The factors on the right hand side of \eqref{equation:mainconjecture} identify with motivic cohomology groups via the Milnor and Bloch-Kato conjectures \cite{Voevodsky:Z/2}, \cite{Voevodsky:Z/l} 
(see \cite{MR2103541}).
Following \cite{Rognes-Weibel} for the two-primary part, 
this implies 
\begin{equation}
\label{equation:zetaKtheory}
\zeta_{F}(1-q)
=
\pm
2^{r_{1}}
\frac{\# K_{2q-2}(\OO_{F})}{\# K_{2q-1}(\OO_{F})}
\end{equation}
for $q \geq 2$ even, see \cite[Theorem 0.11]{zbMATH02143527} (the sign in \eqref{equation:zetaKtheory} is positive if $r_{1}$ is even and negative if $r_{1}$ is odd according to the functional equation).
We relate special values of $\zeta$-functions to hermitian $K$-groups in the following amelioration of \cite[Theorem 5.9]{BKSO}.

\begin{theorem}
\label{theorem:zetahermitiankgroupsintroduction}
For $k\geq 0$ and $F$ a totally real abelian number field with ring of $2$-integers $\OF[\frac{1}{2}]$, 
the Dedekind $\zeta$-function of $F$ takes the values
\begin{align*}
\zeta_{F}(-1-4k) 
= & \;
\frac{\# h^{2,4k+3}}{\# h^{1,4k+3}} \cdot \frac{\#\KQ_{8k+2}(\OF[\frac{1}{2}];\Z_{2})}{\#\KQ_{8k+3}(\OF[\frac{1}{2}];\Z_{2})}  \\
\zeta_{F}(-3-4k) 
= & \;
2^{r_{1}}\cdot\#\ker(\rho_{2,4k+4})\cdot\frac{\#\KQ_{8k+6}(\OF[\frac{1}{2}];\Z_{2})}{\#\KQ_{8k+7}(\OF[\frac{1}{2}];\Z_{2})}
\end{align*}
up to odd multiples.
\end{theorem}

Our calculations depend on strong convergence of the slice spectral sequences for mod $2^{n}$ reduced higher Witt-theory $\KW/2^{n}$ over $\OFS$, 
$n\geq 1$.
Showing strong convergence turns out to be related to finding a complete set of invariants for quadratic forms over $\OFS$ involving the fundamental ideal, 
the Brauer group, 
the Clifford invariant, 
and motivic cohomology groups.
In \aref{theorem:milnor-ofs} we show a form of Milnor's conjecture for quadratic forms \cite{Milnor} over $\OFS$ by generalizing the proof given in \cite{slices}.
Our formulation of the Milnor conjecture for quadratic forms involves the slice filtration for higher Witt-theory, 
the element $-1\in h^{0,1}$, 
and the mod $2$ Picard group $h^{2,1}$ of $\OFS$.

\subsection*{Outline of proofs}
Our approach is based on applications of the Milnor and Bloch-Kato conjectures on Galois cohomology and $K$-theory \cite{Voevodsky:Z/2}, \cite{Voevodsky:Z/l}, 
the motivic Steenrod algebra \cite{HKO}, \cite{Voevodsky:steenrod},
the slice filtration \cite{Voevodsky:open}, 
and the identification of the slices of hermitian $K$-theory in \cite{slices}, \cite{April1}.

Recall that the slice filtration for  a motivic spectrum $\E$ gives rise to distinguished triangles in the stable motivic homotopy category
\begin{equation}
\label{equation:slicefiltrationofE}
\f_{q+1}(\E) \to \f_{q}(\E) \to \s_{q}(\E) \to \Sigma^{1,0}\f_{q+1}(\E),
\end{equation}
where $\{\f_{q}(\E)\}$ exhausts $\E$, and the slices $\s_{q}(\E)$ are uniquely determined up to isomorphism by \eqref{equation:slicefiltrationofE}, 
see \cite[Theorem 2.2]{Voevodsky:open}.
Applying the motivic homotopy groups $\pi_{\ast,w}$ to \eqref{equation:slicefiltrationofE} yields an exact couple and an associated slice spectral sequence in weight $w$, 
with $E^{1}$-page 
\begin{equation}
\label{equation:slicespectralsequenceofE}
E^{1}_{p,q,w}(\E)=\pi_{p,w}\s_{q}(\E)\implies \pi_{p,w}(\E).
\end{equation}
The slice $d^{1}$-differential in \eqref{equation:slicespectralsequenceofE} is induced by $\dd^{1}\colon\s_{q}(\E) \to \Sigma^{1,0}\s_{q+1}(\E)$ obtained from \eqref{equation:slicefiltrationofE}.
Here the slices of $\E$ are modules over the zero slice $\s_{0}(\One)$ of the motivic sphere spectrum by \cite[\S6 (iv),(v)]{GRSO}, \cite[Theorem 3.6.13(6)]{Pelaez}.
Let $\f_{q}\pi_{p,w}(\E)$ denote the image of $\pi_{p,w}\f_{q}(\E)$ in $\pi_{p,w}(\E)$.
While $\{\f_{q}\pi_{p,n}(\E)\}$ is an exhaustive filtration of $\pi_{p,n}(\E)$, 
convergence of \eqref{equation:slicespectralsequenceofE} is unclear in general (see \aref{lem:sc-convergence}).

Over any field $F$ of characteristic $\Char(F)\neq 2$ the slices of the motivic spectra of algebraic $K$-theory $\KGL$, 
hermitian $K$-theory $\KQ$, 
and higher Witt-theory $\KW$ are identified in \cite{slices}:
\begin{equation}
\label{equation:algebraicktheoryslices}
\s_{q}(\KGL)
\simeq
\Sigma^{2q,q}\MZ,
\end{equation}
\begin{equation}
\label{equation:hermitianktheoryslices}
\s_{q}(\KQ)
\simeq
\begin{cases}
\Sigma^{2q,q}\MZ
\vee  
\bigvee_{i<\frac{q}{2}}\Sigma^{2i+q,q}\MZ/2 & q\equiv 0\bmod 2 \\
\bigvee_{i<\frac{q+1}{2}}\Sigma^{2i+q,q}\MZ/2 & q\equiv 1\bmod 2, \\
\end{cases}
\end{equation}
\begin{equation}
\label{equation:wittheoryslices}
\s_{q}(\KW)
\simeq
\bigvee_{i\in\Z}\Sigma^{2i+q,q}\MZ/2.
\end{equation}
Here $\MZ$ (resp.~$\MZ/2$) denotes the integral (resp.~mod $2$) motivic cohomology or Eilenberg-MacLane spectrum.
The slices are motives or $\MZ$-modules by \cite[Theorem 2.7]{April1} and the slice $\dd^{1}$-differentials are maps between Eilenberg-MacLane spectra.
These facts make the slice spectral sequence amenable to calculations over base schemes affording an explicit description of the action of the motivic Steenrod algebra on its motivic cohomology ring.
More generally, 
using Spitzweck's work of motivic cohomology in \cite{Spitzweck}, 
after localization the isomorphisms in \eqref{equation:algebraicktheoryslices}, \eqref{equation:hermitianktheoryslices}, and \eqref{equation:wittheoryslices} hold over Dedekind domains of mixed characteristic 
with no residue fields of characteristic $2$,
see \cite[\S2.3]{April1}.

We investigate the convergence properties of \eqref{equation:slicespectralsequenceofE} for $\KGL$, $\KQ$, and $\KW$. 
Earlier convergence results include \cite[Theorem 4]{Levine:sss}, 
\cite[Theorem 3.50]{April1}, 
and 
\cite[Lemma 7.2]{Voevodsky:open}.
Assuming $\vcd(F)<\infty$ and $\Char(F)\neq 2$, 
we show the slice spectral sequence for $\KQ/2^{n}$ is conditionally convergent \cite{Boardman}.
In the proof we note the filtrations of $W(F)/2^{n}$ by the powers of $I(F)/2^{n}$ and of $_{2^{n}}W(F)$ by the powers of $_{2^{n}}I(F)$ are both exhaustive, Hausdorff, and complete.
The Wood cofiber sequence \cite[Theorem 3.4]{slices} identifying $\KGL$ with the cofiber of the Hopf map $\eta$ on $\KQ$ is also used in the argument.
This confirms a special case of Levine's conjecture on convergence of the fundamental ideal completed slice tower \cite{Levine:sss}.

The most technical part of the paper concerns the calculations of the slice $\dd^{1}$-differentials for the mod $2^{n}$ reductions of $\KW$ and $\KQ$.
We give succinct formulas for the $\dd^{1}$'s as quintuples of motivic Steenrod operations generated by $\Sq^{1}$ and $\Sq^{2}$ and motivic cohomology classes of the base scheme.
Over rings of $\mathcal{S}$-integers $\OO_{F,\mathcal{S}}$ in a number field $F$ we show by calculation that the slice spectral sequence for $\KQ/2^{n}$ collapses, 
so that it converges strongly, 
and we identify its $E^{\infty}=E^{n+1}$-page.
This finishes the proof of \aref{theorem:mod2hermitiankgroupsintroduction}. 
Using the Bloch-Kato conjecture for odd primes and finite generation we also deduce an integral calculation of the hermitian $K$-groups of $\OO_{F,\mathcal{S}}$.

Throughout the paper we use the following notation.
\vspace{0.1in}

\begin{tabular}{l|l} 
$\KGL$, $\KQ$, $\KW$ & algebraic $K$-theory, hermitian $K$-theory, higher Witt-theory \\
$\One$, $\MZ$  & motivic sphere, motivic cohomology \\
$\f_{q}(\E)$, $\s_{q}(\E)$ & $q$th effective cover, $q$th slice of a motivic spectrum $\E$ \\
$\Sigma^{p,q}$ & motivic $(p,q)$-suspension \\
$\SH$, $\SH^{\eff}$ & stable motivic homotopy category of a field $F$ or $\OO_{F,\mathcal{S}}$, effective version\\
$K^{\Mil}_{*}$, $k^{\Mil}_{*}$, $\KMW_{*}$ & integral and mod $2$ Milnor $K$-theory, integral Milnor-Witt $K$-theory \\
$\pi_{p,q}(\E) = \E_{p,q}$ & bidegree $(p,q)$ motivic homotopy group of $\E$ \\
$H^{p,q}(X;A)$ & bidegree $(p,q)$ motivic cohomology of $X$ with $A$-coefficients \\
$\OO_{F,\mathcal{S}}$ & ring of $\mathcal{S}$-integers in a number field $F$,  $\mathcal{S} \supset \{2,\infty\}$ \\
$h^{p,q}$, $H_{n}^{p,q}$ & bidegree $(p,q)$ mod $2$ and mod $2^{n}$ motivic cohomology of $F$, $\OO_{F,\mathcal{S}}$ \\
$\rho$, $\tau$ &  the class of $-1$ in $h^{1,1}$ and $h^{0,1}$, respectively \\
$h^{p,q}/\rho^{i}$, $\ker(\rho^{i}_{p,q})$ & cokernel of $\rho^{i}: h^{p-i,q-i} \to h^{p,q}$, kernel of $\rho^{i}: h^{p,q} \to h^{p+i,q+i}$ \\
$\cd_{2}(F)$, $\vcd(F)$ & mod $2$ cohomological and virtual cohomological dimension of a field $F$ \\
\end{tabular}

We use matrices to represent maps between suspensions of free $\MZ/2$-modules;
the entries will be ordered according to the simplicial degrees of the summands. 
\vspace{-0.1in}

\subsection*{Guide to the paper}
\aref{sec:slice} recalls some basic properties of the slice filtration in $\SH$.
Our calculations rely on convergence results shown here for the slice spectral sequences of the mod $2^{n}$ reductions of $\KGL$ (see \aref{theorem:KGL2ncc}), 
$\KW$ (see \aref{theorem:KWmod2nconvergence} and \aref{theorem:KWmod2nconvergence2}), 
and $\KQ$ (see \aref{theorem:best-conv} and \aref{theorem:best-conv2}).
Moreover, 
we study multiplicative properties of the slice spectral sequence for pairings of motivic Moore spectra.
This part allows us to circumvent the lack of an algebra structure on the slice spectral sequence for calculations with mod $2$ reductions of motivic spectra.

One of the points of \aref{section:akt} is to show that our methods provide effective tools for calculations of algebraic $K$-groups
(see \aref{thm:KGL2-E2}, \aref{theorem:KGLnumberfields}, and \aref{theorem:integralKGL}).

In \aref{section:hwtahkt} we first identify the mod $2$ higher Witt-groups and the mod $2$ hermitian $K$-groups of $\OO_{F,\mathcal{S}}$ 
(see \aref{theorem:mod2hermitiankgroupsintroduction}, \aref{thm:KW2-E2}, and \aref{thm:KQ2-E2}).
Second, 
we extend these calculations to mod $2^{n}$ coefficients for all $n\geq 1$ (see \aref{thm:KW2-E2n} and \aref{thm:E2-KQ2n}), 
and consequently to $2$-adic coefficients.
The corresponding calculations for odd-primary coefficients are straightforward.
Using this we deduce an integral calculation of the homotopy groups of $\KQ$ over $\OFS$ in terms of motivic cohomology groups 
(see \aref{theorem:integralKQ} for the odd-primary calculations).

\aref{section:svDzf} relates the orders of the hermitian $K$-groups to special values of Dedekind $\zeta$-functions of number fields 
(see \aref{theorem:zetahermitiankgroupsintroduction} and \aref{theorem:nice}).
We perform some technical work in \aref{section:multstrKQ2n} where we determine the multiplicative structure on the graded slices $\s_*(\KQ/2^{n})$ for $n\geq 2$ (see \aref{thm:KQ2n-mult});
this part is needed to determine extension problems arising in the slice spectral sequence for $\KQ/2^{n}$ (see \aref{theorem:integralKQ} and \aref{section:svDzf}).
In \aref{section:mcatSa} we review background on motivic cohomology and the mod $2$ motivic Steenrod algebra over fields and rings of $\mathcal{S}$-integers 
--- with focus on low weights and coefficient rings ---
which is used throughout our calculations.
Finally,
in \aref{sec:tables} we give charts and tables summarizing the calculations in the main body of the paper.

\subsection*{Relation to other works}
Our results are more complete than the calculations of hermitian $K$-groups in \cite{BK}, \cite{BKO}, \cite{BKO2}, and \cite{BKSO} in the sense that we consider arbitrary number fields.
Our motivic homotopy-theoretic techniques apply more readily to algebraic $K$-theory than to higher Witt-theory and hermitian $K$-theory;
we use this to revisit some of the results in \cite{Kahn97}, \cite{Levine99}, and \cite{Rognes-Weibel} based on the Bloch-Lichtenbaum spectral sequence \cite{BL} 
(which is unpublished and may forever remain so).

\subsection*{Acknowledgements}
The authors acknowledge hospitality and support from Institut Mittag-Leffler in Djursholm and the Hausdorff Research Institute for Mathematics in Bonn, 
and funding by the RCN Frontier Research Group Project number~250399 "Motivic Hopf Equations."
R\"ondigs is grateful for support from the DFG priority program ``Homotopy theory and algebraic geometry''. {\O}stv{\ae}r is supported by a Friedrich Wilhelm Bessel Research Award from the Alexander von Humboldt Foundation and a Nelder Visiting Fellowship from Imperial College London.

\section{The slice spectral sequence and its convergence}
\label{section:sss}
\label{sec:slice}
In this section we discuss the slice filtration in $\SH$ developed in \cite{GRSO}, \cite{Pelaez}, and \cite{Voevodsky:open}.
Over fields for which the filtration of the Witt ring by the powers of the fundamental ideal is complete, exhaustive, and Hausdorff, 
we show conditional convergence of the slice spectral sequence for the mod $2^{n}$ reduction of hermitian $K$-theory (see \aref{theorem:best-conv}).
This lays the foundation for our calculations.
For later reference we summarize results on the multiplicative structure of the slices and the slice spectral sequence.

\subsection{Convergence of the slice spectral sequence}
With reference to \eqref{equation:slicefiltrationofE}, 
recall that a motivic spectrum $\E$ is slice complete \cite[Definition 3.8]{April1} if the homotopy limit
\[
\holim_{q\to\infty}\f_{q}(\E) 
\]
is contractible.
The algebraic $K$-theory spectrum $\KGL$ is slice complete over fields \cite[Lemma 3.11]{April2}.
Recall the coeffective cover $\f^{q-1}(\E)$
--- see \cite[\S3.1]{April1} ---
is defined by the cofiber sequence
\[
\f_{q}(\E) \to \E \to \f^{q-1}(\E),
\]
and the slice completion $\SC(\E)$
--- see \cite[Definition 3.1]{April1} ---
is defined as the homotopy limit
\[
\SC(\E) = \holim_{q \to\infty}\f^{q}(\E).
\]
The slice completion is related to the convergence of the slice spectral sequence \eqref{equation:slicespectralsequenceofE}.

\begin{lemma}[\protect{\cite[Section 3]{April2}}]
For every motivic spectrum $\E$ we have the following.
\begin{itemize}
\item A cofiber sequence
\begin{equation}
\label{equation:coeffcof}
\s_{q}(\E) \to \f^{q}(\E) \to \f^{q-1}(\E) \to \Sigma^{1,0}\s_{q}(\E).
\end{equation}
\item A commutative diagram
\begin{equation}
\label{equation:coeffcof2}
\begin{tikzcd}
\s_{q}(\E) \ar[r]\ar[d] & \f^{q}(\E) \ar[d]\ar[ld] \\
\Sigma^{1,0}\f_{q+1}(\E) \ar[r] & \Sigma^{1,0}\s_{q+1}(\E).
\end{tikzcd}
\end{equation}
\item A map of cofiber sequences
\begin{equation}
\label{equation:coeffcof3}
\begin{tikzcd}
\s_{q}(\E) \ar[r]\ar[d, "="] & \ar[d] \f^{q}(\E)  \ar[r]\ar[d] & \f^{q-1}(\E)\ar[d]\ar[r] & \Sigma^{1,0}\s_{q}(\E)\ar[d, "="] \\
\s_{q}(\E) \ar[r] & \Sigma^{1,0}\f_{q+1}(\E) \ar[r] & \Sigma^{1,0}\f_{q}(\E) \ar[r] &
\Sigma^{1,0}\s_{q}(\E),
\end{tikzcd}
\end{equation}
which induces an isomorphism from the slice spectral sequence \eqref{equation:slicespectralsequenceofE} to the spectral sequence obtained from \eqref{equation:coeffcof}, 
up to reindexing.
\end{itemize}
\label{lem:coeffdiff}
\end{lemma}
\begin{proof}
Both \eqref{equation:coeffcof} and \eqref{equation:coeffcof2} are obtained by filling in the diagram
\[
\begin{tikzcd}
\f_{q+1}(\E) \ar[r]\ar[d] & \E \ar[r]\ar[d,"\id"] & \f^{q}(\E)\ar[d] \\
\f_{q}(\E) \ar[r]\ar[d] & \E \ar[r]\ar[d] & \f^{q-1}(\E)\ar[d] \\
\s_{q}(\E) \ar[r] & * \ar[r] & \Sigma^{1,0}\s_{q}(\E)
\end{tikzcd}
\]
similarly to Verdier's octahedral axiom (TR4) \cite{zbMATH01573275} (because of $*$ all squares in this diagram commute).
Moreover, 
\eqref{equation:coeffcof2} implies \eqref{equation:coeffcof3}, 
and the last claim follows since the spectral sequences have isomorphic $E^{1}$-pages.
\end{proof}

\begin{lemma}[\protect{\cite[\S4]{April2}}]
\label{lem:sc-convergence}
The slice spectral sequence for $\E\in\SH$ is conditionally convergent to
\begin{equation}
\label{equation:ssscc}
E^{1}_{p,q,w}(\E)
= 
\pi_{p,w}\s_{q} (\E) 
\Rightarrow 
\pi_{p,w}\SC(\E).
\end{equation}
For a fixed $w$ this is a half plane spectral sequence with entering $d^{r}$-differentials of degree $(-1,r+1)$.
\end{lemma}
\begin{remark}
\label{cor:strong-conv}
If $E^{r}(\E)=E^{\infty}(\E)$ for some $r\geq 1$, 
then \eqref{equation:ssscc} converges strongly to $\pi_{p,w}\SC(\E)$ by conditional convergence and \cite[Theorem 7.1]{Boardman}.
This applies to all the examples considered in this paper.
\end{remark}

\subsection{Convergence for higher Witt-theory}
Following \cite{slices} we discuss the slice filtration for $\KW/2^{n}$ over a field $F$ of characteristic different than $2$.
Recall that $\f_{q}\pi_{p,w}(\E)$ denotes the image of $\pi_{p,w}\f_{q}(\E)$ in $\pi_{p,w}(\E)$.
The slice spectral sequence for $\E$ converges if for all $p,w,q\in\Z$ we have 
\[
\bigcap_{i \geq 0} \f_{q+i}\pi_{p,w}\f_{q}(\E) = 0.
\]
In this case the exact sequence of \cite[Lemma 5.6]{Boardman}, \cite[Lemma 7.2]{Voevodsky:open}
\[
0 
\to \f_{q}\pi_{p,w}(\KW/2^{n})/\f_{q+1}\pi_{p,w}(\KW/2^{n})
\to E^{\infty}_{p,q,w}(\KW/2^{n})
\to \bigcap_{i\geq 1}\f_{q+i}\pi_{p-1,w}\f_{q+1}(\KW/2^{n}) \to \]
\[
\to \bigcap_{i\geq 0}\f_{q+i}\pi_{p-1,w}\f_{q}(\KW/2^{n}),
\]
yields the short exact sequence
\begin{equation}
\label{equation:sesE00}
0 \to \f_{q+1}\pi_{p,w}(\KW/2^{n}) \to \f_{q}\pi_{p,w}(\KW/2^{n}) \to E^{\infty}_{p,q,w}(\KW/2^{n}) \to 0.
\end{equation}

\begin{lemma}
\label{lem:E/n-vanish}
For all integers $p,w,q\in\Z$, 
assume $\pi_{p,w}\f_{q}(\E)$ is a finite abelian group and  
\[
\bigcap_{i \geq 0} \f_{q+i}\pi_{p,w}\f_{q}(\E) = 0.
\]
Then for $n\geq 2$ we have
\[
\bigcap_{i \geq 0} \f_{q+i}\pi_{p,w}\f_{q}(\E/n) = 0.
\]
\end{lemma}
\begin{proof}
Follows from the universal coefficient sequence since $\f_{q+i}\pi_{p,w}/n(\E)$ and ${}_n \f_{q+i}\pi_{p,w}(\E)$ are finite.
\end{proof}

\begin{lemma}
\label{lem:E00terms1}
If $\# F^{\times}/2 <\infty$ then the slice spectral sequence for $\KW/2^{n}$ converges.
\end{lemma}
\begin{proof}
By $\eta$-periodicity $\Sigma^{1,1}\KW \simeq \KW$ \cite[Example 2.3]{slices} we may assume $w=0$.
When $q \leq 0$ we have
\[
\pi_{p, 0}\f_{q}(\KW)
=
\pi_{p, 0}(\KW)
=
\begin{cases}
W(F) &  p \equiv 0 \bmod 4 \\
0 & \text{otherwise}.
\end{cases}
\]
By the assumption $\pi_{p,0}\s_{q}(\KW)$ and $\pi_{p,0}\f_{q}(\KW)$ are finitely generated abelian groups for all $p,q\in\Z$.
We conclude using \aref{lem:E/n-vanish}.
\end{proof}

\begin{lemma}
\label{lem:E00terms2}
If $\vcd(F)<\infty$ then the slice spectral sequence for $\KW/2^{n}$ converges.
\end{lemma}
\begin{proof}
As in \aref{lem:E00terms1} we may assume $w = 0$ and prove the vanishing
\begin{equation}
\label{equation:vanishing}
\bigcap_{i\geq 1}\f_{q+i}\pi_{p-1,0}\f_{q+1}(\KW/2^{n})=0,
\end{equation}
where $\f_{q+i}\pi_{p-1,0}\f_{q+1}(\KW/2^{n})$ is defined as the image 
$$
\im(\pi_{p-1,0}\f_{q+i}\f_{q+1}(\KW/2^{n})\to\pi_{p-1,0}\f_{q+1}(\KW/2^{n})).
$$
Owing to the natural isomorphism $\f_{m}\f_{n}\simeq \f_{m}$ for $n<m$ we may identify the latter with 
$$
\im(\pi_{p-1,0}\f_{q+i}(\KW/2^{n})\to\pi_{p-1,0}\f_{q+1}(\KW/2^{n})),
i\geq 1.
$$
Since $\pi_{p,0}\KW=0$ when $p\not\equiv 0\bmod 4$, 
we may assume $p\equiv 0,1\bmod 4$.
We claim there is an isomorphism
\begin{equation}
\label{equation:mod2imageisomorphism}
\im(\pi_{p-1,0}(\f_{q+i}(\KW) \to \f_{q}(\KW/2^{n})))
\cong 
\im(\pi_{p-1,0}(\f_{q+i}(\KW/2^{n}) \to \f_{q}(\KW/2^{n}))).
\end{equation}
To prove \eqref{equation:mod2imageisomorphism} we use the naturally induced commutative diagram of universal coefficient short exact sequences
\begin{equation}
\label{equation:mod2diagrams}
\begin{tikzcd}
0 \ar[r] & \pi_{p-1,0}\f_{q+i}(\KW)/2^{n} \ar[r]\ar[d] & \pi_{p-1,0}\f_{q+i}(\KW/2^{n}) \ar[r]\ar[d] &  {}_{2^{n}}\pi_{p-2,0}\f_{q+i}(\KW) \ar[r]\ar[d]& 0 \\
0 \ar[r] & \pi_{p-1,0}\f_{q}(\KW)/2^{n} \ar[r] & \pi_{p-1,0}\f_{q}(\KW/2^{n}) \ar[r] & {}_{2^{n}}\pi_{p-2,0}\f_{q}(\KW) \ar[r] & 0.
\end{tikzcd}
\end{equation}
According to \cite[Lemma 6.13]{slices} the map 
\begin{equation}
\label{equation:lem163}
\pi_{l,0}\f_{q + i}(\KW) \to \pi_{l,0}\f_{q}(\KW)
\end{equation}
is trivial for $l \equiv 1, 2, 3 \bmod 4$.
Hence the rightmost vertical map in \eqref{equation:mod2diagrams} is trivial.
By \cite[Corollary 6.15]{slices} there are isomorphisms 
\begin{equation*}
{}_{2^{n}}\pi_{m,0}\f_{q}(\KW)
\cong
\pi_{m,0}\f_{q}(\KW)
\cong
\pi_{m,0}\f_{q}(\KW)/2^{n}
\cong
h^{q-i,q} \oplus h^{q-i-4,q} \oplus \cdots, 
\end{equation*}
for $m\equiv i \bmod 4$, $q\geq 0$, and $i=1,2,3$.
When $q \geq 1$, 
\cite[Corollary 6.16]{slices} shows there is a naturally split short exact sequence
\begin{equation}
0 \to h^{q-4,q-1} \oplus h^{q-8,q-1} \oplus \cdots \to \pi_{0,0}\f_{q}(\KW) \to \f_{q}\pi_{0,0}(\KW) = I^{q} \to 0.
\label{equation:split-effective-KW}
\end{equation}
This identifies the outer terms of the short exact sequences in \eqref{equation:mod2diagrams}.
The naturally induced diagram
\begin{equation}
\label{equation:slicemod2diagram}
\begin{tikzcd}
\s_{q+i-1}(\KW/2^{n}) \ar[r]\ar[d] &
\Sigma^{1,0}\f_{q+i}(\KW/2^{n}) \ar[d] & \\
\s_{q+i-1}(\Sigma^{1,0}\KW) \ar[r] &
\Sigma^{1,0}\f_{q+i}(\Sigma^{1,0}\KW)
\end{tikzcd}
\end{equation}
in $\SH(F)$ yields the commutative diagram
\begin{equation}
\label{equation:homotopyslicemod2diagram}
\begin{tikzcd}
\pi_{p,0}\s_{q+i-1}(\KW/2^{n}) \ar[r]\ar[d] &
\pi_{p-1,0}\f_{q+i}(\KW/2^{n}) \ar[d] & \\
\pi_{p-1,0}\s_{q+i-1}(\KW) \ar[r] &
\pi_{p-2,0}\f_{q+i}(\KW). 
\end{tikzcd}
\end{equation}
In \eqref{equation:homotopyslicemod2diagram} the left vertical map is a split surjection by \eqref{equation:wittheoryslices} and \eqref{equation:KW/2slices}.
Moreover, 
the lower horizontal map is surjective for $p\equiv 0,1,3\bmod 4$ by \eqref{equation:lem163}.
It follows that 
\begin{equation}
\label{equation:homotopyslicemod2surjection}
\pi_{p,0}\s_{q+i-1}(\KW/2^{n}) \to \pi_{p-2,0}\f_{q+i}(\KW) 
\end{equation}
is surjective for $p\equiv 0,1,3\bmod 4$.
Since $\f_{q+i}(\KW/2^{n})\to\f_{q}(\KW/2^{n})$ factors through $\f_{q+i}(\KW/2^{n})\to\f_{q+i-1}(\KW/2^{n})$ for all $i\geq 1$, 
the image of the upper horizontal map in \eqref{equation:homotopyslicemod2diagram} injects into the kernel of the middle map in \eqref{equation:mod2diagrams},
i.e., 
\begin{equation}
\label{equation:mod2inclusion}
\im(\pi_{p,0}\s_{q+i-1}(\KW/2^{n}) \to \pi_{p-1,0}\f_{q+i}(\KW/2^{n}))
\subseteq
\ker(\pi_{p-1,0}\f_{q+i}(\KW/2^{n})\to\pi_{p-1,0}\f_{q}(\KW/2^{n})).
\end{equation}

From \eqref{equation:homotopyslicemod2surjection} and \eqref{equation:mod2inclusion} we deduce a naturally induced surjection
\begin{equation}
\label{equation:mod2surjectioninclusion}
\ker(\pi_{p-1,0}\f_{q+i}(\KW/2^{n})\to\pi_{p-1,0}\f_{q}(\KW/2^{n}))
\to
{}_{2^{n}}\pi_{p-2,0}\f_{q+i}(\KW)
\cong \pi_{p-2,0}\f_{q+i}(\KW).
\end{equation}
Combined with \eqref{equation:mod2diagrams} this proves \eqref{equation:mod2imageisomorphism}.
Note that $\f_{q+i}(\KW) \to \f_{q}(\KW/2^{n})$ factors as the composite of the canonical maps $\f_{q+i}(\KW) \to \f_{q}(\KW)$ and $\f_{q}(\KW) \to \f_{q}(\KW/2^{n})$.
Using \eqref{equation:lem163} this readily implies \eqref{equation:vanishing} for $p \equiv 0\bmod 4$. %

For $p\equiv 1\bmod 4$ we show
\begin{equation}
\label{equation:kw-p=1}
\bigcap_{i \geq 1} \im(\pi_{p-1,0}\f_{q + i}(\KW) \to \pi_{p-1,0}\f_{q}(\KW)/2^{n}) = 0.
\end{equation}
To begin we first note there is a short exact sequence
\begin{align*}
0 \to \f_{q+i}\pi_{0,0}\f_{q}(\KW) \bigcap 2^{n} \pi_{0,0}\f_{q} (\KW)
&\to \f_{q+i}\pi_{0,0}\f_{q}(\KW) \\
& \to \im(\f_{q+i}\pi_{0,0}\f_{q}(\KW) \to \pi_{0,0}\f_{q}(\KW)/2^{n})
\to 
0.
\end{align*}
For $i \gg 0$ we claim 
\begin{equation}
\label{equation:id-final}
\f_{q+i}\pi_{0,0}\f_{q}(\KW) \bigcap 2^{n} \pi_{0,0}\f_{q}(\KW)
\to \f_{q+i}\pi_{0,0}\f_{q}(\KW)
\end{equation}
is the identity map.
Hence the Milnor exact sequence implies the vanishing in \eqref{equation:kw-p=1}.
Now the leftmost terms in \eqref{equation:split-effective-KW} for $\f_{q+i}(\KW)$ map trivially to \eqref{equation:split-effective-KW} for $\f_{q}(\KW)$.
Thus the image of $\pi_{0,0}\f_{q + i}(\KW)$ in $\pi_{0,0}\f_{q}(\KW)$ is contained in the direct summand $I^{q}$.
From \cite[Lemma 2.1]{AE} we get $I^{i+1} = 2I^{i}$, 
where $I^{i}$ is torsion free for $i \gg 0$ (here we use the assumption $\vcd(F) < \infty$, see also the proof of \aref{theorem:KWmod2nconvergence}).
Hence for $i \gg 0$ the image of $\pi_{0,0}\f_{q + i}(\KW)$ in $\pi_{0,0}\f_{q}(\KW)$ is a multiple of $2^{n}$.
\end{proof}

\begin{theorem}
\label{theorem:convergence1}
For $p\equiv 0\bmod 4$ and $w\in\Z$ there are isomorphisms
\begin{align}
\f_{q} \pi_{p+w,w}(\KW/2^{n}) &\cong \im (I^{q-w} \to W(F)/2^{n}) \label{equation:KW2-filt1}, \\
\f_{q} \pi_{p+w+1,w}(\KW/2^{n}) &\cong {}_{2^{n}} I^{q-w}  \label{equation:KW2-filt2}.
\end{align}
By convention $I^{q-w} = W(F)$ for $q\leq w$.
\end{theorem}
\begin{proof}
We may assume $p=w=0$ by $(4,0)$- and $(1,1)$-periodicity of $\KW$ \cite[\S6.3]{slices}.
To show \eqref{equation:KW2-filt1} we consider the commutative diagram of universal coefficient short exact sequences
\[
\begin{tikzcd}
  0 \ar[r] & \pi_{0,0}\f_{q}(\KW)/2^{n} \ar[r]\ar[d] & \pi_{0,0}\f_{q}(\KW/2^{n}) \ar[r]\ar[d, "\alpha"] & _{2^{n}}\pi_{-1,0}\f_{q}(\KW) \ar[r]\ar[d]& 0 \\
  0 \ar[r] & \pi_{0,0}(\KW)/2^{n} \ar[r, "\cong"] & \pi_{0,0}(\KW/2^{n}) \ar[r] & _{2^{n}}\pi_{-1,0}(\KW)=0 \ar[r] & 0.
\end{tikzcd}
\]
Recall that $\pi_{0,0}(\KW)$ is the Witt ring $W(F)$.
As in \eqref{equation:homotopyslicemod2surjection} there is a surjection
\begin{equation*}
\pi_{1,0}\s_{q-1}(\KW/2^{n}) 
\to 
\pi_{-1,0}\f_{q}(\KW),
\end{equation*}
and as in \eqref{equation:mod2inclusion} there is a natural inclusion
\begin{equation*}
\im(\pi_{1,0}\s_{q-1}(\KW/2^{n}) \to \pi_{0,0}\f_{q}(\KW/2^{n}))
\subseteq
\ker(\pi_{0,0}\f_{q}(\KW/2^{n}) \to \pi_{0,0}(\KW/2^{n})).
\end{equation*}
Similarly to \eqref{equation:mod2surjectioninclusion} and \eqref{equation:mod2imageisomorphism}, 
we obtain a naturally induced surjection 
\begin{equation*}
\ker(\pi_{0,0}\f_{q}(\KW/2^{n})\to\pi_{0,0}(\KW/2^{n}))
\to
\pi_{-1,0}\f_{q}(\KW),
\end{equation*}
and an isomorphism
\begin{equation*} 
\im(\pi_{0,0}(\f_{q}(\KW/2^{n}) \to \KW/2^{n}))
\cong 
\im(\pi_{0,0}(\f_{q}(\KW) \to \KW/2^{n})).
\end{equation*}
The latter group identifies with $\im(I^{q} \to W(F)/2^{n})$ by \cite[Corollary 6.11]{slices}.

To prove \eqref{equation:KW2-filt2}, 
recall from \cite[Corollary 6.16]{slices} the split short exact sequence
\begin{equation}
0 \to h^{q-4,q-1} \oplus h^{q-8,q-1} \oplus \cdots \to \pi_{0,0}\f_{q}(\KW) \to \f_{q}\pi_{0,0}(\KW) = I^{q} \to 0.
\label{equation:split-effective-KW-2}
\end{equation}
By \cite[Lemma 6.4]{slices} we have $\pi_{0,0}\f_{0}(\KW)\cong\f_{0}\pi_{0,0}(\KW)\cong W(F)$.
The mod $2^{n}$ universal coefficient exact sequence shows there is a commutative diagram with surjective vertical maps
\[
\begin{tikzcd}
\pi_{1,0}\f_{q}(\KW/2^{n}) \ar[r]\ar[d] & \pi_{1,0}(\KW/2^{n}) \ar[d, "\cong"] \\
_{2^{n}} \pi_{0,0}\f_{q}(\KW) \ar[r] & _{2^{n}} \pi_{0,0}(\KW).
\end{tikzcd}
\]
The right vertical map is an isomorphism since $\pi_{1,0}(\KW)=0$.
Thus the image of $\pi_{1,0}\f_{q}(\KW/2^{n})$ in $\pi_{1,0}(\KW/2^{n})$ coincides with $_{2^{n}}\f_{q}\pi_{0,0}(\KW)$, 
and our claim follows from \eqref{equation:split-effective-KW-2}.
\end{proof}

\begin{theorem}
\label{theorem:KWmod2nconvergence}
Assuming $\vcd(F)<\infty$ the filtrations of $W(F)/2^{n}$ by $I^{q}(F)/2^{n}$ and of ${}_{2^{n}} W(F)$ by ${}_{2^{n}} I^{q}(F)$ are exhaustive, Hausdorff, and complete.
Hence the slice spectral sequence for $\KW/2^{n}$ is strongly convergent.
\end{theorem}
\begin{proof}
We claim our assumption implies $I^{q}(F(\sqrt{-1}))=0$ for $q\geq \vcd(F)+1$.
By the Milnor conjecture on quadratic forms over fields \cite{MR2276765}, \cite{slices},
we find $I^{q}(F(\sqrt{-1}))=I^{q+i}(F(\sqrt{-1}))$ for $i\geq 1$ since the \'etale cohomology group $H^{q}_{\et}(F(\sqrt{-1});\mu_{2})=0$.
It follows that $I^{q}(F(\sqrt{-1}))=0$ by the Arason-Pfister Haupsatz \cite{Arason-Pfister}.
Thus by \cite[Corollary 35.27]{EKM} we deduce $I^{q}(F)=2I^{q-1}(F)$ is torsion free for $q \gg 0$.
Hence both filtrations in question are finite, and therefore complete and Hausdorff.
The filtrations are exhaustive since $I^{0} = W(F)$.
Our last claim follows in combination with \aref{theorem:convergence1}.
\end{proof}
\begin{remark}
The filtration of ${}_{2^{n}} W(F)$ by ${}_{2^{n}} I^{q}(F)$ is always Hausdorff.
In the filtration of $W(F)/2^{n}$ by $I^{q}(F)/2^{n}$ there is a possibly nonzero $\lim^1$-term that obstructs the Hausdorff condition.
\end{remark}

The proofs of \aref{lem:E00terms2} and \aref{theorem:convergence1} are based on results shown over fields in \cite[\S6]{slices}. 
In the following we extend these results to rings of $\mathcal{S}$-integers in number fields,
assuming $\{2,\infty\}\subset\mathcal{S}$.

\begin{theorem}
\label{thm:KW-OFS}
Over $\OFS$ the $0$th slice spectral sequence for $\KW$ collapses at its $E^2$-page, 
and there are isomorphisms 
\[
E^{\infty}_{p,q,0}(\KW) 
\cong 
\begin{cases}
h^{q,q} & p \equiv 0 \bmod 4, q \neq 2 \\
h^{2,2}/\tau & p \equiv 0 \bmod 4, q = 2 \\
h^{2,1} & p \equiv 3 \bmod 4, q = 1 \\
0 & \text{otherwise}.
\end{cases}
\]
\end{theorem}
\begin{proof}
The proof is similar to the calculations for fields in \cite[Theorem 6.3]{slices} with the exceptions that $E^{\infty}_{4p,2,0}(\KW)(\OFS) \cong h^{2,2}/\tau$ by \aref{lem:pic-tau} and 
$E^{\infty}_{4p+3,1,0}(\KW)(\OFS) \cong h^{2,1}\cong \Pic(\OFS)/2$.
The $\dd^{1}$-differentials take the same form as in \cite[Theorem 5.3]{slices} by base change, 
see the proof of \aref{thm:KW2-diff}.
Thus $E^{1}_{p,q,0}(\KW)(F)$ and $E^{1}_{p,q,0}(\KW)(\OFS)$ agree in all degrees with the exception of 
$$
E^{1}_{2p+1,1,0}(\KW)(\OFS) \cong h^{2,1}\oplus h^{1,1}\oplus h^{0,1}.
$$
The summand $h^{2,1}$ of $E^{1}_{4p+1,1,0}(\KW)(\OFS)$ supports a $d^{1}$-differential given by $\tau$-multiplication, 
which is injective by \aref{lem:pic-tau}.
The $d^{1}$-differential on the summand $h^{2,1}$ of $E^{1}_{4p+2,1,0}(\KW)(\OFS)$ is trivial.
This yields the claimed $E^{2}=E^{\infty}$-page along the lines of \cite[Theorem 6.3]{slices}.
\end{proof}

We refer to \cite{Deglise-Jin-Khan} for the construction of the ``defect of purity'' transformation
\[
\Sigma^{-2,-1}i^{\ast}(-) \to i^{!}(-).
\]
Following Quillen's purity theorem for algebraic $K$-theory we will make use of the following special case of absolute purity for hermitian $K$-theory.

\begin{theorem}
\label{theorem:KQKWpurity}
Let $i\colon x\to\Spec(\OFS)$ be the inclusion of a closed point $x\not\in\mathcal{S}$.
Then in $\SH(k(x))$ there exist absolute purity isomorphisms 
\begin{equation}
\label{equation:KQ-purity}
\Sigma^{-2,-1}i^{\ast}(\KQ) \xrightarrow{\simeq} i^{!}(\KQ)
\end{equation}
and
\begin{equation}
\label{equation:KW-purity}
\Sigma^{-2,-1}i^{\ast}(\KW) \xrightarrow{\simeq} i^{!}(\KW).
\end{equation}
\end{theorem}
\begin{proof}
This is shown for $\KQ$ in \cite{Deglise-Jin-Khan}.
The case of higher Witt-theory follows since $\KW = \KQ[\frac{1}{\eta}]$ and $i^{!}$ commutes with sequential colimits \cite[Proposition 5.4.7.7]{MR2522659}.
\end{proof}

\begin{theorem}
\label{theorem:MZpurity}
With the notation in \aref{theorem:KQKWpurity} there is an absolute purity isomorphism 
\begin{equation}
\label{equation:MZ-purity}
\Sigma^{-2,-1}i^{\ast} \s_{\ast}(\KW) \xrightarrow{\simeq} i^{!}\s_{\ast}(\KW).
\end{equation}
\end{theorem}
\begin{proof}
Combine \eqref{equation:wittheoryslices} with absolute purity for motivic cohomology in \cite[Corollary 3.2]{Spitzweck}. 
\end{proof}

\begin{lemma}
\label{lemma:KWOFS}
There are isomorphisms 
\begin{equation}
\label{equation:KWOFS}
\pi_{p,q}(\KW)(\OFS)
\cong
\begin{cases}
W(\OFS) & p-q\equiv 0\bmod 4 \\
\Pic(\OFS)/2 & p-q\equiv 3\bmod 4 \\
0 & \text{otherwise}.
\end{cases}
\end{equation}
\end{lemma}
\begin{proof}
Combine the exact sequence and vanishing in \cite[Corollary 92]{Balmer} with the Knebusch-Milnor exact sequence
\begin{equation}
\label{equation:KMes}
0
\to
W(\OFS)
\to
W(F)
\to
\bigoplus_{x\not\in\mathcal{S}} W(k(x))
\to
\Pic(\OFS)/2 
\to
0
\end{equation}
from \cite[p.93]{MilnorHusemoller}, \cite[p.227]{Scharlau}.
\end{proof}

\begin{lemma}
\label{lem:lemma65}
Over $\OFS$ the slice filtration for $\KW$ induces a commutative diagram
\[
\begin{tikzcd}
\pi_{0,0}\f_{1}(\KW) \ar[r]\ar[d] & \pi_{0,0}\f_{0}(\KW)\ar[d]\\
I(\OFS) \ar[r] & W(\OFS).
\end{tikzcd}
\]
Here $I(\OFS)$ is the kernel of the rank map $\rk_{2}\colon W(\OFS) \to\Z/2$.
By multiplicativity of the slice filtration this yields an inclusion $I^{q}(\OFS) \subseteq\f_{q}\pi_{0,0}(\KW)(\OFS)$.
\end{lemma}
\begin{proof}
This is shown over $F$ in \cite[Lemma 6.4]{slices}.
Our first claim follows since $W(\OFS) \to W(F)$ is injective --- see \cite[Corollary IV.3.3]{MilnorHusemoller}, \cite[Theorem 6.1.6]{Scharlau} ---
and $\pi_{0,0}\s_{0}(\KW)(\OFS) \to \pi_{0,0}\s_{0}(\KW)(F) \cong h^{0,0}$ is an isomorphism.
The second claim follows as in \cite[Corollary 6.5]{slices}.
\end{proof}

\begin{lemma}
\label{lemma:important}
The composite map 
\[
\pi_{-1,0}\f_{1}(\KW) \to \pi_{-1,0}\s_{1}(\KW) 
\cong h^{2,1} \oplus h^{0,1} 
\overset{\pr}{\to} h^{2,1}
\]
is an isomorphism.
\end{lemma}
\begin{proof}
In the proof we make use of \aref{lemma:KWOFS}.
The naturally induced map 
\[
\pi_{3,0}\f_{1}(\KW) \to \pi_{3,0}\f_{0}(\KW),
\]
where $\pi_{3,0}\f_{0}(\KW) \cong h^{2,1}$,
follows from the long exact sequence
\[
\pi_{4,0}\f_{0}(\KW) \to \pi_{4,0}\s_{0}(\KW) \to \pi_{3,0}\f_{1}(\KW) \to \pi_{3,0}\f_{0}(\KW) \to \pi_{3,0}\s_{0}(\KW) = 0,
\]
and the natural surjection $W(\OFS) \cong \pi_{4,0}\f_{0}(\KW) \to \pi_{4,0}\s_{0}(\KW) \cong h^{0,0}$ by \aref{lem:lemma65}.

For a closed point $i\colon x \to \Spec \OFS$ there is a commutative diagram
\[
\begin{tikzcd}
\pi_{-1,0}i_*i^{\ast}\Sigma^{-2,-1}(\KW) &  \ar[l, "=" above] \pi_{-1,0}i_*i^{\ast}\Sigma^{-2,-1}(\KW) \ar[r, "\cong"] & \pi_{-1,0}i_!i^{!}(\KW) \\
\pi_{-1,0}\f_{0}i_*i^{\ast}\Sigma^{-2,-1}(\KW)\ar[d, "f_{1}"]\ar[u, "g_{1}"] &  \ar[l] \pi_{-1,0}i_*i^{\ast}\Sigma^{-2,-1}\f_{1}(\KW) \ar[r]\ar[d, "f_{2}"]\ar[u, "g_{2}"] & \pi_{-1,0}i_!i^{!}\f_{1}(\KW) \ar[d, "f_3"]\ar[u, "g_3"] \\
\pi_{-1,0}\s_{0}i_*i^{\ast}\Sigma^{-2,-1}(\KW) &  \ar[l, "h_{1}" above] \pi_{-1,0}i_*i^{\ast}\Sigma^{-2,-1}\s_{1}(\KW) \ar[r, "\cong"] & \pi_{-1,0}i_!i^{!}\s_{1}(\KW).
\end{tikzcd}
\]
Here $h_{1}$ is obtained using \cite[Lemma 4.2.23]{Kelly}.
Absolute purity as in \aref{theorem:KQKWpurity} and \aref{theorem:MZpurity} imply the indicated isomorphisms.
The maps $g_{1}$ and $g_{2}$ are isomorphisms since $\pi_{-1,0}\f_{0}(-) = \pi_{-1,0}(-)$.
Since $\pi_{q,0}i_!i^{!}\s_{0}(\KW) = 0$ when $q=-1,-2$ we get a bijection $\pi_{-1,0}i_!i^{!}\f_{1}(\KW)\to\pi_{-1,0}i_!i^{!}\f_{0}(\KW)$.
It follows that $g_{3}$ is an isomorphism by contemplating the diagram with exact rows
\[
\begin{tikzcd}
0 \ar[r] & \pi_{0,0}\f_{0}(\KW) \ar[r]\ar[d, "\cong"] & \pi_{0,0}j_*j^{\ast}\f_{0}(\KW)\ar[r]\ar[d, "\cong"] & \oplus_{x\not\in\mathcal S}\pi_{-1,0}i_!i^{!}\f_{0}(\KW) \ar[r]\ar[d] & \pi_{-1,0}\KW \ar[r]\ar[d, "\cong"] & 0 \\
0 \ar[r] & \pi_{0,0}\KW \ar[r] & \pi_{0,0}j_*j^{\ast}(\KW) \ar[r] & \oplus_{x\not\in\mathcal S}\pi_{-1,0}i_!i^{!}(\KW) \ar[r] & \pi_{-1,0}\KW \ar[r] & 0.
\end{tikzcd}
\]
To show that $h_{1}$ is an isomorphism we reduce to showing $h_{1}': \pi_{0,0}i^{\ast}\s_{0}(\KW) \to \pi_{0,0}\s_{0}i^{\ast}(\KW)$ is an isomorphism, 
and conclude using the naturally induced commutative diagram
\[
\begin{tikzcd}
\pi_{0,0}\s_{0}i^{\ast}(\KW) & \ar[l, "h_{1}'" above] \pi_{0,0}i^{\ast}\s_{0}(\KW) & \pi_{0,0}\s_{0}(\KW)\ar[l,"\cong" above] \\
\pi_{0,0}\f_{0}i^{\ast}(\KW)\ar[u]\ar[dr, "\cong"] & \ar[l] \pi_{0,0}i^{\ast}\f_{0}(\KW) \ar[u]\ar[d]& \pi_{0,0}\f_{0}(\KW) \ar[u]\ar[l]\ar[d, "\cong"] \\
& \pi_{0,0}i^{\ast}(\KW) & \pi_{0,0}\KW. \ar[l]
\end{tikzcd}
\]
Here $1\in \pi_{0,0}\KW$ maps to the units in $\pi_{0,0}\s_{0}i^{\ast}(\KW)$ and $\pi_{0,0}i^{\ast}\s_{0}(\KW)$.
The map $f_{1}$ is surjective since it identifies with the rank map $\rk_{2}\colon W(\OFS)\to \Z/2$ as in the proof of \cite[Lemma 6.4]{slices}.
It follows that $f_{2}$ and $f_{3}$ are surjective.

Next we consider the commutative diagram with exact rows and columns
\[
\begin{tikzcd}
\oplus_{x\not\in\mathcal S}\pi_{-1,0}i_!i^{!}\f_{1}(\KW) \ar[r, "f_3"]\ar[d] & \oplus_{x\not\in\mathcal S}\pi_{-1,0}i_!i^{!}\s_{1}(\KW) \ar[r]\ar[d] & \oplus_{x\not\in\mathcal S}\pi_{-2,0}i_!i^{!}\f_{2}(\KW) \ar[d] \\
\pi_{-1,0}\f_{1}(\KW) \ar[r]\ar[d] & \pi_{-1,0}\s_{1}(\KW) \cong h^{2,1}\oplus h^{0,1} \ar[r]\ar[d] & \pi_{-2,0}\f_{2}(\KW) \ar[d] \\
\pi_{-1,0}j_*j^{\ast}\f_{1}(\KW) = 0 \ar[r] & \pi_{-1,0}j_*j^{\ast}\s_{1}(\KW) \cong h^{0,1} \ar[r, "\cong"] &\pi_{-2,0}j_*j^{\ast}\f_{2}(\KW)\cong h^{0,1}.
\end{tikzcd}
\]
Since $f_{3}$ is surjective we conclude $\pi_{-2,0}i_!i^{!}\f_{2}(\KW) = 0$.
Chasing the latter diagram shows 
$$
\pi_{-2,0}\f_{2}(\KW) 
\cong 
h^{0,1}
$$
and the composite map 
\[\pi_{-1,0}\f_{1}(\KW) \to \pi_{-1,0}\s_{1}(\KW) \cong h^{2,1} \oplus h^{0,1} \to h^{2,1}
\]
is surjective.
\end{proof}

\begin{lemma}
\label{lem:OFS-KW-conv}
Over $\OFS$ we have
\[
\bigcap_{i\geq0} \f_{q+i}\pi_{p,w}\f_{q}(\KW) = 0.
\]
Hence the slice spectral sequence for $\KW$ is convergent.
\end{lemma}
\begin{proof}
By $\eta$-periodicity of $\KW$ \cite[Example 2.3]{slices} we are reduced to showing
\begin{equation}
\label{equation:toshow}
\bigcap_{i\geq0} \f_{q+i}\pi_{p,0}\f_{q}(\KW) = 0
\end{equation}
for $p = 0, 1, 2, 3$.

When $p = 0, 1$, \aref{thm:KW-OFS} shows $E^{\infty}_{p+1,q,0}(\KW) = 0$.
As in \cite[Corollary 6.16]{slices} we deduce that $\f_{q+1}\pi_{p,0}\f_{q}(\KW) \to \pi_{p,0}\f_{q-1}(\KW)$ is injective.
Hence, we obtain the injection
\begin{equation}
\label{equation:KWinjection}
\bigcap_i \f_{q+i} \pi_{p,0}\f_{q+1}(\KW)
\to
\bigcap_i \f_{q+i} \pi_{p,0}\f_{q}(\KW).
\end{equation}
The target of \eqref{equation:KWinjection} is trivial by induction. 
When $p = 0$ recall that the natural map $W(\OFS)\to W(F)$ is injective \cite[Corollary IV.3.3]{MilnorHusemoller} and that it factors via the slice filtration,
i.e.,
there is a canonical map $\f_{q}(\KW_{\OFS}) \to j_*\f_{q}(\KW_{F})$ for the generic point $j:\Spec F \to \Spec \OFS$.
The group $\pi_{1,0}(\KW) = 0$ according to \aref{lemma:KWOFS}.

When $p = 2, 3$ we consider the unrolled exact couple for $\KW$
\[
\begin{tikzcd}
& \vdots \ar[d] & \vdots \ar[d] & \vdots \ar[d] \\
& \pi_{4,0}\s_{1}(\KW) \ar[d, "0"] & \pi_{4,0}\s_{0}(\KW) \ar[d, "0"] & \pi_{4,0}\s_{-1}(\KW)\ar[d, "0"] \\
\dots \ar[r] & \pi_{3,0}\f_{2}(\KW) \ar[r]\ar[d, "0"] & \pi_{3,0}\f_{1}(\KW) \ar[r, "\cong"]\ar[d] &  \pi_{3,0}\f_{0}(\KW)\ar[d, "0"] \\
& \pi_{3,0}\s_{2}(\KW) \ar[d] & \pi_{3,0}\s_{1}(\KW) \ar[d] &  \pi_{3,0}\s_{0}(\KW)\ar[d, "0"] \\
\dots \ar[r] & \pi_{2,0}\f_3(\KW) \ar[r]\ar[d] & \pi_{2,0}\f_{2}(\KW) \ar[r]\ar[d] &  \pi_{2,0}\f_{0}(\KW)\ar[d, "0"] \\
& \vdots & \vdots & \vdots
\end{tikzcd}
\]
Inspection of the differentials shows the indicated trivial maps,
see \cite[Theorem 5.3]{slices}.
By \aref{lemma:important} there is an isomorphism
$$
\coker(\pi_{3,0}\f_{1}(\KW) \to \pi_{3,0}\s_{1}(\KW)) 
\cong 
h^{0,1}.
$$
This implies the vanishing $\pi_{3,0}\f_{2}(\KW) = 0$.
Since $E^{\infty}_{3,q,0}(\KW) = 0$ for $q > 1$ according to \aref{thm:KW-OFS} and $\pi_{3,0}\f_{1}(\KW) \to \ker(d^1_{3,1,0})$ is surjective,
a diagram chase implies the map
\[
\f_{q+1}\pi_{2,0}\f_{q}(\KW) \to \pi_{2,0}\f_{q-1}(\KW)
\]
is injective for all $q$.
This implies \eqref{equation:toshow} when $p = 2$.

When $p = 3$ and $q>2$ we show in \aref{lem:lemma69} that there is a naturally induced surjection
\[
\pi_{4,0}\f_{q-1}(\KW) \to \pi_{4,0}s_{q-1}(\KW) \to E^{\infty}_{4,q-1,0}(\KW).
\]
This implies the map $\f_{q+1}\pi_{3,0}\f_{q}(\KW) \to \pi_{3,0}\f_{q-1}(\KW)$ is injective (in fact, it is trivial).
\end{proof}

To proceed we recall some facts about the Clifford invariant for rings of $\mathcal S$-integers from \cite{Hahn}.
Any quadratic space $(P, \phi)$ over $\OFS$ has an associated Clifford algebra $C(P, \phi)$. 
This is a graded Azumaya algebra; in particular, it defines an element of the Brauer-Wall group $BW(\OFS)$.
The Clifford invariant is the induced group homomorphism $\Cliff\colon W(\OFS)\to BW(\OFS)$, 
see \cite[p.~206]{Hahn}.
Over the field $F$,
the restriction of the Clifford invariant to the square of the fundamental ideal $I^2(F)$ factors through the $2$-torsion subgroup ${}_{2}\Br(F)$ of the Brauer group, 
see \cite[p.~207, Theorem 13.14]{Hahn}, \cite[Lemma 4.4]{MilnorHusemoller}.
From these considerations we obtain the commutative diagram
\[
\begin{tikzcd}
\Cliff^{-1}({}_{2}\Br(\OFS))\cap I^2(F) \ar[r, hook]\ar[d, hook]\ar[rrr, bend left=10] &W(\OFS)\ar[d, hook] \ar[r, "\Cliff"] & BW(\OFS)\ar[d] & \ar[l, hook]{}_{2}\Br(\OFS)\ar[d, hook] \\
I^2(F)\ar[r, hook]\ar[rrr, bend right=10] & W(F) \ar[r, "\Cliff"] & BW(F) &\ar[l, hook] {}_{2}\Br(F).
\end{tikzcd}
\]

\begin{lemma}
\label{lem:lemma69}
There is a naturally induced isomorphism
\[
\f_{q}\pi_{0,0}(\KW)/\f_{q+1}\pi_{0,0}(\KW) \cong E^{\infty}_{0,q,0}(\KW).
\]
\end{lemma}
\begin{proof}
This is shown over $F$ in \cite[Lemma 6.9]{slices}.
By \aref{lem:lemma65} there are maps 
$$
I^{q}(\OFS)/I^{q+1}(\OFS) \to \f_{q}\pi_{0,0}(\KW)/\f_{q+1}\pi_{0,0}(\KW)(\OFS),
$$
and a commutative diagram
\begin{equation}
\label{equation:E-inf-iso}
\begin{tikzcd}
I^{q}(\OFS)/I^{q+1}(\OFS) \ar[r]\ar[d] & \f_{q}\pi_{0,0}(\KW)/\f_{q+1}\pi_{0,0}(\KW)(\OFS) \ar[hook, r]\ar[d, hook] & E^{\infty}_{0,q,0}(\KW)(\OFS) \ar[d] \\
I^{q}(F)/I^{q+1}(F) \ar[r, "\cong"] & \f_{q}\pi_{0,0}(\KW)/\f_{q+1}\pi_{0,0}(\KW)(F) \ar[r, "\cong"] & h^{q,q}(F).
\end{tikzcd}
\end{equation}
The injections in \eqref{equation:E-inf-iso} follow from \cite[Lemma 5.6]{Boardman} and the isomorphisms follow from \cite[Lemma 6.9]{slices}.
In \eqref{equation:E-inf-iso} we want to show there is a naturally induced surjection
\[
\f_{q}\pi_{0,0}(\KW)/\f_{q+1}\pi_{0,0}(\KW)\to E^{\infty}_{0,q,0}(\KW).
\] 

When $q > 2$, \aref{thm:KW-OFS} and \aref{thm:hOFS} imply $E^{\infty}_{0,q,0}(\KW)(\OFS) \cong h^{q,q}(F)$.
Moreover,  
$I^{q}(F)$ is generated by forms defined over $\OFS$.
Indeed, 
by \cite[Corollary IV.4.5]{MilnorHusemoller} the image of $W(\OFS) \cap I^2(F)$ by the signature map is $4\Z^{r_{1}}$.
It follows that $\sigma(I(\OFS)) \supset 4\Z^{r_{1}}$ and $I^3(\OFS) = I^3(F) \cong 8\Z^{r_{1}}$.
Hence the leftmost vertical map in \eqref{equation:E-inf-iso} is surjective.
This implies $\f_{q}\pi_{0,0}(\KW)/\f_{q+1}\pi_{0,0}(\KW) \to h^{q,q}$ is surjective.

When $q = 1$ we note the map $\pi_{4,0}\s_{1}(\KW) \to \pi_{3,0}\f_{2}(\KW)$ is trivial. 
It follows that $\pi_{4,0}\f_{1}(\KW) \to \pi_{4,0}\s_{1}(\KW) \cong E^{\infty}_{0,1,0}(\KW)$ is surjective.

In the more complicated case $q = 2$ we first show there is an injection 
$$
\Cliff^{-1}({}_{2}\Br(\OFS))\cap I^2(F) \cong \f_{2}\pi_{0,0}(\KW) \hookrightarrow \pi_{0,0}(\KW) \cong W(\OFS).
$$
For this we consider the commutative diagram with exact rows and columns obtained by localization and the slice filtration
\begin{equation}
\label{equation:diag-cliff}
\begin{tikzcd}
h^{0,1}(\OFS) = \pi_{1,0}\s_{1}(\KW) \ar[r, "0"] \ar[d, "\cong"] & \pi_{0,0}\f_{2}(\KW) \ar[r, hook]\ar[d, hook] & \pi_{0,0}\f_{1}(\KW) \ar[d, hook] \\
h^{0,1}(F) = \pi_{1,0}j_*j^{\ast}\s_{1}(\KW) \ar[r, "0"] \ar[d, "0"] & \pi_{0,0}j_*j^{\ast}\f_{2}(\KW) \ar[r, hook]\ar[d] & \pi_{0,0}j_*j^{\ast}\f_{1}(\KW)\ar[d] \\
\oplus_{x\not\in\mathcal S}\pi_{1,0}\s_{1}(i_!i^{!}\s_{1}(\KW)) \ar[r] \ar[d] &  \oplus_{x\not\in\mathcal S}\pi_{0,0}i_!i^{!}\f_{2}(\KW) \ar[r]\ar[d] & \oplus_{x\not\in\mathcal S}\pi_{0,0}i_!i^{!}\f_{1}(\KW)\\
\pi_{2,0}\s_{1}(\KW) \ar[r, hook] & \pi_{1,0}\f_{2}(\KW).
\end{tikzcd}
\end{equation}
An inspection of the slice differentials as in \cite[Theorem 5.3]{slices} yields the indicated injective and trivial maps.
Suppose $x \in \Cliff^{-1}({}_{2}\Br(\OFS)) \cap I^2(F) \subseteq \pi_{0,0}j_*j^{\ast}\f_{2}(\KW) \cong I^2(F) = \Cliff^{-1}({}_{2}\Br(F))\cap I^2(F)$ maps to 
$y \in \oplus_{x\not\in\mathcal S}\pi_{0,0}i_!i^{!}\f_{2}(\KW)$.
(The isomorphism is \cite[Corollary 6.16]{slices}, and the equality holds because the Clifford-invariant surjects onto ${}_{2}\Br(F)$ by \cite[Theorem 14.6]{Hahn}). 
Then $y$ maps trivially to $\oplus_{x\not\in\mathcal S}\pi_{0,0}i_!i^{!}\f_{0}(\KW)$ and $\pi_{1,0}\f_{2}(\KW)$, 
so it is the image of some element $z \in \oplus_{x\not\in\mathcal S}\pi_{1,0}i_!i^{!}\s_{1}(\KW)$.
Suppose $z$ maps to $w$ under the injective map to $\pi_{2,0}\s_{1}(\KW)$.
By commutativity of \eqref{equation:diag-cliff} $w$ goes to zero in $\pi_{1,0}\f_{2}(\KW)$.
Hence 
we get $y=0$, 
which implies $x$ is in the image of $\pi_{0,0}\f_{2}(\KW)$.

The commutative diagram of short exact sequences (see \aref{lem:pic-tau} and its proof)
\[
\begin{tikzcd}
0 \ar[r] & h^{2,1}(\OFS) \ar[r, "\tau"]\ar[d] & h^{2,2}(\OFS)\ar[r]\ar[d] & {}_{2} H_{\et}^{2}(\OFS;\Gm) \ar[r]\ar[d] & 0 \\
0 \ar[r]& 0 \ar[r, "\tau"] & h^{2,2}(F)\ar[r] & {}_{2} H_{\et}^{2}(F;\Gm) \ar[r] & 0
\end{tikzcd}
\]
yields a naturally induced injection $E^{\infty}_{0,2,0}(\KW) \cong h^{2,2}(\OFS)/\tau \to h^{2,2}(F)$.
Here we use the injection of Brauer groups $H_{\et}^{2}(\OFS;\Gm) \to H_{\et}^{2}(F;\Gm)$ from \cite[p.~107]{MR559531}.
Hence we may identify 
$$
\f_{2}\pi_{0,0}\KW \to \f_{2}\pi_{0,0}\KW/\f_{3}\pi_{0,0}\KW \to E^{\infty}_{0,2,0}(\KW)
$$ 
with the upper horizontal map in the commutative diagram
\[
\begin{tikzcd}
\f_{2}\pi_{0,0}\KW \cong \Cliff^{-1}({}_{2}\Br(\OFS))\cap I^2(F)  \ar[r, "\Cliff"]\ar[d, hook] & {}_{2} H_{\et}^{2}(\OFS;\Gm)  \ar[d, hook] \cong h^{2,2}(\OFS)/\tau\\
\f_{2}\pi_{0,0}j_*j^{\ast}(\KW) \cong \Cliff^{-1}({}_{2}\Br(F))\cap I^2(F) \ar[r, "\Cliff"] & {}_{2} H_{\et}^{2}(F;\Gm) \cong h^{2,2}(F).
\end{tikzcd}
\]
By \cite[Theorem 14.6]{Hahn} the horizontal maps are surjective,
and we conclude the naturally induced map 
$$
\f_{2}\pi_{0,0}\KW/\f_{3}\pi_{0,0}\KW \to E^{\infty}_{0,2,0}(\KW)
$$ 
is surjective.
\end{proof}

From the proof of \aref{lem:lemma69} we obtain the following form of Milnor's conjecture for quadratic forms \cite{Milnor} over the Dedekind domain $\OFS$.
\begin{theorem}
\label{theorem:milnor-ofs}
Set $I_{q}(\OFS) = \f_{q}\pi_{0,0}(\KW)(\OFS)$ for all integers $q$.
Then $I_{q}(\OFS) \cong W(\OFS)$ for $q\leq 0$, $I_{1}(\OFS)\cong\ker(\rk_{2})$, $I_{2}(\OFS) \cong \Cliff^{-1}({}_{2}\Br(\OFS))\cap I^2(F)$, and $I_{q}(\OFS) = I_{1}^q(\OFS)$ for $q>2$.
For $\tau\in h^{0,1}$ the class of $-1$ and every $q\geq 0$ there is an isomorphism
\begin{equation}
\label{equation:MCOFS}
I_{q}(\OFS)/I_{q+1}(\OFS) \cong h^{q,q}(\OFS)/\tau.
\end{equation}
\end{theorem}
\begin{remark}
Note that $h^{q,q}/\tau = h^{q,q}$ for $q \neq 2$.
Via the Clifford invariant the proof of \aref{lem:lemma69} identifies $I_{2}(\OFS)$ with the kernel of the Arf invariant \cite[p.~206]{Hahn}.
\end{remark}

By now familiar arguments this allows us to conclude the following generalization of \aref{theorem:convergence1}.
\begin{theorem}
\label{theorem:convergence2}
Over $\OFS$ there are isomorphisms for $n\geq 1$, $p\equiv 0\bmod 4$, and $w\in\Z$
\begin{align*}
\f_{q} \pi_{p+w,w}(\KW/2^{n}) &\cong h^{2,1}(\OFS)\bullet\im (I_{q-w}(\OFS) \to W(\OFS)/2^{n}), q \leq 1 \\
\f_{q} \pi_{p+w,w}(\KW/2^{n}) &\cong \im (I_{q-w}(\OFS) \to W(\OFS)/2^{n}), q > 1 \\
\f_{q} \pi_{p+w+1,w}(\KW/2^{n}) &\cong {}_{2^{n}} I_{q-w}(\OFS) \\
\f_{q} \pi_{p+w+2,w}(\KW/2^{n}) &= 0 \\
\f_{0} \pi_{p+w+3,w}(\KW/2^{n}) &\cong \f_{1} \pi_{p+w+3,w}(\KW/2^{n}) \cong h^{2,1}(\OFS) \\
\f_{q} \pi_{p+w+3,w}(\KW/2^{n}) & = 0, q > 1.  \\
\end{align*}
\end{theorem}
Here we write $A\bullet B$ for an abelian group extension of $B$ by $A$,
i.e., 
there is an exact sequence
$$
0
\to
A
\to
A\bullet B
\to
B
\to 
0.
$$

\begin{theorem}
\label{theorem:KWmod2nconvergence2}
The filtrations of $W(\OFS)/2^{n}$ by $I_{q}(\OFS)/2^{n}$ and of ${}_{2^{n}} W(\OFS)$ by ${}_{2^{n}} I_{q}(\OFS)$ are exhaustive, Hausdorff, and complete.
Hence the slice spectral sequence for $\KW/2^{n}$ over $\OFS$ is strongly convergent.
\end{theorem}

\subsection{Convergence for hermitian $K$-theory}
In this section we combine the convergence results for $\KGL$ and $\KW$ to conclude conditional convergence of the slice spectral sequences for $\KQ$.
Recall that $\KW$ is obtained by inverting the Hopf map $\eta$ on $\KQ$ while $\KGL$ and $\KQ$ are related via the Wood cofiber sequence \cite[Theorem 3.4]{slices}
\begin{equation}
\label{equation:wood}
\Sigma^{1,1} \KQ \xrightarrow{\eta} \KQ \to \KGL.
\end{equation}

Recall from \cite[Lemma 2.1]{slices} the canonical isomorphism in $\SH$
\begin{equation}
\label{equation:fqsigma}
\f_{q+k}\Sigma^{0,k}(\E) 
\simeq
\Sigma^{0,k}\f_{q}(\E).
\end{equation}

\begin{defn}
\label{defn:extendedslice}
For $i \geq 0$ define $\s^{q+i}_{q}(\E)$ by the cofiber sequence
\[
\f_{q+i}(\E) \to \f_{q}(\E) \to \s^{q+i}_{q}(\E).
\]
\end{defn}

\begin{lemma}
For every motivic spectrum $\E$ there are cofiber sequences
\begin{equation}
\label{equation:extended-slices-cof}
\s^{q + i}_{q+1}(\E) \to \s^{q+i}_{q}(\E) \to \s_{q}(\E).
\end{equation}
\end{lemma}
\begin{proof}
This follows from \aref{defn:extendedslice} and the octahedral axiom \cite[Definition 7.1.4]{Hovey}.
\end{proof}

\begin{lemma}
\label{lem:KGLvanishing}
For integers $p, w, q, i\in\Z$,  $l \geq 0$, and $k \gg 0$ we have
\begin{equation}
\label{equation:0vanishing}
\pi_{p+k,w}\s^{q+i}_{q}(\KGL) = 0, 
\end{equation}
\begin{equation}
\label{equation:1vanishing}
\pi_{p-l,w}\f_{q+k}(\KGL) = 0.
\end{equation}
\end{lemma}
\begin{proof}
We note that $\pi_{n+k,w}\s_{q}(\KGL) = H^{2q-n-k, q - w - k} = 0$ for $k \gg 0$.
Induction on $i$ using \eqref{equation:extended-slices-cof} proves \eqref{equation:0vanishing} for all $i$.
Moreover, 
\eqref{equation:1vanishing} follows because the slice spectral sequence for $\f_{q+k}(\KGL)$ is strongly convergent (see \aref{theorem:KGL2ncc}).
\end{proof}

\begin{proposition}
\label{proposition:KQ-conv}
The slice spectral sequences for $\KQ$ and $\KQ/2^{n}$ are convergent, 
i.e.,
\[
\bigcap_{i\geq0} \f_{q+i}\pi_{\ast,\ast}\f_{q}(\KQ) 
= 
\bigcap_{i\geq0} \f_{q+i}\pi_{\ast,\ast}\f_{q}(\KQ/2^{n}) 
= 
0.
\]
\end{proposition}
\begin{proof}
By applying $\pi_{p,w}$ to the diagram of cofiber sequences
\[
\begin{tikzcd}
&\vdots\ar[d] & \vdots\ar[d] & \Sigma^{-2,-1}\s^{q+i+1}_{q+1}(\KGL) \ar[d] \\
\dots\ar[r] & \f_{q+i}(\KQ) \ar[r]\ar[d] &  \Sigma^{-1,-1}\f_{q+i+1}(\KQ) \ar[r]\ar[d] & \Sigma^{-1,-1}\f_{q+i+1}(\KGL) \ar[d] \ar[r] & \dots \\
\dots\ar[r] & \f_{q}(\KQ) \ar[r]\ar[d] &  \Sigma^{-1,-1}\f_{q+1}(\KQ) \ar[r]\ar[d] & \Sigma^{-1,-1}\f_{q+1}(\KGL) \ar[d]  \ar[r] & \dots \\
\dots\ar[r] & \s^{q+i}_{q}(\KQ) \ar[r]\ar[d] &  \Sigma^{-1,-1}\s^{q+i+1}_{q+1}(\KQ) \ar[r]\ar[d] & \Sigma^{-1,-1}\s^{q+i+1}_{q+1}(\KGL)  \ar[r]\ar[d] & \dots \\
&\vdots & \vdots & \vdots
\end{tikzcd}
\]
we obtain a double complex $A^{\ast,\ast}$ with $A^{0,0} = \pi_{p,w}\f_{q}(\KQ)$ and exact rows and columns.
\aref{lem:double-complex} applies to $A^{\ast,\ast}$ since $A^{2+3k, 2+3k} = \pi_{p+2+k,w+1}\s^{q+i+1}_{q+1}(\KGL)= 0$ for $k \gg 0$ according to \aref{lem:KGLvanishing}.
Thus for $i \gg 0$, 
so that $\pi_{p+2,w+1}\f_{q+i}(\KGL) = 0$,
the composition
\[
\f_{q+i}\pi_{p,w}\f_{q}(\KQ) \to \pi_{p,w}\f_{q}(\KQ) \to \pi_{p+1,w+1}\f_{q+1}(\KQ)
\]
is injective.
More generally,
\[
\f_{q+l+i}\pi_{p+l,w+l}\f_{q}(\KQ) \to \pi_{p+l,w+l}\f_{q+l}(\KQ) \to \pi_{p+l+1,w+l+1}\f_{q+l+1}(\KQ)
\]
is injective for $l \geq 0$ because $\pi_{p+l+2,w+l+1}\f_{q+l+i}(\KGL) = 0$ for $i \gg 0$, 
see \aref{lem:KGLvanishing}.

Since $\bigcap_{i\geq0} \f_{q+i}\pi_{\ast,\ast}\f_{q}(\KW) = 0$,
any $x \in \bigcap_{i\geq0} \f_{q+i}\pi_{p,w}\f_{q}(\KQ)$ maps trivially under the composition
\[
\pi_{p,w}\f_{q}(\KQ) \to
\pi_{p+1,w+1}\f_{q+1}(\KQ) \to
\pi_{p+2,w+2}\f_{q+2}(\KQ) \to \cdots
\to \pi_{p,w}\f_{q}(\KW).
\]
But since $\f_{q+i}\pi_{p,w}\f_{q}(\KQ)$ maps injectively under this composition for $i \gg 0$,
this implies $x=0$.

A verbatim argument applies to $\KQ/2^{n}$.
\end{proof}

\begin{lemma}
\label{lem:double-complex}
Suppose $A^{\ast,\ast}$ is a double complex with exact rows and columns such that $A^{k,k} = 0$ for some $k>0$,
and with differentials $d_h^{p,q} : A^{p,q} \to A^{p+1,q}$, $d_v^{p,q}: A^{p,q} \to A^{p,q-1}$.
Then for any $x \in A^{0,1}$ in the kernel of the composite map $A^{0,1} \to A^{0,0} \to A^{1,0}$,
there exists an element $y \in A^{-1,1}$ such that $x$ and $y$ have the same image in $A^{0,0}$.
\end{lemma}
\begin{proof}
We set $x_{1} = d^{0,1}_h(x)$ and inductively $x_{k+1} = d^{k,k+1}_h(x_{k}')$.
Here $x_{k}' \in A^{k,k+1}$ is a lift of $x_{k}$,
i.e.,
$d^{k,k+1}_v(x_{k}') = x_{k}$.
Note that $x_k \in A^{k,k}$,
$x_k' \in A^{k,k+1}$,
$d^{k,k}_v(x_k) = 0$,
and $x_{n} = 0$ for some $n\gg 0$.
Next we construct elements $y_k \in A^{k,k+1}$ such that $d^{k,k+1}_v(y_{k}) = d^{k,k+1}_v(x_{k}') = x_{k}$ and $d^{k,k+1}_h(y_{k}) = 0$.
First we set $y_{n-1} = x_{n-1}'$ for $n$ as above.
Since $d^{k+1,k+2}_h(y_{k+1}) = 0$,
there exists an $x_{k+1}''$ such that $d^{k,k+2}_h(x_{k+1}'') = y_{k+1}$.
To conclude we set $y_{k} = x_{k}' - d^{k,k+2}_v(x_{k+1}'')$ and $y = x_{-1}''$.
\end{proof}

\begin{proposition}
\label{proposition:KQ-strong-conv}
Let $(\E,\E')$ be shorthand for $(\KQ,\KW)$ or $(\KQ/2^{n},\KW/2^{n})$.
Then we have
\begin{equation}
\label{equation:KQlimKW}
\lim_{q} \pi_{p,w}\f_{q}(\E) \cong \lim_{q} \pi_{p,w}\f_{q}(\E'),
{\lim_{q}}^1 \pi_{p,w}\f_{q}(\E) \cong {\lim_{q}}^1 \pi_{p,w}\f_{q}(\E').
\end{equation}
It follows that $\E\to\SC(\E)$ is a $\pi_{\ast,\ast}$-isomorphism if and only if $\E'\to\SC(\E')$ is a $\pi_{\ast,\ast}$-isomorphism.
\end{proposition}
\begin{proof}
The Wood cofiber sequence \eqref{equation:wood} induces the commutative diagram
\begin{equation}
\label{equation:KQtower}
\begin{tikzcd}[column sep=10pt]
\vdots\ar[d] & \ar[d] \vdots & \ar[d] \vdots & \vdots \ar[d] & \vdots & \vdots \ar[d] \\
\pi_{p,w}\f_{q+l}(\KQ) \ar[r]\ar[d] & \pi_{p+1,w+1}\f_{q+l+1}(\KQ) \ar[r]\ar[d] & \dots \ar[r]\ar[d] & \pi_{p+k,w+k}\f_{q+k+l}(\KQ) \ar[r]\ar[d] & \dots \ar[r] & \pi_{p,w}\f_{q+l}(\KW)\ar[d] \\
\vdots\ar[d] & \ar[d] \vdots & \ar[d] \vdots & \vdots \ar[d] & \vdots & \vdots \ar[d] \\
\pi_{p,w}\f_{q}(\KQ) \ar[r] & \pi_{p+1,w+1}\f_{q+1}(\KQ) \ar[r] & \dots \ar[r] & \pi_{p+k,w+k}\f_{q+k}(\KQ) \ar[r] & \dots \ar[r] & \pi_{p,w}\f_{q}(\KW), \\
\end{tikzcd}
\end{equation}
where each horizontal map is part of a long exact sequence
\begin{align*}
\dots
\to \pi_{p+k+2,w+k+1}\f_{q+k+1}(\KGL)
\to \pi_{p+k,w+k}\f_{q+k}(\KQ)
&\to 
\pi_{p+k+1,w+k+1}\f_{q+k+1}(\KQ) \\
&\to \pi_{p+k+1,w+k+1}\f_{q+k+1}(\KGL)
\to \cdots.
\end{align*}
By Bott periodicity for $\KGL$ \cite[Theorem 6.8]{Voevodsky:icm} and \eqref{equation:fqsigma} there are isomorphisms
\begin{align*}
\pi_{p+k+1,w+k+1}\f_{q+k+1}(\KGL)
&=
\pi_{p,w}\Sigma^{-(k+1),-(k+1)}\f_{q+k+1}(\KGL) \\
&\cong
\pi_{p,w}\f_{q}(\Sigma^{-(k+1),-(k+1)}\KGL) \\
&\cong
\pi_{p,w}\f_{q}(\Sigma^{k+1,0}\KGL) \\
&\cong
\pi_{p-k-1,w}\f_{q}(\KGL).
\end{align*}
According to \aref{lem:KGLvanishing} we have $\pi_{p-l,w}\f_{q}(\KGL) = 0$ for all $l \geq 0$ and $q \gg 0$.
Thus the horizontal maps in \eqref{equation:KQtower} are isomorphisms,
and the inverse systems $\{\pi_{p,w}\f_{q}(\KQ) \}_{q}$ and $\{\pi_{p,w}\f_{q}(\KW) \}_{q}$ are levelwise isomorphic for $q\gg 0$.
This proves the isomorphisms in \eqref{equation:KQlimKW}.

The final claim follows from the Milnor exact sequence
\[
0 \to {\lim_{q}}^1 \pi_{p+1,w}\f_{q}(\E)
\to \pi_{p,w}\holim_{q} \f_{q}(\E)
\to \lim_{q} \pi_{p,w}\f_{q}(\E)
\to 0.
\]

A verbatim argument applies to $(\KQ/2^{n},\KW/2^{n})$.
\end{proof}

We are ready to state our main convergence result for fields.
\begin{theorem}
\label{theorem:best-conv}
Suppose the filtrations $\{\im(I^{q}(F) \to W(F)/2^{n})\}_{q}$ of $W(F)/2^{n}$ and $\{{}_{2^{n}} I^{q}(F)\}_{q}$ of ${}_{2^{n}} W(F)$ are exhaustive, 
Hausdorff, 
and complete (e.g., if $\vcd(F) < \infty$ or $F^{\times}/2$ is finite).
Then the slice spectral sequence for $\KQ/2^{n}$, $n\geq 1$, is conditionally convergent with abutment $\pi_{\ast,\ast}(\KQ/2^{n})$.
\end{theorem}
\begin{proof}
Follows from \aref{lem:sc-convergence}, \aref{theorem:convergence1}, \aref{theorem:KWmod2nconvergence}, and \aref{proposition:KQ-strong-conv}.
\end{proof}

Similarly we obtain a generalization of \aref{theorem:best-conv} to rings of $\mathcal{S}$-integers based on \aref{theorem:convergence2} and \aref{theorem:KWmod2nconvergence2}.
\begin{theorem}
\label{theorem:best-conv2}
Over $\OFS$ the slice spectral sequence for $\KQ/2^{n}$, $n\geq 1$, is conditionally convergent with abutment $\pi_{\ast,\ast}(\KQ/2^{n})$.
\end{theorem}

\subsection{Multiplicative structure and pairings of slice spectral sequences}
Let us begin by constructing a pairing of slice spectral sequences based on the motivic version
\begin{equation}
\label{equation:okapairing1}
\One/4 \wedge \One/2 \to \One/2
\end{equation}
of Oka's module action of the mod 4 by the mod 2 Moore spectrum \cite[\S6]{Oka}.
If $\E$ is a motivic ring spectrum, 
i.e., 
a monoid in $\SH$, 
then \eqref{equation:okapairing1} induces the more general pairing
\begin{equation}
\label{equation:okapairing2}
\E/4 \wedge \E/2 \to \E/2.
\end{equation}
The slice filtration of $\E$ gives rise to an Eilenberg-MacLane system in the sense of \cite{MR0077480} and hence to an exact couple.
By \cite[Proposition 2.24]{April1}, \eqref{equation:okapairing2} induces a pairing of slice spectral sequences
\begin{equation}
\label{equation:R-pairing}
E^r_{p,q,w}(\E/4) \otimes E^r_{p',q',w'}(\E/2) \to E^r_{p+p',q+q',w+w'}(\E/2),
\end{equation}
satisfying the Leibniz rule $d^{r}(a\cdot b)=d^{r}(a)\cdot b + (-1)^{p}a\cdot d^{r}(b)$ for $a\in E^r_{p,q,w}(\E/4)$, $b\in E^r_{p',q',w'}(\E/2)$.
Here $\E/4$ is a motivic ring spectrum, 
and the groups $E^r_{p,q}(\E/4)$ form the $E^r$-page of an algebra spectral sequence whose differentials satisfy the Leibniz rule.
On the level of slices there is a multiplication map
\[
\s_{m}(\One) \wedge \s_{q}(\E) \to \s_{q+m}(\E).
\]
Recall that all the positive slices of the motivic sphere spectrum contain $\MZ/2$ as a direct summand up to $\Gm$-suspensions \cite[Corollary 2.13]{April1}. 
More precisely,
we have
\[
\s_m(\One) 
\simeq 
\Sigma^{m,m}\MZ/2 \vee \cdots.
\]
Any differential entering or exiting the direct summand $h^{m,m} \hookrightarrow \pi_{0,0}\s_m(\One)$ is trivial \cite[Theorem 4.8]{April1}.
By the Leibniz rule with respect to the differentials \cite[Proposition 2.24]{April1},
there is an induced pairing
\begin{equation}
\label{equation:h-mult}
h^{m,m} \tensor E^{r}_{p,q,w}(\E) \to E^{r}_{p,q+m,w}(\E) .
\end{equation}
Under this pairing we have $d^r(x \cdot y) = x \cdot d^r(y)$ for all $x \in h^{m,m}$, $y \in E^{r}_{p,q,w}(\E) $.

\begin{lemma}
\label{lem:<-1>}
Over a field $F$ of characteristic $\Char(F)\neq 2$,
the canonical map $\pi_{0,0}\f_{1}(\One) \to \pi_{0,0}\s_{1}(\One)$ sends $\langle -1\rangle - \langle 1 \rangle\in GW(F)$ to $\rho\in h^{1,1}$.
\end{lemma}
\begin{proof}
In the proof we will make use of Milnor-Witt $K$-theory $\KMW_{\ast}(F)$ of $F$ (see \cite{morelmotivicpi0}).
Consider the commutative diagram
\[
\begin{tikzcd}
\Sigma^{1,1}\f_{0}(\One) \ar[r, "\simeq"]\ar[d] & \f_{1}(\Sigma^{1,1}\One) \ar[r, "\f_{1}(\eta)"]\ar[d] & \f_{1}(\One)\ar[d] \\
\Sigma^{1,1}\s_{0}(\One) \ar[r, "\simeq"]\ar[d] & \s_{1}(\Sigma^{1,1}\One) \ar[r, "\s_{1}(\eta)"]& \s_{1}(\One)\ar[d] \\
\Sigma^{1,1}\MZ \ar[rr, "\pr^{\infty}_{2}"] && \Sigma^{1,1}\MZ/2.
\end{tikzcd}
\]
See \cite[Lemma 2.1]{slices} for the upper leftmost square and \cite[Lemma 2.32]{April1} for the bottom square.
Applying $\pi_{0,0}$ we get the commutative diagram
\[
\begin{tikzcd}
\pi_{-1,-1}(\One) \ar[r]\ar[d] & \f_{1}\pi_{1,1}(\One)\ar[d] \\
H^{1,1}(F;\Z) \ar[r] & h^{1,1}.
\end{tikzcd}
\]
The left vertical map identifies with the quotient map $\KMW_{1}(F) \to \KMW_{1}(F) /\eta$ under the isomorphism $\pi_{-1,-1}(\One) \cong \KMW_{1}(F)$.
Recall $\eta \in \KMW_{-1}(F)$ corresponds to the Hopf map in $\pi_{1,1}(\One)$ and $\f_{n}\pi_{0,0}(\One) \cong I^{n}$ over perfect fields \cite[Theorem 1]{Levine:GW}.
(For a general field use base change and \cite[Lemma 2.5]{slices}).
That is, 
the isomorphism $\KMW_{0}(F) \to GW(F)$ given by $1+\eta[u] \mapsto \langle u \rangle$ induces a commutative diagram with surjective maps
\begin{equation}
\label{equation:KMWetadiagram}
\begin{tikzcd}
\KMW_{1} (F) \ar[r]\ar[d, "\cdot\eta"] & \KMW_{1}(F) /\eta \ar[d]\ar[rd] \\
I \ar[r]& I/I^{2} \ar[r] & \KMW_{1}(F) /(\eta, 2) \cong h^{1,1}.
\end{tikzcd}
\end{equation}
A diagram chase in \eqref{equation:KMWetadiagram} shows $\langle -1\rangle - \langle 1 \rangle$ maps to $\rho$.
\end{proof}

To state our next result we set $\Lambda: = \{x \not\in \mathcal S\}^{-1}\Z$ --- the localization of $\Z$ inverting every rational prime not in $\mathcal S\cap \Z$ ---
and let $\E_{\Lambda}$ (resp.~$A_{\Lambda}$) be the corresponding localization of the motivic spectrum $\E$ (resp.~abelian group $A$), 
see also \cite[Definition 2.1, Remark 2.2]{April1}.
\begin{lemma}
\label{lem:<-1>2}
Over $\OFS$ the canonical map $\pi_{0,0}\f_{1}(\KQ_{\Lambda}) \to \pi_{0,0}\s_{1}(\KQ_{\Lambda})$ sends $\langle -1\rangle - \langle 1 \rangle\in GW(\OFS)_{\Lambda}$ to $\rho\in h^{1,1}$.
\end{lemma}
\begin{proof}
Since $\pi_{0,0}(\One)(F) \to \pi_{0,0}(\KQ)(F)$ is an isomorphism the corresponding statement holds over $F$ by \aref{lem:<-1>}.
Comparing $\OFS$ with $F$ we have the commutative diagram
\[
\begin{tikzcd}
\pi_{0,0}\f_{1}(\KQ_{\Lambda})(\OFS) \ar[r]\ar[d, hook] & \pi_{0,0}\f_{1}(\KQ_{\Lambda})(F) \ar[d, hook] \\
\pi_{0,0}(\KQ_{\Lambda})(\OFS) \ar[r]\ar[d, "\cong"] & \pi_{0,0}(\KQ_{\Lambda})(F) \ar[d, "\cong"] \\
GW(\OFS)_{\Lambda} \ar[r, hook] & GW(F)_{\Lambda}.
\end{tikzcd}
\]
The maps $GW(\OFS) \to GW(F)$ and $h^{1,1}(\OFS) \to h^{1,1}(F)$ are injective (see \cite[Corollary IV.3.3]{MilnorHusemoller} and \aref{section:mcatSa}).
Hence $\langle -1\rangle - \langle 1 \rangle\in GW(\OFS)$ maps to $\rho\in h^{1,1}$, since this holds over $F$.
\end{proof}

\begin{remark}
If absolute purity holds for the sphere $\One$ --- see \aref{theorem:KQKWpurity} for $\KQ$ --- then localization implies \aref{lem:<-1>} holds over $\OFS$.
By comparing the slice spectral sequences for $\One$ and $\KQ$ one obtains a version of \aref{lem:<-1>} for $\eta$-completions.
\end{remark}

\section{Algebraic $K$-theory $\KGL$}
\label{section:akt}
Throughout we work over a base field $F$ of $\Char(F)\neq 2$ or the ring of $\mathcal{S}$-integers in a number field. 
Recall from \cite[\S5]{slices} the slice calculation
\begin{equation}
\label{equation:slicesKGL2n}
\s_{q}(\KGL/2^{n})
\simeq
\Sigma^{2q,q}\MZ/2^{n}.
\end{equation}
Since $\mathcal{S}$ contains all dyadic primes, 
the same calculation holds over $\OO_{F,\mathcal{S}}$ by \cite[Theorem 2.19]{April1}.
Recall the slice spectral sequence for $\KGL/2^{n}$ converges conditionally over fields $F$ by \cite[Lemma 3.11]{April2} and over rings of $\mathcal{S}$-integers in number fields by the 
Wood cofiber sequence \eqref{equation:wood} and \aref{theorem:best-conv2}.
We give an alternate proof in \aref{theorem:KGL2ncc}.

By the proof of \cite[Lemma 5.1]{slices} the slice $\dd^{1}$-differential for $\KGL/2$ is the first Milnor operation
\begin{equation}
\label{equation:KGL-d1}
\QQ_{1} = \Sq^{3} + \Sq^{2}\Sq^{1}.
\end{equation}
In weight $w=0$ we obtain $E^{1}_{p,q,0}(\KGL/2)=h^{2q-p,q}$, 
which vanishes if $2q<p$ or $q>p$ with the possible exception of the mod $2$ Picard group $h^{2,1}\cong\Pic(\OO_{F,\mathcal{S}})/2$.

Over a general base scheme there is a canonical orientation map $\Phi\colon\MGL \to \KGL$ of motivic ring spectra. 
By passing to effective covers we obtain the factorization 
$$
\MGL \xrightarrow{\Psi} \f_{0}(\KGL) \to \KGL.
$$
In the Lazard ring $\mathbb{L}=\Z[x_{1},x_{2},\dotsc]$ the generator $x_{n}$ can be viewed as a map $\Sigma^{2n,n}\One \to \MGL$ defined over the integers. 
By complex realization we have $\Phi(x_n)=0$ and hence $\Psi(x_n)=0$ for all $n\geq 2$.
Hence $\Psi$ defines a map $\Theta\colon \MGL/(x_{2},x_3,\dotsc)\MGL \to \f_{0}(\KGL)$ of quotients, 
following \cite[\S5]{Spitzweck:slices}. 

\begin{proposition}
\label{prop:landweber-kglmod2}
Let $S$ be a Dedekind domain containing $\tfrac{1}{2}$. 
Then $\Theta/2$ induces an equivalence 
$$
\MGL/(2,x_{2},x_{3},\dotsc)
\xrightarrow{\simeq} 
\f_{0}(\KGL/2).
$$
\end{proposition}
\begin{proof}
This follows from \cite[Proposition 5.4]{Spitzweck:slices} in combination with \cite[Theorem 11.3]{Spitzweck}.
\end{proof}

\begin{corollary}
\label{cor:slicefiltration-kglmod2}
Let $S$ be a Dedekind domain containing $\tfrac{1}{2}$. 
Then 
$$
\f_{q}(\KGL/2)=\Sigma^{2q,q}\f_{0}(\KGL/2)
$$ 
is $q$-connective.
\end{corollary}
\begin{proof}
Follows from \aref{prop:landweber-kglmod2} and Bott periodicity for $\KGL$ \cite[Theorem 6.8]{Voevodsky:icm}.
\end{proof}

\begin{theorem}
\label{theorem:KGL2ncc}
Let $S$ be a Dedekind domain containing $\tfrac{1}{2}$ and let $n\geq 1$.
Over $S$ the motivic spectrum $\KGL/2^{n}$ is slice complete and its slice spectral sequence is conditionally convergent with abutment $\pi_{\ast,\ast}(\KGL/2^{n})$.
\end{theorem}
\begin{proof}
We may assume $n=1$ and conclude using \aref{cor:slicefiltration-kglmod2}.
\end{proof}

\begin{theorem}
\label{thm:KGL2-E2}
In the slice spectral sequence for $\KGL/2$ there are isomorphisms
\[
E^{2}_{p,q,w}(\KGL/2) 
\cong 
\begin{cases}
h^{2,1}                 & (p-w,q)=(0,1) \\ 
h^{2q-p,q-w}/\rho^3        & p-w-q \equiv 0, 1 \bmod 4 \\
\ker (\rho_{2q-p,q-w}^{3}) & p-w-q \equiv 2, 3 \bmod 4.
\end{cases}
\]
\end{theorem}
\begin{proof}
See \aref{fig:KGL2-E1} for the $E^{1}$-page when $w=0$.
By \aref{tbl:Sq-action} the first Milnor operation $\QQ_{1}$ acts on $h^{2q - p,q-w}$ as multiplication by $\rho^3$ times a $\tau$-multiple when $p -w - q \equiv 2, 3 \bmod 4$, 
and trivially otherwise.
Note that $E^{2}_{p,q,w}(\KGL/2)$ is the homology of the complex
\[
E^{1}_{p+1,q-1,w} (\KGL/2)
\xrightarrow{d^{1}_{p+1,q-1,w}}  
E^{1}_{p,q,w} (\KGL/2)
\xrightarrow{d^{1}_{p,q,w}} 
E^{1}_{p-1,q+1,w}(\KGL/2).
\]
Depending on $p-w-q\bmod 4$, 
if $d^{1}_{p,q,w}$ is multiplication by $\rho^3$ times a $\tau$-multiple then $d^{1}_{p+1,q-1,w}$ is trivial, and vice versa.
Thus the $E^{2}$-page takes the claimed form.
\end{proof}

\begin{remark}
Note that $E^{r}_{\ast,\ast,*}(\KGL/2)$ is not an algebra spectral sequence.
This follows since $d^{1}(\tau^{2})\neq 0$, while a Leibniz rule would imply $d^{1}(\tau^{2}) = 2 \tau d^{1}(\tau) = 0$.
\end{remark}

\begin{example}
\label{ex:KGL2-R}
If $F$ is a real closed field, 
then $h^{*,*}=\FF_{2}[\tau,\rho]$ and the proof of \aref{thm:KGL2-E2} implies
\[
E^{\infty}_{p,q,0}(\KGL/2)
\cong
\begin{cases}
\Z/2\{\rho^{2q-p}\tau^{p-q}\} & q - p \equiv 0, 3 \bmod 4, \mathrm{ and } \; 2q - p = 0, 1, 2 \\
0 & \mathrm{otherwise}.
\end{cases}
\]
In the abutment, 
this gives the $8$-periodicity $\Z/2$, $\Z/2$, $\Z/4$, $\Z/2$, $\Z/2$, $0$, $0$, $0$ for the mod 2 $K$-groups of the real numbers, 
see e.g., 
\cite[Theorem 4.9]{MR772065}.
\end{example}

Recall there is an isomorphism $\KGL_{p,q}(F)\cong K_{p-2q}(F)$ \cite[\S6.2]{Voevodsky:icm}.

\begin{example}
\label{example:cd2KGL}
If $\cd_{2}(F)\leq 2$ and $n\geq 1$ there is an isomorphism $K_{2n-1}(F;\Z/2)\cong h^{1,n}$ and a short exact sequence
\begin{equation*}
0
\to
h^{2,n}
\to
K_{2n-2}(F;\Z/2)
\to
h^{0,n}
\to
0.
\end{equation*}
\end{example}

\begin{theorem}
\label{theorem:KGLnumberfields}
The mod 2 algebraic $K$-groups of $\OO_{F,\mathcal{S}}$ are computed up to extensions by the following filtrations of length $l$.
\begin{table}[bht]
\begin{center}
\begin{tabular}
[c]{p{0.4in}|p{0.1in}|p{2.5in}}\hline
$n\geq 0$ & $l$ & $K_{n}(\OO_{F,\mathcal{S}};\Z/2)$ \\ \hline
$8k$        & $2$ & $\f_{0}/\f_{1}=h^{0,4k}$, $\f_{1}=\ker(\rho_{2,4k+1})$ \\
$8k+1$    & $1$ & $\f_{0}=h^{1,4k+1}$ \\
$8k+2$    & $2$ & $\f_{0}/\f_{1}=h^{0,4k+1}$, $\f_{1}=h^{2,4k+2}$ \\
$8k+3$    & $2$ & $\f_{0}/\f_{1}=h^{1,4k+2}$, $\f_{1}=h^{3,4k+3}/\rho^{3}$ \\
$8k+4$    & $2$ & $\f_{0}/\f_{1}=h^{2,4k+3}$, $\f_{1}=h^{4,4k+4}/\rho^{3}$ \\
$8k+5$    & $2$ & $\f_{0}/\f_{1}=\ker(\rho^{2}_{1,4k+3})$, $\f_{1}=h^{3,4k+4}/\rho^{3}$ \\
$8k+6$    & $2$ & $\f_{0}/\f_{1}=\ker(\rho_{2,4k+4})$, $\f_{1}=h^{4,4k+5}/\rho^{3}$ \\
$8k+7$    & $1$ & $\f_{0}= \ker(\rho^{2}_{1,4k+4})$ \\ \hline
\end{tabular}
\end{center}
\caption{The mod 2 algebraic $K$-groups of $\OO_{F,\mathcal{S}}$}
\label{table:KGL2groupsOF}
\end{table}
\\ \noindent
In Table \ref{table:KGL2groupsOF}:
$h^{0,q}\cong\Z/2$, 
$h^{1,q}\cong\OO_{F,\mathcal{S}}^{\times}/2\oplus {_{2}}\Pic(\OO_{F,\mathcal{S}})/2$, 
$\ker(\rho_{2,1})\cong h^{2,1}\cong\Pic(\OO_{F,\mathcal{S}})/2$, 
$h^{2,q}\cong\Pic(\OO_{F,\mathcal{S}})/2\oplus {_{2}}\Br(\OO_{F,\mathcal{S}})$ for $q>1$, 
$h^{3,q}/\rho^{3}\cong (\Z/2)^{r_{1}-1}$,
$h^{4,q}/\rho^{3}\cong h^{3,q-1}/\rho^{2}\cong (\Z/2)^{t_{\mathcal{S}}^{+}-t_{\mathcal{S}}}$,
$ \ker(\rho^{2}_{1,q}) \cong \ker(\rho^{3}_{1,q})\cong \im (h^{1,q}_{+}\to h^{1,q})$,
$\ker(\rho_{2,q}) \cong \ker(\rho^{2}_{2,q}) \cong \ker(\rho^{3}_{2,q})\cong \im (h^{2,q}_{+}\to h^{2,q})$ for $q>1$.
Here $t_{\mathcal{S}}^{+}$ is the 2-rank of the narrow Picard group $\Pic_{+}(\OO_{F,\mathcal{S}})$ and $t_{\mathcal{S}}$ is the 2-rank of $\Pic(\OO_{F,\mathcal{S}})$.
Moreover, 
$\ker(\rho^{3}_{1,q})$ has 2-rank $r_{2}+s+t_{\mathcal{S}}^{+}$, 
where $r_{2}$ is the number of pairs of complex embeddings of $F$ and $s_{\mathcal{S}}$ is the number of finite primes in $\mathcal{S}$, 
while $\ker(\rho^{3}_{2,q})$ has 2-rank $s_{\mathcal{S}}+t_{\mathcal{S}}-1$.
\end{theorem}
\begin{proof}
Combining \aref{thm:KGL2-E2} with \eqref{equation:pmc} and \eqref{equation:sesPic} we deduce \aref{table:KGL2groupsOF}.
The 2-rank formulas follow from \eqref{equation:pmc}, 
see \cite[Lemma 13]{MR1928650}.
\end{proof}

\begin{remark}
\aref{theorem:KGLnumberfields} is in agreement with \cite[Theorem 7.8]{Rognes-Weibel}.
The calculations in \cite{Rognes-Weibel} are based on the unpublished work \cite{BL} (cf.~the comments in \cite{Suslin-GSS}).
\end{remark}

\begin{lemma}
\label{lem:KGL2-E284}
There is an isomorphism $E^{\infty}_{8,4,0}(\KGL/4) \cong H^{0,4}(\OO_{F,\mathcal{S}};\Z/4)$.
\end{lemma}
\begin{proof}
The $d^{1}$-differential is trivial on the mod $4$ class $\widetilde{\tau}^{4}$ generating $E^{1}_{8,4,0}(\KGL/4)\cong H^{0,4}(\OO_{F,\mathcal{S}};\Z/4)$.
Here $\widetilde{\tau}$ is the generator of $H^{0,1}(\OO_{F,\mathcal{S}};\Z/4) \cong \mu_4(\OO_{F,\mathcal{S}}^{\times})$.
Since $E^{r}_{\ast,\ast,\ast}(\KGL/4)$ is an algebra spectral sequence, 
this follows from the Leibniz rule,  
i.e., 
$d^{1}(\widetilde{\tau}^4) = 4\widetilde{\tau}^3 d^{1}(\widetilde{\tau}) = 0$.
\end{proof}

Since $\KGL$ is a motivic ring spectrum, 
see e.g., \cite{KGLstrict}, 
\eqref{equation:R-pairing} yields a pairing of slice spectral sequences 
\begin{equation}
\label{equation:KGLpairing}
E^r_{p,q,w}(\KGL/4) \otimes E^r_{p',q',w'}(\KGL/2)
\to 
E^r_{p+p',q+q',w+w'}(\KGL/2).
\end{equation}
The generator $\widetilde{\tau}^4$ commutes with the differentials and acts as multiplication by $\tau^4 \in h^{0,4}$ under the pairing \eqref{equation:KGLpairing}. 
Thus it induces a periodicity isomorphism on $E^r_{p,q,w}(\KGL/2)$ in the range $q\leq p$.
For the abutment we deduce the following result.

\begin{corollary}
\label{corollary:KGLnumberfields}
For $n\geq 1$ the permanent cycle $\widetilde{\tau}^{4}$ induces an 8-fold periodicity isomorphism 
\begin{equation*}
K_{n}(\OO_{F,\mathcal{S}};\Z/2)
\cong 
K_{n+8}(\OO_{F,\mathcal{S}};\Z/2). 
\end{equation*}
\end{corollary} 
\begin{proof}
Follows from Theorem \ref{theorem:KGLnumberfields} and the discussion of \eqref{equation:KGLpairing}.
\end{proof}

More generally,
comparing with real closed fields, see \aref{ex:KGL2-R}, the techniques in \cite[\S4,5]{RO} yield a periodicity result for fields of finite virtual cohomological dimension.
\begin{corollary}
\label{corollary:KGLfields}
Let $F$ be a field of $\Char(F)\neq 2$ and assume $\vcd(F)<\infty$.
The permanent cycle $\widetilde{\tau}^{4}$ induces an 8-fold periodicity isomorphism 
\begin{equation*}
K_{n}(F;\Z/2)\cong K_{n+8}(F;\Z/2) 
\end{equation*}
for all $n\geq\vcd(F)-1$.
\end{corollary}

Our next aim is to determine $\KGL/2^{n}$ for any $n\geq 1$ over rings of $\mathcal{S}$-integers.
We first calculate the differentials in the slice spectral sequence over $\R$, 
and then transfer these to $\OFS$ under the real embeddings of $F$.
As in the mod $2$ case the slice spectral sequence collapses at its $E^2$-page.
This is made possible by the following description of the mod $2^{n}$ motivic cohomology of $\R$.

\begin{lemma}
\label{lem:coh-R-mod2n}
Set $H_n^{a,b} = H^{a,b}(\R; \Z/2^{n})$,
and let
$(\pr^{2^{n}}_{2})^{a,b} : H^{a, b}_n \to h^{a, b}$,
$(\inc^{2}_{2^{n}})^{a,b} : h^{a, b} \to H^{a, b}_n$,
$(\partial^{2^{n}}_{2})^{a,b} : H^{a,b}_n \to h^{a+1,b}$ and
$(\partial^{2}_{2^{n}})^{a,b} : h^{a,b}_n \to H^{a+1,b}_n$
be the maps induced by the short exact sequences 
\[
0 \to \Z/2 \to \Z/2^{n+1} \to \Z/2^{n} \to 0
\text{ and }
0 \to \Z/2^{n} \to \Z/2^{n+1} \to \Z/2 \to 0.
\]
\begin{itemize}
\item 
If $a-b \equiv 1 \bmod 2$, 
$(\inc^{2}_{2^{n}})^{a,b}, (\partial^{2^{n}}_{2})^{a,b}$ and $(\partial^{2}_{2^{n}})^{a,b}$ are isomorphisms, and $(\pr^{2^{n}}_{2})^{a,b}$ is trivial.
\item 
If $a>0$, $a-b \equiv 0 \bmod 2$,
$(\pr^{2^{n}}_{2})^{a,b}$ is an isomorphism, $(\inc^{2}_{2^{n}})^{a,b}, (\partial^{2^{n}}_{2})^{a,b}$ and $(\partial^{2}_{2^{n}})^{a,b}$ are trivial.
\item 
If $a=0$ and $b \equiv 0 \bmod 2$,
$(\partial^{2^{n}}_{2})^{a,b}$ and $(\partial^{2}_{2^{n}})^{a,b}$ are trivial, 
and there are nonsplit extensions
\[
0 \to h^{0, b} \xrightarrow{\inc^{2}_{2^{n}}} H^{0, b}_{n} \to H_{n-1}^{0,b} \to 0
\text{ and }
0 \to H_{n-1}^{0,b} \to H^{0, b}_{n} \xrightarrow{\pr^{2^{n}}_{2}} h^{0, b} \to 0.
\]
\end{itemize}
\end{lemma}
\begin{proof}
This follows by induction on $n$ using the diagrams with exact rows,
\[
\begin{tikzcd}
  & & & & h^{a+1, b} \ar[d, "\inc^{2}_{2^{n}}"] \\
  h^{a-1,b} \ar[r, "\partial^{2}_{2^{n}}"]\ar[rd, "\Sq^1"]  & H_n^{a,b} \ar[r]\ar[d, "\pr^{2^{n}}_{2}"] & H_{n+1}^{a,b} \ar[r, "\pr^{2^{n+1}}_{2}"] & h^{a, b} \ar[r, "\partial^{2}_{2^{n}}"]\ar[rd, "0"] & H^{a+1,b}_n \ar[d, "\partial^{2^{n}}_{2}"] \\
  & h^{a,b} & & & h^{a+2, b}
\end{tikzcd}
\]
when $a-b \equiv 0 \bmod 2$, 
in which case $\Sq^1 : h^{a,b} \to h^{a+1,b}$ is trivial,
and
\[
\begin{tikzcd}
  h^{a-2,b}\ar[d, "\partial^{2}_{2^{n}}"] \ar[rd, "0" ]& & & h^{a,b} \ar[d, "\inc^{2}_{2^{n}}"]\ar[rd, "\Sq^1"] & \\
  H_n^{a-1,b} \ar[r, "\partial^{2^{n}}_{2}"]\ar[d, "\pr^{2^{n}}_{2}"]  & h^{a,b} \ar[r, "\inc^{2}_{2^{n+1}}"] & H_{n+1}^{a,b} \ar[r] & H_n^{a, b} \ar[r, "\partial^{2^{n}}_{2}"] & h^{a+1,b} \\
  h^{a-1,b}
\end{tikzcd}
\]
when $a-b\equiv 1 \bmod 2$, 
in which case $\Sq^1:h^{a,b} \to h^{a+1,b}$ is an isomorphism, $a \geq 0$.
In the above we used that $[\MZ/2, \Sigma^{2,0}\MZ/2] = 0$ (see \aref{section:mcatSa}).
The extensions are nontrivial by a standard Galois cohomology calculation.
\end{proof}

\begin{remark}
In fact there is an isomorphism of algebras
\[
H^{*,*}(\R;\Z/2^{n}) \cong \Z/2^{n}[u,\tau,\rho]/(2\rho,2\tau,\tau^2),
\]
where $u, \tau$, and $\rho$ are the generators of $H^{0,2}(\R;\Z/2^{n})$, $H^{0,1}(\R;\Z/2^{n})$, and $H^{1,1}(\R;\Z/2^{n})$.
\end{remark}

\begin{lemma}
If a $d^1$-differential for $\KGL/2$ over the real numbers is surjective then the corresponding $d^1$-differential for $\KGL/2^{n}$ is also surjective.
\end{lemma}
\begin{proof}
The canonical maps $\KGL/2 \to \KGL/2^{n} \to \KGL/2$ induce a commutative diagram
\[
\begin{tikzcd}
E^{1}_{p,q,w}(\KGL/2) = h^{2q-p,q-w} \ar[r, "d^1"]\ar[d] & h^{2q-p+3,q-w+1} = E^{1}_{p-1,q+1,w}(\KGL/2) \ar[d] \\
E^{1}_{p,q,w}(\KGL/2^{n}) = H_n^{2q-p,q-w} \ar[r, "d^1"]\ar[d] & H_n^{2q-p+3,q-w+1} = E^{1}_{p-1,q+1,w}(\KGL/2^{n}) \ar[d] \\
E^{1}_{p,q,w}(\KGL/2) = h^{2q-p,q-w} \ar[r, "d^1"] & h^{2q-p+3,q-w+1} = E^{1}_{p-1,q+1,w}(\KGL/2).
\end{tikzcd}
\]
If $q - p + w \equiv 0 \bmod 2$ the lower vertical maps are surjections and the lower left vertical map is an isomorphism (cf.~\aref{lem:coh-R-mod2n}).
If $q - p + w \not\equiv 0 \bmod 2$ the upper vertical maps are isomorphisms.
\end{proof}

\begin{corollary}
\label{corollary:2adicsss}
The $E^2=E^{\infty}$-page of the slice spectral sequence for $\KGL/2^{n}$ over $\OO_{F,\mathcal{S}}$ is given by
\[
E^{\infty}_{p,q,w}(\KGL/2^{n})
\cong 
\begin{cases}
\bar{H}^{2q-p,q-w}(\OO_{F,\mathcal{S}};\Z/2^{n}) & p - w - q \equiv 0, 1 \bmod 4 \\
\wt{H}^{2q-p,q-w}(\OO_{F,\mathcal{S}};\Z/2^{n}) & p - w - q \equiv 2, 3 \bmod 4.
\end{cases}
\]
Here we set 
$$
\bar{H}^{p,q}(\OO_{F,\mathcal{S}};\Z/2^{n}) 
:= 
\begin{cases}\coker(H^{p-3,q-1}(\OO_{F,\mathcal{S}};\Z/2^{n}) \to \bigoplus^{r_{1}} H^{p-3,q-1}(\R;\Z/2^{n})) &  p \geq 3 \\
H^{p,q} & p < 3
\end{cases}
$$ 
and
$$
\wt{H}^{p,q}(\OO_{F,\mathcal{S}};\Z/2^{n}) 
:= 
\ker(H^{p,q}(\OO_{F,\mathcal{S}};\Z/2^{n}) \to \bigoplus^{r_{1}} H^{p,q}(\R;\Z/2^{n})).
$$
In particular, 
$E^{\infty}_{p,q,w}(\KGL/2^{n}) = 0$ for $2q - p \geq 5$.
\end{corollary}
\begin{proof}
For $p \geq 3$ the real embeddings of $F$ induce an isomorphism (see \aref{section:mcatSa})
\[
H^{p,q}(\OO_{F,\mathcal{S}};\Z/2^{n}) 
\overset{\cong}{\rightarrow} 
\bigoplus^{r_{1}} H^{p,q}(\R;\Z/2^{n}).
\]
We have a commutative diagram
\[
\begin{tikzcd}
H^{2q - p, q - w}(\OO_{F,\mathcal{S}}; \Z/2^{n}) \ar[r] \ar[d, "d^1"] & \bigoplus^{r_{1}} H^{2q-p,q-w}(\R;\Z/2^{n})\ar[d, "\bigoplus^{r_{1}}d^1"] \\
H^{2q-p+3,q-w+1}(\OO_{F,\mathcal{S}};\Z/2^{n}) \ar[r, "\cong"] & \bigoplus^{r_{1}} H^{2q-p,q-w}(\R;\Z/2^{n}),
\end{tikzcd}
\]
where the lower horizontal map is an isomorphism for $2q - p \geq 0$.
This determines all the differentials, and the $E^2$-page takes the above form. 
For degree reasons this is the $E^{\infty}$-page.
\end{proof}

When $\mathcal{S}$ is a finite set, 
finite generation of $\KGL_{\ast,\ast}(\OO_{F,\mathcal{S}})$ implies the inverse limit
\begin{equation}
\label{equation:inverselimitsss}
\lim_{n} E^{r}_{p,q,w}(\KGL/2^{n})
\end{equation}
defines a spectral sequence with the $d^{r}$-differentials given by inverse limits of the slice $d^{r}$-differentials.
Its $E^{\infty}_{p,q,w}$-term is given by the inverse limit $\lim_{n} E^{\infty}_{p,q,w}(\KGL/2^{n})$ identified in \aref{corollary:2adicsss}.
Due to collapse at the $E^{2}=E^{\infty}$-page for each $n\geq 1$ the inverse limit spectral sequence converges strongly to the $2$-adic algebraic $K$-groups $K_{\ast}(\OO_{F,\mathcal{S}};\Z_{2})$.
In this way we deduce two-primary calculations first carried out in \cite{Kahn97}, \cite{Levine99}, and \cite[Theorem 0.6]{Rognes-Weibel}.
Moreover, 
by Voevodsky's proof of the Bloch-Kato conjecture at every odd prime $\ell$ \cite{Voevodsky:Z/l}, 
assuming $\ell$ is invertible in $\OFS$ and $\mathcal{S}$ is finite,
a straightforward calculation with the slice spectral sequence for $\KGL$ yields isomorphisms for $\ell$-adic coefficients 
\begin{equation}
\label{equation:oddiso}
K_{2n-m}(\OO_{F,\mathcal{S}};\Z_{\ell})
\overset{\cong}{\rightarrow} 
H^{m,n}(\OO_{F,\mathcal{S}};\Z_{\ell})
\end{equation}
for $n\geq 2$, $m=1,2$ (see \aref{section:mcatSa} for a review of the integral motivic cohomology groups of $\OO_{F,\mathcal{S}}$).
\aref{corollary:2adicsss},
localization and purity, 
and \eqref{equation:oddiso} conspire to give an integral calculation for arbitrary $\mathcal{S}$ containing the infinite primes.

\begin{theorem}
\label{theorem:integralKGL}
For $n\geq 2$, $m=1,2$, 
the map 
\begin{equation}
\label{equation:integraliso}
K_{2n-m}(\OO_{F,\mathcal{S}})
\rightarrow
H^{m,n}(\OO_{F,\mathcal{S}})
\end{equation}
is an isomorphism when $2n-m\equiv 0,1,2,7\bmod 8$, 
a surjection with kernel $(\Z/2)^{r_{1}}$ when $2n-m\equiv 3\bmod 8$, 
and an injection with cokernel $(\Z/2)^{r_{1}}$ when $2n-m\equiv 6\bmod 8$.
Finally, 
when $n\equiv 3\bmod 4$, 
there is an exact sequence
\begin{equation}
\label{equation:integralexact}
0
\to
K_{2n-2}(\OO_{F,\mathcal{S}})
\to
H^{2,n}(\OO_{F,\mathcal{S}})
\to
(\Z/2)^{r_{1}}
\to
K_{2n-1}(\OO_{F,\mathcal{S}})
\to
H^{1,n}(\OO_{F,\mathcal{S}})
\to
0.
\end{equation}
\end{theorem}

\section{Higher Witt-theory and hermitian $K$-theory}
\label{section:hwtahkt}
Throughout we work over a base field $F$ of $\Char(F)\neq 2$ or the ring of $\mathcal{S}$-integers in a number field. 
From the identification \eqref{equation:wittheoryslices} in \cite[Theorem 4.28]{slices} of the $q$th slice of higher Witt-theory we obtain
\begin{equation}
\label{equation:KW/2slices}
\s_{q}(\KW/2)\simeq \bigvee_{j\in\Z} \Sigma^{q+j,q}\MZ/2.
\end{equation}
Since $(\Sq^1, \id)$ is a nontrivial automorphism of $\MZ/2\vee\Sigma^{1,0}\MZ/2$,
the wedge product decomposition of $\s_{q}(\KW/2)$ is only unique up to $\MZ$-module isomorphisms.
Similar observations apply to $\KQ$ and its mod $2$ slices in \eqref{equation:hermitianktheoryslices}.
For the purpose of systematic calculations we fix an explicit choice:
\begin{convention}
\label{convention:kq-to-kt-slices}
Let $\E$ denote $\KQ$ or $\KW$.
\begin{enumerate}
\item The canonical maps $\E \to \E/2$ and $\E/2 \to \Sigma^{1,0}\E$ induce either inclusions or projections on the slice summands.
\item The naturally induced map $\s_{q}(\KQ/2) \to \s_{q}(\KW/2)$ is compatible with the canonical map of cofiber sequences from $\KQ \to \KQ \to \KQ/2$ to $\KW \to \KW \to \KW/2$.
\end{enumerate}
\end{convention}
In particular, 
the top summand $\Sigma^{2q,q}\MZ/2$ of $\s_{q}(\KQ/2)$ maps by $(\Sq^1,\id)$ to $\s_{q}(\KW/2)$ if $q$ is even. 
All other summands of $\s_{q}(\KQ/2)$ map by the identity to $\s_{q}(\KW/2)$.

Over a field $F$ of $\Char(F)\neq 2$, 
recall that $\pi_{p,q}\KW \cong W(F)$ if $p \equiv q \bmod 4$ and $\pi_{p,q}\KW$ is trivial in all other degrees.
By the mod $2$ universal coefficient sequence we deduce
\[
\pi_{p,q}(\KW/2)
\cong 
\begin{cases}
W(F)/2 & p \equiv q \bmod 4 \\
{}_{2}W(F) & p \equiv q+1 \bmod 4 \\
0 & \mathrm{otherwise.}
\end{cases}
\]

\begin{theorem}
\label{thm:KW2-diff}
The restriction of the slice $\dd^{1}$-differential to the summand $\Sigma^{q + j,q}\MZ/2$ of $\s_{q}(\KW/2)$ in \eqref{equation:KW/2slices} is given by 
\begin{equation}
\dd^{1}(\KW/2)(q,j)
= 
\begin{cases}
(\Sq^{3}\Sq^{1},0,\Sq^{2},0,0) & j\equiv 0\bmod 4 \\
(\Sq^{3}\Sq^{1},\Sq^{2}\Sq^{1}+\Sq^{3},\Sq^{2},\rho+\tau\Sq^{1},0) & j\equiv 1\bmod 4 \\
(\Sq^{3}\Sq^{1},0,\Sq^{2}+\rho\Sq^{1},0,\tau) & j\equiv 2\bmod 4 \\
(\Sq^{3}\Sq^{1},\Sq^{2}\Sq^{1}+\Sq^{3},\Sq^{2}+\rho\Sq^{1},\tau\Sq^{1},\tau) & j\equiv 3\bmod 4. 
\end{cases}
\label{equation:KW2-diff1}
\end{equation}
Figure \ref{figure:KW2differentials} shows the slice $\dd^{1}$-differentials for $\KW/2$.
Each dot is a suspension of $\MZ/2$.
The simplicial degree (resp.~weight) is indicated horizontally (resp.~vertically).
\begin{figure}

\begin{center}
\begin{tikzpicture}[scale=1,line width=1pt]

{\draw[fill]     
(-4,-1) circle (1pt) (-3,-1) circle (1pt) (-2,-1) circle (1pt) (-1,-1) circle (1pt) (0,-1) circle (1pt) 
(1,-1) circle (1pt) (2,-1) circle (1pt) (3,-1) circle (1pt) (4,-1) circle (1pt) (5,-1) circle (1pt) 
(-5,-1) circle (0pt) node[left=-1pt] {$q-1$}
;}
{\draw[fill]     
(-4,0) circle (1pt) (-3,0) circle (1pt) (-2,0) circle (1pt) (-1,0) circle (1pt) (0,0) circle (1pt) 
(1,0) circle (1pt) (2,0) circle (1pt) (3,0) circle (1pt) (4,0) circle (1pt) (5,0) circle (1pt) 
(-5,0) circle (0pt) node[left=-1pt] {$q$}
;}
{\draw[fill]     
(-4,1) circle (1pt) (-3,1) circle (1pt) (-2,1) circle (1pt) (-1,1) circle (1pt) (0,1) circle (1pt) 
(1,1) circle (1pt) (2,1) circle (1pt) (3,1) circle (1pt) (4,1) circle (1pt) (5,1) circle (1pt) 
(-5,1) circle (0pt) node[left=-1pt] {$q+1$}
;}
{\draw[fill]     
(-4,2) circle (1pt) (-3,2) circle (1pt) (-2,2) circle (1pt) (-1,2) circle (1pt) (0,2) circle (1pt) 
(1,2) circle (1pt) (2,2) circle (1pt) (3,2) circle (1pt) (4,2) circle (1pt) (5,2) circle (1pt) 
(-5,2) circle (0pt) node[left=-1pt] {$q+2$}
;}
{\draw[fill]     
(-4,3) circle (1pt) (-3,3) circle (1pt) (-2,3) circle (1pt) (-1,3) circle (1pt) (0,3) circle (1pt) 
(1,3) circle (1pt) (2,3) circle (1pt) (3,3) circle (1pt) (4,3) circle (1pt) (5,3) circle (1pt) 
(-5,3) circle (0pt) node[left=-1pt] {$q+3$}
;}
{\draw[fill]     
(-4,-2) circle (0pt) node[below=-1pt] {$4i-4$}
(-2,-2) circle (0pt) node[below=-1pt] {$4i-2$}
(0,-2) circle (0pt) node[below=-1pt] {$4i$} 
(2,-2) circle (0pt) node[below=-1pt] {$4i+2$}
(4,-2) circle (0pt) node[below=-1pt] {$4i+4$}
;}
{\draw[fill]     
(-2,-1) circle (0pt) node[below=1pt] {$\Sq^2$}
;}
{\draw[fill,green]     
(0,-1) circle (0pt) node[below=1pt] {$\Sq^2+\rho\Sq^1$}
;}
{\draw[fill,red]     
(2,-1) circle (0pt) node[below=1pt] {$\tau$}
;}
{\draw[fill,blue]     
(4,-1) circle (0pt) node[below=1pt] {$\Sq^3\Sq^1$}
;}

{\draw[red,->] 
(-3,-1) -- (-4,0)
;}
{\draw[green,->] 
(-3,-1) -- (-2,0)
;}
{\draw[blue,->] 
(-3,-1) -- (0,0)
;}
{\draw[->] 
(-1,-1) -- (0,0)
;}
{\draw[blue,->] 
(-1,-1) -- (2,0)
;}
{\draw[red,->] 
(1,-1) -- (0,0)
;}
{\draw[green,->] 
(1,-1) -- (2,0)
;}
{\draw[blue,->] 
(1,-1) -- (4,0)
;}
{\draw[->] 
(3,-1) -- (4,0)
;}
{\draw[blue] 
(3,-1) -- (5.5,-.16)
;}
{\draw[red,->]
(5,-1) -- (4,0)
;}
{\draw[green]
(5,-1) -- (5.5,-.5)
;}
{\draw[blue,->] 
(-4,-.67) -- (-2,0)
;}
{\draw[blue,->] 
(-4,.67) -- (-3,1)
;}
{\draw[->] 
(-4,0) -- (-3,1)
;}
{\draw[blue,->] 
(-4,0) -- (-1,1)
;}
{\draw[red,->] 
(-2,0) -- (-3,1)
;}
{\draw[green,->] 
(-2,0) -- (-1,1)
;}
{\draw[blue,->] 
(-2,0) -- (1,1)
;}
{\draw[->] 
(0,0) -- (1,1)
;}
{\draw[blue,->] 
(0,0) -- (3,1)
;}
{\draw[red,->] 
(2,0) -- (1,1)
;}
{\draw[green,->] 
(2,0) -- (3,1)
;}
{\draw[blue,->] 
(2,0) -- (5,1)
;}
{\draw[->] 
(4,0) -- (5,1)
;}
{\draw[blue] 
(4,0) -- (5.5,.5)
;}
{\draw[red,->]
(5.5,0.5) -- (5,1)
;}
{\draw[->] 
(-3,1) -- (-2,2)
;}
{\draw[blue,->] 
(-3,1) -- (0,2)
;}
{\draw[blue,->] 
(-4,1.33) -- (-2,2)
;}
{\draw[red,->] 
(-1,1) -- (-2,2)
;}
{\draw[green,->] 
(-1,1) -- (0,2)
;}
{\draw[blue,->] 
(-1,1) -- (2,2)
;}
{\draw[->] 
(1,1) -- (2,2)
;}
{\draw[blue,->] 
(1,1) -- (4,2)
;}
{\draw[red,->] 
(3,1) -- (2,2)
;}
{\draw[green,->] 
(3,1) -- (4,2)
;}
{\draw[blue] 
(3,1) -- (5.5,1.83)
;}
{\draw
(5,1) -- (5.5,1.5)
;}
{\draw[blue] 
(5,1) -- (5.5,1.17)
;}
\end{tikzpicture}
\end{center}

\caption{Slice $\dd^{1}$-differentials for $\KW/2$}
\label{figure:KW2differentials}
\end{figure}
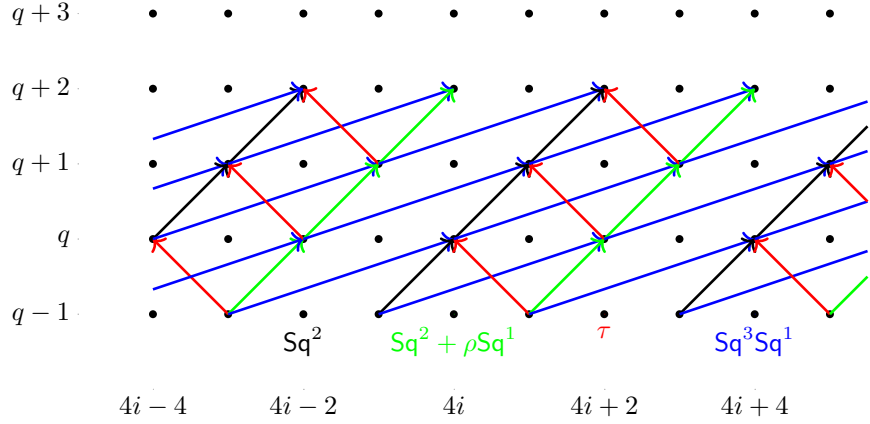
\end{theorem}
\begin{proof}
According to \cite[Theorem 6.3]{slices} the corresponding slice $\dd^{1}$-differential of $\KW$ is given by 
\[
\dd^{1}(\KW)(q,j)
= 
\begin{cases}
(\Sq^3\Sq^{1},\Sq^{2},0) & j \equiv 0 \bmod 4 \\
(\Sq^3\Sq^{1},\Sq^{2}+\rho\Sq^{1},\tau) & j \equiv 2 \bmod 4.
\end{cases}
\]
From the homotopy cofiber sequence $\KW\to \KW/2\to \Sigma^{1,0}\KW\to \Sigma^{1,0}\KW$ we get $\dd^{1}(\KW)(q,j)=\dd^{1}(\KW/2)(q,j)$ when $j$ is even.
This proves \eqref{equation:KW2-diff1} for $j\equiv 0, 2 \bmod 4$.
When $j$ is odd, 
$\pr\circ \dd^{1}(\KW)(q,j) = \pr\circ \dd^{1}(\KW/2)(q,j)$, 
where $\pr$ is the projection of $\Sigma^{1,0}\s_{q+1}(\KW/2)$ onto the odd summands.
Hence, by \aref{lem:steenrod-alg}, we have
\begin{equation}
\label{equation:KW2-diff2}
\dd^{1}(\KW/2)(q,j)
=
\begin{cases}
(\Sq^{3}\Sq^{1},a\Sq^{2}\Sq^{1}+b\Sq^{3},\Sq^{2},\phi+c\tau\Sq^{1},0) & j\equiv 1\bmod 4 \\
(\Sq^{3}\Sq^{1},a^{\prime}\Sq^{2}\Sq^{1}+b^{\prime}\Sq^{3},\Sq^{2}+\rho\Sq^{1},
\phi^{\prime}+c^{\prime}\tau\Sq^{1},\tau) & j\equiv 3\bmod 4, 
\end{cases}
\end{equation}
where $a$, $b$, $c$, $a^{\prime}$, $b^{\prime}$, $c^{\prime}\in h^{0,0} \cong \Z/2$ and $\phi,\phi^{\prime}\in h^{1,1}$.
Since the slice $\dd^{1}$-differential squares to zero, the product of the matrices
\[ 
\begin{pmatrix} 
\Sq^{3}\Sq^{1} & 0  & 0   & 0   & 0  \\
a\Sq^{2}\Sq^{1}+b\Sq^{3} & \Sq^{3}\Sq^{1}  & 0  & 0  & 0 \\
\Sq^{2}  & 0 & \Sq^{3}\Sq^{1}  & 0    & 0 \\ 
\phi+c\tau\Sq^{1}  & \Sq^{2}  & a^{\prime}\Sq^{2}\Sq^{1}+b^{\prime}\Sq^{3} & \Sq^{3}\Sq^{1}  & 0 \\
0 & 0  & \Sq^{2}+\rho\Sq^{1}  & 0  & \Sq^{3}\Sq^{1} \\
0 & 0  & \phi^{\prime} +c^{\prime} \tau\Sq^{1} & \Sq^{2}+\rho\Sq^{1} & a\Sq^{2}\Sq^{1}+b\Sq^{3} \\
0 & 0  &  \tau  & 0  & \Sq^{2}\\
0 & 0  &  0  & \tau  & \phi+c\tau\Sq^{1}\\
0 & 0 & 0 & 0 & 0
\end{pmatrix} 
\]
and $(\Sq^{3}\Sq^{1},a^{\prime}\Sq^{2}\Sq^{1}+b^{\prime}\Sq^{3},\Sq^{2}+\rho\Sq^{1},\phi^{\prime}+c^{\prime}\tau\Sq^{1},\tau)$ is zero. 
That is, 
the matrix
\[
\begin{pmatrix} 0 \\ 0 \\ 0 \\ 
(\phi+\phi^{\prime}+c^{\prime} \rho)\Sq^{3}\Sq^{1}+(b^{\prime}+a^{\prime}) \Sq^{2}\Sq^{3} \\ 0\\   
(a+b)\rho \Sq^{2}+(b+c^{\prime})\tau\Sq^{3}+(a+c^{\prime})\tau\Sq^{2}\Sq^{1} +(b+c^{\prime})\rho^{2}\Sq^{1}\\ 0\\ 
\tau(\phi+\phi^{\prime}+c\rho)+(c+c^{\prime})\tau^{2}\Sq^{1} \\ 0
\end{pmatrix} 
\]
has zero entries.
This implies $\phi+\phi^{\prime}+c^{\prime}\rho=a^{\prime}+b^{\prime}=0$, and $a=b=c=c^{\prime}$.

Next we consider the commutative diagram for $q$ even
\[
\begin{tikzcd}
\s_{q}(\KQ/2) \ar[r]\ar[d, "\dd^{1}(\KQ/2)(q)"] & \s_{q}(\KW/2)\ar[d, "\dd^{1}(\KW/2)(q)"] \\
\Sigma^{1,0}\s_{q+1}(\KQ/2) \ar[r] & \Sigma^{1,0}\s_{q+1}(\KW/2).
\end{tikzcd}
\]
Here the upper horizontal map restricts to $(\Sq^{1},\id)$ on the top summand $\Sigma^{2q,q}\MZ/2$ of $\s_{q}(\KQ/2)$.
The top summand of $\Sigma^{1,0}\s_{q+1}(\KQ/2)$ is $\Sigma^{2q+3,q+1}\MZ/2$.
Hence $\Sigma^{2q,q}\MZ/2$ maps trivially to the summand $\Sigma^{2q+4,q+1}\MZ/2$ of $\Sigma^{1,0}\s_{q+1}(\KW/2)$, 
i.e.,
\[
0 = 
\begin{cases}
(a\Sq^{2}\Sq^{1} + b\Sq^3)\Sq^{1} + \Sq^3\Sq^{1}\id & q \equiv 0\bmod 4 \\
(a'\Sq^{2}\Sq^{1} + b'\Sq^3)\Sq^{1} + \Sq^3\Sq^{1}\id & q \equiv 2 \bmod 4.
\end{cases}
\]
It follows that $b=b'=1$, and thus $a=a'=b=b'=c=c'=1$.

The relation $\phi=\rho\in h^{1,1}$ is shown in \aref{lem:phi=rho}.

Finally,
a base change argument as in \cite[Lemma 5.1]{slices} extends the result to $\OFS$.
\end{proof}

\begin{theorem}
\label{thm:KW2-E2}
Over fields $F$ of $\Char(F)\neq 2$ and $q'=q-w$, 
the $E^{2}$-page of the slice spectral sequence for $\KW/2$ is given by
\[
E^{2}_{p,q,w}(\KW/2)
\cong 
\begin{cases}
h^{q',q'}/\rho & p\equiv w\bmod 4 \\
\ker(\rho_{q',q'}) & p\equiv w+1\bmod 4 \\
0 & \mathrm{otherwise}.
\end{cases}
\]
The same identifications hold over the ring of $\mathcal{S}$-integers in a number field with the exceptions
\[
E^{2}_{p,w+2,w}(\KW/2) 
\cong 
\begin{cases}
h^{2,2}/(\rho,\tau) & p - w \equiv 0 \bmod 4 \\
\ker(\rho_{2,2})/\tau & p - w \equiv 1 \bmod 4 \\
0 & p - w \equiv 2, 3 \bmod 4,
\end{cases}
\]
\[
E^{2}_{p,w+1,w}(\KW/2) 
\cong 
\begin{cases}
h^{1,1}\oplus h^{2,1} & p - w \equiv 0 \bmod 4 \\
\ker(\rho_{1,1}) & p - w \equiv 1 \bmod 4 \\
0 & p - w \equiv 2 \bmod 4 \\
h^{2,1} & p - w \equiv 3 \bmod 4.
\end{cases}
\]
\end{theorem}
\begin{proof}
Note that $E^{1}_{p,q,w}(\KW/2) \cong \bigoplus_{i=0}^{q'}h^{i,q'}$ by \eqref{equation:KW/2slices}.
The slice differentials 
$$
d^{1}_{p,q,w}\colon\bigoplus_{i} h^{i,q'} \to \bigoplus_{i} h^{i,q'+1}
$$
from \aref{thm:KW2-diff} are given by matrices with entries $0$ and $\rho^{a}\tau^{b}$, 
where $a\in\N$, $b\in\Z$, 
cf.~\aref{tbl:Sq-action}.
When $b<0$ this makes sense in the range where multiplication by $\tau$ is an isomorphism on the mod $2$ motivic cohomology ring $h^{*,*}$, 
see \aref{section:mcatSa}.

Let $d^{1}_{p,q,w}(i,j)$ be the restriction of $d^{1}_{p,q,w}$ to $h^{i,q'} \to h^{j,q'+1}$. 
For $i,j\leq q$ we have 
\begin{equation}
d^{1}_{p,q,w}(i,j)=0 \text{ if } \vert i-j\vert > 4,\
d^{1}_{p,q,w}(i,j)=d^{1}_{p,q,w}(i+4,j+4),\
d^{1}_{p,q+1,w}(i,j)=d^{1}_{p,q,w}(i,j).
\label{equation:KW-diff-periodicity}
\end{equation}
From \eqref{equation:KW-diff-periodicity} we deduce the repetitive form of the matrix $(d^{1}_{p,q,w}(i,j))_{i,j}$ as indicated in \aref{fig:KW2outline1}. 
Every void box indicates a trivial map.

\begin{figure}[h!]

\begin{center}
\begin{tikzpicture}[yscale=-1,font=\footnotesize,line width=2pt]
        \draw[help lines] (0,0) grid (6.5,6.5);
        {\node[label=right:$h^{q'+1,q'+1}$] at (-2.0,.5) {};}
        {\node[label=right:$h^{q',q'+1}$] at (-2.0,1.5) {};}
        {\node[label=right:$\vdots$] at (-1.5,4.5) {};}
        \foreach \i in {1,...,2} {\node[label=right:$h^{q' -\i,q'+1}$] at (-2.,\i+1.5) {};}

        {\node[label=above:$h^{q',q'}$] at (.5,0.) {};}
        {\node[label=above:$\cdots$] at (4+.5,0) {};}
        \foreach \i in {1,...,3} {\node[label=above:$h^{q'-\i,q'}$] at (\i+.5,0.) {};}
        
{\draw[black,] (0,0) -- (4,0);}
{\draw[black,] (4,0) -- (4,4);}
{\draw[black,] (0,0) -- (0,4);}
{\draw[black,] (0,4) -- (4,4);}
{\node[label=center:$\rho^{2} \tau^{-1}$] at (1.5,0.5) {};}
{\node[label=center:$\rho^{4} \tau^{-3}$] at (3.5,0.5) {};}
{\node[label=center:$\rho$] at (1.5,1.5) {};}
{\node[label=center:$\rho^{2} \tau^{-1}$] at (2.5,1.5) {};}
{\node[label=center:$\rho^{3} \tau^{-2}$] at (3.5,1.5) {};}
{\node[label=center:$\tau$] at (1.5,2.5) {};}
{\node[label=center:$\rho^{2} \tau^{-1}$] at (3.5,2.5) {};}
{\node[label=center:$\tau$] at (2.5,3.5) {};}
{\draw[black,] (4,4) -- (7,4);}
{\draw[black,] (4,4) -- (4,7);}
{\node[label=center:$\rho^{2} \tau^{-1}$] at (5.5,4.5) {};}
{\node[label=center:$\rho$] at (5.5,5.5) {};}
{\node[label=center:$\rho^{2} \tau^{-1}$] at (6.5,5.5) {};}
{\node[label=center:$\tau$] at (5.5,6.5) {};}
{\draw[red,dashed] (0,0) -- (4,0);}
{\draw[red,dashed] (4,0) -- (4,5);}
{\draw[red,dashed] (0,0) -- (0,5);}
{\draw[red,dashed] (0,5) -- (4,5);}
        \end{tikzpicture}
        \end{center}

\caption{The matrix $(d^{1}_{p,q,0}(i,j))_{i,j\leq q}$ for $p \equiv 0, 3 \bmod 4$. 
The entries for $p\equiv 0\bmod 4$ are shown. 
The dotted red rectangle shows the matrix for $d^{1}_{p,q,w}$ when $q' = 3$.}
\label{fig:KW2outline1}
\end{figure}
The $d^{1}$-differential exiting $E^{1}_{p,q,w}(\KW/2)$ is given by a $(q'+2)\times (q'+1)$ matrix similar to the one in \aref{fig:KW2outline1}.
Since $h^{i,q'}$ is trivial for $i<0$ this matrix is of size $(4N+1)\times 4N$ for some $N$.
\aref{tbl:KW2-diffs} displays all details for $q'=3$.
All other differentials are determined by \aref{tbl:KW2-diffs} and \eqref{equation:KW-diff-periodicity}.
\begin{table}[h!]
\begin{center}

\begin{tabular}{>{$}c<{$}|>{$}c<{$}}
p-w\equiv 0\bmod 4, \begin{pmatrix}
0 & \rho^{2} \tau^{-1} & 0 & \rho^{4} \tau^{-3}\\
0 & \phi & \rho^{2} \tau^{-1} & \rho^{3} \tau^{-2}\\
0 & \tau & 0 & \rho^{2} \tau^{-1}\\
0 & 0 & \tau & \phi + \rho\\
0 & 0 & 0 & 0\end{pmatrix} 
& 
p-w\equiv 1\bmod 4, \begin{pmatrix}
\phi & 0 & \rho^{3} \tau^{-2} & \rho^{4} \tau^{-3}\\
0 & 0 & \rho^{2} \tau^{-1} & 0\\
0 & 0 & \phi + \rho & 0\\
0 & 0 & \tau & 0\\
0 & 0 & 0 & \tau\end{pmatrix}\\
\\ \hline
\\
p-w\equiv 2\bmod 4, \begin{pmatrix}
0 & 0 & 0 & \rho^{4} \tau^{-3}\\
\tau & \phi + \rho & \rho^{2} \tau^{-1} & \rho^{3} \tau^{-2}\\
0 & 0 & 0 & 0\\
0 & 0 & 0 & \phi\\
0 & 0 & 0 & \tau\end{pmatrix} 
& 
p-w\equiv 3\bmod 4, \begin{pmatrix}
\phi + \rho & \rho^{2} \tau^{-1} & \rho^{3} \tau^{-2} & \rho^{4} \tau^{-3}\\
\tau & 0 & \rho^{2} \tau^{-1} & 0\\
0 & \tau & \phi & \rho^{2} \tau^{-1}\\
0 & 0 & 0 & 0\\
0 & 0 & 0 & 0\end{pmatrix}
\end{tabular}
\end{center}
\caption{The matrix for $d^{1}_{p,q,w}$, $N=1$}
\label{tbl:KW2-diffs}
\end{table}
  
Using \aref{fig:KW2outline1} and \aref{tbl:KW2-diffs} it is straightforward to determine the kernel, image, and homology in each column $E^{1}_{p,*,w}(\KW/2)$. 
We summarize the calculations in \aref{tbl:KW2-answer}.
As an example, 
we identify the kernel of multiplication by $(d^{1}_{p,q,0}(i,j))_{i,j}$ when $p\equiv 1 \bmod 4$.
In each column corresponding to $h^{i,q}$ where $q\equiv i \bmod 4$ we obtain $\ker(\phi_{i,q})$ if $i=q$, 
and if $i\neq q$ and for $x \in h^{i,q}$ the element $(\rho^4\tau^{-4},\phi\tau^{-1},1)x$ is in the kernel.

\begin{table}[h!]
\begin{center}
\begin{tabular}{>{$} c <{$} | >{$} l <{$} }
p-w \bmod 4 & \ker (d^{1}_{p,q,w}) \\
\hline
0 & \oplus_{q'-i \equiv 0} h^{i,q'}\oplus \oplus_{q'-i \equiv 3} (\rho^{2}\tau^{-2},\phi + \rho,1)h^{i,q'} \\
1 & \ker (\phi_{q',q'}) \oplus \oplus_{q'-i \equiv 1} h^{i,q'}\oplus \oplus_{q'- i \equiv 0, i < q'} (\rho^4\tau^{-4},\phi\tau^{-1},1)h^{i,q'} \\
2 & \oplus_{q'-i \equiv 1} ((\phi + \rho)\tau^{-1},1)h^{i,q'}\oplus \oplus_{q'-i \equiv 2} (\rho^{2}\tau^{-2},1)h^{i,q'} \\
3 & \oplus_{q'-i \equiv 2} (\rho^{2}\tau^{-2},\phi\tau^{-1},1)h^{i,q'}\oplus \oplus_{q'-i \equiv 3} (\rho^{2}\tau^{-2},1)h^{i,q'} \\
\hline
& \im(d^{1}_{p+1,q-1,w}) \\
\hline
0 & \im(\phi_{q'-1,q'-1})\oplus \oplus_{q'-1-i \equiv 2} (\rho^3\tau^{-2},\rho^{2}\tau^{-1},\phi + \rho,\tau) h^{i,q'-1}\oplus \oplus_{q'-1-i \equiv 3} (\rho^4\tau^{-3},\tau) h^{i,q'-1} \\
1 & \oplus_{q'-1-i \equiv 0} \tau h^{i,q'-1}\oplus \oplus_{q'-1-i \equiv 3} (\rho^4\tau^{-3},\phi,\tau)h^{i,q'-1} \\
2 & \oplus_{q'-1-i \equiv 0} (\phi + \rho,\tau)\tau h^{i,q'-1}\oplus \oplus_{q'-1-i \equiv 1} (\rho^{2}\tau^{-1},\tau) h^{i,q'-1} \\
3 & \oplus_{q'-1-i \equiv 1} (\rho^{2}\tau^{-1},\phi,\tau) h^{i,q'-1}\oplus \oplus_{q'-1-i \equiv 2} (\rho^{2}\tau^{-1},\tau) h^{i,q'-1} \\
\end{tabular}
\end{center}
\caption{The homology of the column $E^{1}_{p,\ast,w}(\KW/2)$}
\label{tbl:KW2-answer}
\end{table}
\end{proof}

\begin{lemma}
\label{lem:phi=rho}
Over a field $F$ of characteristic $\Char(F)\neq 2$ we have $\phi=\rho\in h^{1,1}$.
\end{lemma}
\begin{proof}
Over the real numbers, 
$W(\R)\cong\Z$ and $I^{q}=(2^{q})$ via the index.
By \eqref{equation:sesE00} and \eqref{equation:KW2-filt1} this implies 
\begin{equation}
E^{\infty}_{0,q,0}(\KW/2)
\cong
\begin{cases}
\Z/2 & q=0 \\
0 & q>0.
\end{cases}
\end{equation}
Moreover, 
by the proof of \aref{thm:KW2-E2}, 
we have 
\begin{equation}
\label{equation:E2page}
E^{2}_{p,q,0}(\KW/2)
\cong 
\begin{cases}
h^{q,q}/\phi h^{q-1,q-1} & p \equiv 0\bmod 4 \\
\ker(\phi\colon h^{q,q} \to h^{q+1,q+1}) & p \equiv 1\bmod 4 \\
0 & \mathrm{otherwise.}
\end{cases}
\end{equation}
Here $h^{*,*}\cong\FF_{2}[\rho,\tau]$, 
where $\vert\tau\vert=(0,1)$, $\vert\rho\vert=(1,1)$, 
see \aref{tbl:mot}.
It follows that $\phi=\rho\in h^{1,1}$ over $\R$.
Over $\Q$ we compare with the completions $\R$ and $\Q_{\ell}$, 
$\ell$ a prime number, 
via the injective map
\begin{equation}
\label{equation:injective}
h^{1,1}(\Q) \to h^{1,1}(\R)\oplus \bigoplus_{\ell}h^{1,1}(\Q_{\ell}).
\end{equation}
According to \cite[Theorems 2.2, 2.29]{Lam} and \cite[Theorem 6.6]{Scharlau} the Witt-group and the nonzero powers of the fundamental ideal of the $\ell$-adic completion of $\Q$ are given by:
\begin{center}
\begin{tabular}{>{$} l <{$} | >{$} l <{$}   | >{$} l <{$}  | >{$} l <{$} }
\Q_{\ell} & W(\Q_{\ell}) & I  & I^{2} \\ \hline
\ell=2 & \Z/8\oplus(\Z/2)^{2} & \Z/4\oplus(\Z/2)^{2} &  \Z/2 \\
\ell \equiv 1 \bmod 4 & (\Z/2)^4 & (\Z/2)^{3} & \Z/2 \\
\ell \equiv 3 \bmod 4 & (\Z/4)^{2} & \Z/4\oplus\Z/2 & \Z/2
\end{tabular}
\captionof{table}{The Witt-group of $\Q_{\ell}$ and nonzero powers of the fundamental ideal}
\label{tbl:Witt}
\label{tbl:wittQ}
\end{center}
From \eqref{equation:injective} and the calculation over $\R$ it follows that $\phi=\rho+\rho'$ over $\Q$.  
Over the $\ell$-adic completions we identify the groups in \eqref{equation:E2page} with the corresponding mod $2$ Milnor $K$-groups given in \cite[Example 1.7]{Milnor}, \cite[\S3.3]{MR3073932}:
\begin{center}
\begin{tabular}{>{$} l <{$} | >{$} l <{$} }
\Q_{\ell} & k^{\Mil}_{*}  \\ \hline
\ell=2 & \FF_{2}[\rho,x,y]/(x^{2},y^{2},\rho^{2}+xy,\rho x, \rho y) \\
\ell \equiv 1 \bmod 4 & \FF_{2}[u,\ell]/(u^{2},\ell^{2})  \\
\ell \equiv 3 \bmod 4 & \FF_{2}[\rho,\ell]/(\rho^{2},\ell(\rho-\ell))
\end{tabular}
\captionof{table}{The mod $2$ Milnor $K$-groups of $\Q_{\ell}$}
\label{tbl:MilnorK}
\end{center}
Here $u$ is a nonsquare in the Teichm{\"u}ller lift $\FF_{\ell}^{\times}\subseteq\Q_{\ell}^{\times}$.
For $\ell=2$ we write $x$ and $y$ for the square classes of $2$ and $5$, 
respectively.

If $\ell \equiv 1 \bmod 4$ we have $\rho=0$ over $\Q_{\ell}$. 
By \eqref{equation:KW2-filt1} we obtain $E^{\infty}_{0,1,0}(\KW/2)\cong (\Z/2)^{2}$. 
Hence the image of $\rho'\colon h^{0,0}\to h^{1,1}$ is trivial, 
i.e., 
$\rho'=0$ over $\Q_{\ell}$ when $\ell \equiv 1 \bmod 4$.
The same analysis is inconclusive over $\Q_{2}$ and $\Q_{\ell}$ for $\ell \equiv 3 \bmod 4$ in the sense that we are unable to determine the value of $\rho'$ by comparing with 
\eqref{equation:sesE00}, \eqref{equation:KW2-filt1}, and \eqref{equation:KW2-filt2}.
The filtration of the abutment is the same regardless of the value of $\rho'$ 
(that is, up to isomorphism of each filtration quotient).

Next we show $\rho'=0$ over $\Q$.
Over $\Q_{\ell}$ the slice spectral sequence for $\KQ/2$ converges strongly to $\pi_{*,*}(\KQ/2)$ by \aref{theorem:best-conv}.
The slice $\dd^{1}$-differentials are identified in terms of those of $\KW/2$ in \aref{thm:diff-kq-mod2}, 
with the caveat that one should replace $\rho$ by $\rho+\rho'$ (before concluding $\rho'=0$).
Its $E^{2}$-page with $\rho' = 0$ is identified in \aref{thm:KQ2-E2}.
To show $\rho' = 0$ we consider weight zero and $E^{2}_{2,\ast,0}(\KQ/2)$.
The proof of \aref{thm:KQ2-E2} shows $E^{2}_{2,i,0}(\KQ/2)=E^{2}_{2,i,0}(\KW/2)$ if $i \geq 4$,
and $E^{2}_{2,i,0}(\KQ/2)=0$ if $i\leq 0$.
We have $E^{2}_{2,i,0}(\KQ/2) \neq 0$ only if $i=1,2,3$ by \aref{thm:KW2-E2}.
Calculations as in \aref{thm:KQ2-E2} show
\[
E^{2}_{2,0,0}(\KQ/2) \cong 0,\;
E^{2}_{2,1,0}(\KQ/2) \cong \ker(\rho_{0,1}'),\;
E^{2}_{2,2,0}(\KQ/2) \cong h^{1,2} \oplus h^{0,2},\;
E^{2}_{2,3,0}(\KQ/2) \cong  h^{2,3}/(\rho\rho').
\]
If $\rho' \neq 0$ this filtration is too small to produce $\KQ_{2}(\Q_{\ell};\Z/2)$,
see \aref{lem:KQ2}.
Thus we conclude $\rho'=0$.

Via \eqref{equation:injective} we conclude the claim for $\Q$ and hence for any field of characteristic zero using base change.
For fields of positive odd characteristic one may proceed as in the proof of \cite[Lemma 5.1]{slices}.
\end{proof}

Next we calculate the $E^2$-page of the slice spectral sequence for $\KQ/2$ over fields and rings of $\mathcal{S}$-integers in number fields.
Specializing to fields with $\vcd(F) \leq 2$ or $\OFS$ the spectral sequence collapses, and we deduce a calculation of the mod $2$ hermitian $K$-groups up to extensions.

The slices \eqref{equation:hermitianktheoryslices} of $\KQ$ were identified in \cite[Theorem 5.18]{slices}.
It follows there is an isomorphism 
\begin{equation}
\s_{q}(\KQ/2) 
\simeq
\bigvee_{j \leq q}\Sigma^{q+j,q}\MZ/2.
\label{equation:KQ2-slices}
\end{equation}
\begin{remark}
We shall refer to the summands $\Sigma^{q+j,q}\MZ/2$ of $\s_{q}(\KQ/2)$ as even or odd when $j$ is even respectively odd.
These are summands arising from $\s_{q}(\KQ)$ or $\Sigma^{1,0}\s_{q}(\KQ)$ in the cofiber sequence defining $\s_{q}(\KQ/2)$.
As it turns out it is often easier to compute with the even summands.
\end{remark}

\begin{theorem}
\label{thm:diff-kq-mod2}
The restriction of the slice $\dd^{1}$-differential to the summand $\Sigma^{q+j,q}\MZ/2$ of $\s_{q}(\KQ/2)$ in \eqref{equation:KQ2-slices} is given by
\begin{align*}
\dd^{1}(\KQ/2)(q,j) 
& =  
\begin{cases}
(\Sq^{3}\Sq^{1},0,\Sq^{2},0,0) & q>j\equiv 0\bmod 4 \\
(\Sq^{3}\Sq^{1},\Sq^{2}\Sq^{1}+\Sq^{3},\Sq^{2},\rho+\tau\Sq^{1},0) & q>j\equiv 1\bmod 4 \\
(\Sq^{3}\Sq^{1},0,\Sq^{2}+\rho\Sq^{1},0,\tau) & q>j\equiv 2\bmod 4 \\
(\Sq^{3}\Sq^{1},\Sq^{2}\Sq^{1}+\Sq^{3},\Sq^{2}+\rho\Sq^{1},\tau\Sq^{1},\tau) & q>j\equiv 3\bmod 4, 
\end{cases} \\
\dd^{1}(\KQ/2)(q,q) & =  
\begin{cases}
(0,\Sq^{2}\Sq^{1},\Sq^{2}+\rho\Sq^{1},0,0) & q\equiv 0\bmod 4 \\
(0,\Sq^{2}\Sq^{1}+\Sq^{3},\Sq^{2},\rho+\tau\Sq^{1},0) & q\equiv 1\bmod 4 \\
(0,\Sq^{2}\Sq^{1},\Sq^{2}+\rho\Sq^{1},\tau\Sq^{1},\tau) & q\equiv 2\bmod 4 \\
(0,\Sq^{2}\Sq^{1}+\Sq^{3},\Sq^{2}+\rho\Sq^{1},\tau\Sq^{1},\tau) & q\equiv 3\bmod 4. 
\end{cases} 
\end{align*}
Here the $i$th component of $\dd^{1}(\KQ/2)(q,j)$ is a map $\Sigma^{q+j,q}\MZ/2\to\Sigma^{q+j+i,q+1}\MZ/2$. 
\end{theorem}
\begin{proof}
This follows from \aref{thm:KW2-diff} by applying \aref{convention:kq-to-kt-slices} to the commutative diagram
\begin{equation}
\label{equation:d1KQ2KW2}
\begin{tikzcd}
\s_{q}(\KQ/2) \ar[r] \ar[d, "\dd^{1}(\KQ/2)"] & \s_{q}(\KW/2) \ar[d, "\dd^{1}(\KW/2)"] \\
\Sigma^{1,0} \s_{q+1}(\KQ/2) \ar[r] & \Sigma^{1,0}\s_{q+1}(\KW/2).
\end{tikzcd}
\end{equation}

Suppose $q$ is even.
When $q<j$ it is immediate that $\dd^{1}(\KQ/2)(q,j)=\dd^{1}(\KW/2)(q,j)$.
When $q=j$ the top horizontal map in \eqref{equation:d1KQ2KW2} equals $(\Sq^{1},\id)$, while the lower horizontal map is an inclusion.
Hence we obtain $\dd^{1}(\KQ/2)(q,q)=\dd^{1}(\KW/2)(q,q)+\dd^{1}(\KW/2)(q,q+1)\Sq^{1}$;
e.g., 
if $q \equiv 0\bmod 4$,
\begin{align*}
(0,\Sq^{3}\Sq^{1},0,\Sq^{2},0,0) &+ 
(\Sq^{3}\Sq^{1},\Sq^{2}\Sq^{1}+\Sq^{3},\Sq^{2},\rho+\tau\Sq^{1},0,0)\Sq^{1} \\
& = 
(0,0,\Sq^{2}\Sq^{1},\Sq^{2} + \rho\Sq^{1},0,0).
\end{align*}

Suppose $q$ is odd.
When $j<q-1$ it is immediate that $\dd^{1}(\KQ/2)(q,j)=\dd^{1}(\KW/2)(q,j)$.
Note that $\dd^{1}(\KQ/2)(q,q-1)$ takes the asserted form since $\Sq^3\Sq^{1}\Sq^{1}=0$.
If $j=q$, 
then $\Sigma^{2q+2,q+1}\MZ/2$ maps by the identity under the lower horizontal map in \eqref{equation:d1KQ2KW2}.
Thus $\dd^{1}(\KQ/2)(q,q)$ agrees with $\dd^{1}(\KW/2)(q,q)$ except for on the summand $\Sigma^{2q+3,q+1}\MZ/2$ of $\s_{q+1}(\KW/2)$ (this is not a summand of $\s_{q+1}(\KQ/2)$).
Finally, 
note that $\s_{q}(\KQ/2)\to\s_{q}(\KW/2)$ is a split monomorphism.
\end{proof}

\begin{theorem}
\label{thm:KQ2-E2}
The groups $E^{2}_{p,q,w}(\KQ/2)$ over fields $F$ of characteristic different than $2$ and rings of $\mathcal{S}$-integers in number fields are given by \aref{tbl:KQ2-answer1} and \aref{tbl:KQ2-answer4}.
(In the tables $p-w$ and $q-p+w$ are congruence classes modulo $4$ with the exception of $q-p+w$ in \aref{tbl:KQ2-answer4}, and $a=2q-p$, $q'=q-w$):

\begin{center}
\captionof{table}{The group $E^{2}_{p,q,w}(\KQ/2)$ is trivial if $p/2>q$ and for $q+w \leq p$ it is given by:}

\begin{tabular}{>{$}c <{$}| >{$}c <{$}| >{$}c <{$}}
$\backslashbox{$p - w$}{$q - p + w$}$ & 0 & 1\\
\hline 0 & h^{a, q'}/\rho^5 & \ker(\rho^2_{a, q'}) \oplus h^{a - 1, q'}/\rho^2\\
1 & \ker(\rho_{a, q'}) \oplus h^{a - 1, q'}/\rho^2 & \ker(\rho^2_{a - 1, q'}) \oplus h^{a - 2, q'}/\rho^2\\
2 & h^{a - 1, q'}/\rho^3 \oplus h^{a - 2, q'}/\rho & h^{a - 2, q'}/\rho^3\\
3 & h^{a - 2, q'}/\rho^2 & \ker(\rho^3_{a, q'})\\
\hline  & 2 & 3\\
\hline 0 & \ker(\rho^2_{a - 1, q'}) \oplus h^{a - 2, q'}/\rho & h^{a - 2, q'}/\rho^2\\
1 & \ker(\rho^5_{a - 2, q'}) & \ker(\rho^2_{a, q'})\\
2 & \ker(\rho^2_{a, q'}) & h^{a, q'}/\rho^3 \oplus \ker(\rho^2_{a - 1, q'})\\
3 & \ker(\rho_{a, q'}) \oplus \ker(\rho^3_{a - 1, q'}) & h^{a - 1, q'}/\rho^2 \oplus \ker(\rho^3_{a - 2, q'})
\end{tabular}

\label{tbl:KQ2-answer1}
\end{center}

\begin{center}
\captionof{table}{For $q+w \geq p+1$ the group $E^{2}_{p,q,w}(\KQ/2)$ is given by:}

\begin{tabular}{>{$}c <{$}| >{$}c <{$}| >{$}c <{$}}
$\backslashbox{$p - w$}{$q - p + w$}$ & 1 & > 1\\
\hline 0 & h^{a - 1, q'}/\rho^2 & h^{q', q'}/\rho\\
1 & \ker(\rho_{a - 1, q'}) \oplus h^{a - 2, q'}/\rho^2 & \ker(\rho_{q', q'})\\
2 & h^{a - 2, q'}/\rho^3 & 0\\
3 & 0 & 0
\end{tabular}

\label{tbl:KQ2-answer4}
\end{center}
\end{theorem}
\begin{proof}
Taking homotopy groups in \eqref{equation:KQ2-slices} yields
\[
E^1_{p,q,w}(\KQ/2)
\cong
\bigoplus_{i=0}^{\min\{2q - p, q'\}} h^{i, q - w}.
\]

If $q \geq p-w + 1$ the canonically induced map $E^1_{p,q,w}(\KQ/2) \to E^1_{p,q,w}(\KW/2)$ is an isomorphism.
Moreover, 
the entering and exiting $d^{1}$-differentials for $\KQ/2$ and $\KW/2$ coincide when $q>p-w+1$,
see \aref{thm:diff-kq-mod2}.
Thus $E^2_{p,q,w}(\KQ/2) \cong E^2_{p,q,w}(\KW/2)$ in this region.
By \aref{thm:KW2-E2} we obtain the second column in \aref{tbl:KQ2-answer4}.

If $q \leq p-w$ we proceed as in \aref{thm:KW2-E2} by writing $d^{1}_{p,q,w}$ as a matrix $(d^{1}_{p,q,w}(i,j))_{i,j}$, 
where
\[
d^{1}_{p,q,w}(i,j)\colon h^{i,q'} \to h^{j,q'+1}.
\]
To determine this matrix we combine \aref{thm:diff-kq-mod2} with \aref{tbl:Sq-action}.
Set $a= 2q-p$, $a'=a-b$, and 
\begin{align*}
b  &= (q-p \bmod 4) \in \{0,1,2,3\} \\
A  &= 2(q-1)-(p+1) = a-3 \\
B  &= ((q-1)-(p+1)\bmod 4) = (b+2 \bmod 4) \\
A' &= A-B = a-3-(b+2\bmod 4).
\end{align*}
Note that $h^{a,q'}$ is the top summand of $E^{1}_{p,q,w}(\KQ/2)$, and likewise for $h^{A',q'-1}$ and $E^{1}_{p+1,q-1,w}(\KQ/2)$.
The entry $d^{1}_{p,q,w}(i,j)$ is possibly nontrivial only when $(i,j) \in [0,b+4n)\times[0,b + 4n + 1)$, 
where $n \geq \lceil (2q-p-b)/4 \rceil$.
Thus $(d^{1}_{p,q,w}(i,j))_{i,j}$ is a $(b+4n)\times(b + 4n + 1)$-matrix.
To simplify we consider its submatrices $M^{1}_{p,q} := (d^{1}_{p,q}(i,j))_{i,j}$ for $i \in [a',a]$ given in  \aref{tbl:KQ2-diffs},
and $M^{2}_{p,q,w} := (d^{1}_{p,q,w}(i,j))_{i,j}$ for $i \in [0,a')$.
Here $M^{1}_{p,q,w}$ is a $(b+1)\times (b+4n+1)$-matrix, cf.~\aref{tbl:KQ2-diffs}, 
and $M^{2}_{p,q,w}$ is the block matrix obtained by inserting one of the matrices in \aref{tbl:KW2-diffs} along its diagonal and zero entries elsewhere.
In the first block of $M^{2}_{p,q}$ we remove the first column in the corresponding matrix from \aref{tbl:KW2-diffs}.
That is, $M^2_{p,q,w}(i, j) = M_{p-w \bmod 4}(i \bmod 4,j+1 \bmod 4)$ if $\vert i - (j+1) \vert \leq 4$ and $0$ otherwise.
Here $M_{p-w}$ is the $(p-w \bmod 4)$th matrix in \aref{tbl:KW2-diffs}.

It is helpful to note the equality $d^1_{p,q,0} = d^1_{p+w,q,w}$.
Indeed,
we have
\[
d^1_{p+w,q,w}:\pi_{p+w,w}(\bigvee_{j\leq q}\Sigma^{q+j,q}\MZ/2) \to E^1_{p-1,q+1,w},
\]
and
\[
\pi_{p+w,w}(\bigvee_{j\leq q}\Sigma^{q+j,q}\MZ/2)
= 
\bigoplus_{j \leq q} h^{q + j - p -w, q - w}
= 
\bigoplus_{j \leq q} \tau^{j - p} h^{q + j - p-w, q + j - p-w}.
\]
Here the $d^1$-differential depends only on the powers of $\tau$, and the integers $q$ and $j$.
The decompositions 
$$
\ker(d^{1}_{p,q,w})\cong\ker(M^{1}_{p,q,w})\oplus\ker(M^{2}_{p,q,w}),
$$ and 
$$
\im(d^{1}_{p,q,w})\cong\im(M^{1}_{p,q,w})\oplus\im(M^{2}_{p,q,w})
$$ 
follow by inspection.
To determine the $E^{2}$-page we identify the kernels, images and homologies for all the matrices $M^{1}_{p,q,w}$ in \aref{tbl:KQ2-diffs}, 
and likewise for $M^{2}_{p,q,w}$. 
A part of the calculation for $M^{2}_{p,q,w}$ was carried out in \aref{thm:KW2-E2}, see \aref{tbl:KW2-answer}.
The kernels and images together with the values of $a'$ and $A'$ are determined by \aref{tbl:KQ2-kerim}.
Using this data we deduce \aref{tbl:KQ2-answer1}.
If $q=p-w+1$, 
the entering $d^{1}$-differential is given by \aref{tbl:KQ2-diffs} and the exiting $d^{1}$-differential by \aref{tbl:KW2-diffs}.
By combining \aref{tbl:KW2-answer} (or \aref{tbl:KQ2-kerim} for $q-p+w\equiv 0\bmod 4$) with \aref{tbl:KQ2-kerim} for $q-p+w\equiv 1\bmod 4$, 
we deduce the first column in \aref{tbl:KQ2-answer4}.
The $E^2$-page in weight $w=0, 1, 2, 3$ is shown in \aref{fig:KQ2-E2}, \aref{fig:KQE2_{1}}, \aref{fig:KQE2_{2}}, and \aref{fig:KQE2_3}, respectively.
\end{proof}

Combining the ring structure on $\KQ$ in \cite{PW} with \eqref{equation:R-pairing} we obtain a pairing of spectral sequences
\begin{equation}
E^{r}_{p,q,w}(\KQ/4) \otimes E^{r}_{p',q',w'}(\KQ/2) \to E^{r}_{p+p',q+q',w+w'}(\KQ/2).
\label{equation:KQ2-pairing}
\end{equation}
The proof of \aref{lem:KGL2-E284} shows the group $E^{r}_{8,4,0}(\KQ/4)$ is isomorphic to $H^{0,4}(F;\Z/4)$, 
generated by $\widetilde{\tau}^4$ for all $r\geq 1$.
Here $\widetilde{\tau}^4$ commutes with the $d^{1}$-differential and it defines an $(8,4,0)$-periodicity element on the $E^{r}$-pages under the paring \eqref{equation:KQ2-pairing}.
The generator $\widetilde{\tau}^4 \in H^{0,4}(F;\Z/4)$ acts as $\tau^4 \in h^{0,4}$, 
and we obtain:
\begin{lemma}
\label{lem:KQ2-E284}
There is an isomorphism $E^{\infty}_{8,4,0}(\KQ/4) \cong H^{0,4}(F; \Z/4)$.
\end{lemma}

Recall the slice spectral sequence for $\KQ/2^{n}$ is conditionally convergent over $F$ when $\vcd(F)<\infty$ by \aref{theorem:best-conv}.

\begin{theorem}
\label{theorem:KQ2-period}
If $F$ is a field of characteristic $\Char(F)\neq 2$ and $\vcd(F)<\infty$, 
the $E^{2}$-page $E^{2}_{p,q,w}(\KQ/2)$ of the slice spectral sequence for $\KQ/2$ is $(8,4,0)$-periodic for $p-2w \geq\vcd(F)-\delta_{F}$.
Here $\delta_{F}=2$ if $\vcd(F)+w\equiv 0\bmod 4$, 
and $\delta_{F}=1$ if $\vcd(F)+w\not\equiv 0\bmod 4$.
The periodicity isomorphism is induced by multiplication by $\widetilde{\tau}^{4} \in E^{2}_{8,4,0}(\KQ/4)$ in the pairing \eqref{equation:KQ2-pairing}.

When $\vcd(F)=2$ the mod $2$ hermitian $K$-groups of $F$ are given up to extensions as follows.
\begin{center}

\begin{tabularx}{\textwidth}{>{$}l <{$}|>{$}l <{$}|>{$}l <{$}  }
\hline n \geq 0 & l & \KQ_{n,{0}}(F; \Z/2)\\
\hline 8k & 3 & \f_{0}/\f_{1} = h^{0, 4k} , \f_{1}/\f_{2} = \ker(\rho_{2, 4k + 1}) \oplus h^{1, 4k + 1} , \f_{2} = h^{2, 4k + 2}/\rho\\
8k + 1 & 2 & \f_{0}/\f_{1} = \ker(\rho_{1, 4k + 1}) \oplus h^{0, 4k + 1} , \f_{1} = \ker(\rho_{2, 4k + 2}) \oplus h^{1, 4k + 2}\\
8k + 2 & 3 & \f_{0}/\f_{1} = h^{0, 4k + 1} , \f_{1}/\f_{2} = h^{1, 4k + 2} \oplus h^{0, 4k + 2} , \f_{2} = h^{2, 4k + 3}\\
8k + 3 & 2 & \f_{0}/\f_{1} = h^{0, 4k + 2} , \f_{1} = h^{1, 4k + 3}\\
8k + 4 & 2 & \f_{0}/\f_{1} = h^{0, 4k + 3} , \f_{1} = h^{4, 4k + 4}\\
8k + 5 & 0 & 0\\
8k + 6 & 1 & \f_{0} = \ker(\rho_{2, 4k + 4})\\
8k + 7 & 2 & \f_{0}/\f_{1} = \ker(\rho^2_{1, 4k + 4}) , \f_{1} = \ker(\rho_{2, 4k + 5})\\
\hline 
\end{tabularx}

\vskip1ex

\begin{tabularx}{\textwidth}{>{$}l <{$}|>{$}l <{$}|>{$}l <{$}  }
\hline n \geq 0 & l & \KQ_{n + 2,{1}}(F; \Z/2)\\
\hline 8k & 3 & \f_{0}/\f_{1} = \ker(\rho_{0, 4k}) , \f_{1}/\f_{2} = \ker(\rho^2_{1, 4k + 1}) \oplus h^{0, 4k + 1} , \f_{2} = \ker(\rho_{2, 4k + 2})\\
8k + 1 & 2 & \f_{0}/\f_{1} = h^{0, 4k + 1} , \f_{1} = h^{1, 4k + 2}\\
8k + 2 & 1 & \f_{0} = h^{0, 4k + 2}\\
8k + 3 & 1 & \f_{0} = h^{3, 4k + 3}\\
8k + 4 & 1 & \f_{0} = \ker(\rho_{2, 4k + 3})\\
8k + 5 & 2 & \f_{0}/\f_{1} = \ker(\rho^2_{1, 4k + 3}) , \f_{1} = h^{3, 4k + 4}/\rho^3 \oplus \ker(\rho_{2, 4k + 4})\\
8k + 6 & 3 & \f_{0}/\f_{1} = \ker(\rho^3_{0, 4k + 3}) , \f_{1}/\f_{2} = \ker(\rho_{2, 4k + 4}) \oplus \ker(\rho^2_{1, 4k + 4}) , \f_{2} = \ker(\rho_{2, 4k + 5})\\
8k + 7 & 2 & \f_{0}/\f_{1} = \ker(\rho^2_{1, 4k + 4}) \oplus h^{0, 4k + 4} , \f_{1} = \ker(\rho_{2, 4k + 5}) \oplus h^{1, 4k + 5}/\rho\\
\hline 
\end{tabularx}

\vskip1ex

\begin{tabularx}{\textwidth}{>{$}l <{$}|>{$}l <{$}|>{$}l <{$}  }
\hline n \geq 0 & l & \KQ_{n + 4,{2}}(F; \Z/2)\\
\hline 8k & 1 & \f_{0} = h^{0, 4k + 1}\\
8k + 1 & 0 & 0\\
8k + 2 & 1 & \f_{0} = h^{2, 4k + 2}\\
8k + 3 & 2 & \f_{0}/\f_{1} = \ker(\rho^2_{1, 4k + 2}) , \f_{1} = h^{2, 4k + 3}/\rho^2\\
8k + 4 & 3 & \f_{0}/\f_{1} = \ker(\rho^2_{0, 4k + 2}) , \f_{1}/\f_{2} = h^{2, 4k + 3} \oplus \ker(\rho^2_{1, 4k + 3}) , \f_{2} = h^{3, 4k + 4}/\rho^3 \oplus h^{2, 4k + 4}/\rho\\
8k + 5 & 2 & \f_{0}/\f_{1} = \ker(\rho_{1, 4k + 3}) \oplus \ker(\rho^3_{0, 4k + 3}) , \f_{1} = h^{2, 4k + 4}/\rho^2 \oplus \ker(\rho^2_{1, 4k + 4})\\
8k + 6 & 3 & \f_{0}/\f_{1} = \ker(\rho^2_{0, 4k + 3}) , \f_{1}/\f_{2} = \ker(\rho^2_{1, 4k + 4}) \oplus h^{0, 4k + 4} , \f_{2} = h^{2, 4k + 5}/\rho^2\\
8k + 7 & 2 & \f_{0}/\f_{1} = \ker(\rho^2_{0, 4k + 4}) , \f_{1} = \ker(\rho^2_{1, 4k + 5})\\
\hline 
\end{tabularx}

\vskip1ex

\begin{tabularx}{\textwidth}{>{$}l <{$}|>{$}l <{$}|>{$}l <{$}  }
\hline n \geq 0 & l & \KQ_{n + 6,{3}}(F; \Z/2)\\
\hline 8k & 1 & \f_{0} = \ker(\rho_{2, 4k + 1})\\
8k + 1 & 2 & \f_{0}/\f_{1} = h^{1, 4k + 1} , \f_{1} = h^{2, 4k + 2}/\rho^2\\
8k + 2 & 3 & \f_{0}/\f_{1} = \ker(\rho^2_{0, 4k + 1}) , \f_{1}/\f_{2} = \ker(\rho_{2, 4k + 2}) \oplus h^{1, 4k + 2} , \f_{2} = h^{2, 4k + 3}/\rho^2\\
8k + 3 & 3 & \f_{0}/\f_{1} = h^{1, 4k + 2} \oplus \ker(\rho^2_{0, 4k + 2}) , \f_{1}/\f_{2} = h^{2, 4k + 3} \oplus h^{1, 4k + 3}/\rho , \f_{2} = h^{3, 4k + 4}/\rho^3\\
8k + 4 & 3 & \f_{0}/\f_{1} = \ker(\rho_{0, 4k + 2}) , \f_{1}/\f_{2} = h^{1, 4k + 3} \oplus \ker(\rho^3_{0, 4k + 3}) , \f_{2} = h^{2, 4k + 4}/\rho^2\\
8k + 5 & 3 & \f_{0}/\f_{1} = \ker(\rho^2_{0, 4k + 3}) , \f_{1}/\f_{2} = h^{1, 4k + 4} , \f_{2} = h^{5, 4k + 5}/\rho^5\\
8k + 6 & 1 & \f_{0} = \ker(\rho^3_{0, 4k + 4})\\
8k + 7 & 0 & 0\\
\hline 
\end{tabularx}

\end{center}
\end{theorem}
\begin{proof}
We make the following observations:
\begin{itemize}
\item If $q + w \leq p$, the group $E^{2}_{p,q,w}(\KQ/2)$ is identified in \aref{tbl:KQ2-answer1}.
\item The direct summands of $E^{2}_{p,q,w}(\KQ/2)$ are subquotients of $h^{2q-p-i,q-w}$, for $0\leq i\leq 2$. 
Such a subquotient is trivial if $2q-p-2 > \vcd(F)$.
\end{itemize}
Now assume $p - 2w >\vcd(F)$.
\begin{itemize}
\item If $q \leq\frac{1}{2}(p+\vcd(F)+2)$, then $q + w \leq p$, and $(8,4,0)$-periodicity follows.
\item
If $q >\frac{1}{2}(p +\vcd(F)+2)$, then $2q-p>\vcd(F)+2$, and hence $E^{2}_{p,q,w}(\KQ/2)=0$.
\end{itemize}
It remains to consider degrees with $p = \vcd(F) + 2w - \delta$, $\delta = 0, 1$, 
i.e., 
compare $E^2_{p,q,w}$ for $q \in (\vcd(F) + w - \delta, \vcd + w + 1 - (\delta/2)]$ with $E^2_{p+8,q+4,w}$.
If $q$ is not in this interval, 
$E^2_{p,q,w}$ is either zero or determined by the $(8,4,0)$-periodic \aref{tbl:KQ2-answer1}, as observed above.
It remains to show the $E^{2}$-page in degrees $(p,q, w)\in\{(\vcd(F) + 2w,\vcd(F)+w+1,w),(\vcd(F)+2w-1,\vcd(F)+w,w)\}$ and $(p+8,q+4,w)$ are isomorphic via the map $\tau^4$.
This follows by inspection of \aref{tbl:KQ2-answer1} and \aref{tbl:KQ2-answer4} in \aref{thm:KQ2-E2}.
The sharper bound for $\vcd(F) + w\equiv 0\bmod 4$ follows by inspection of the degrees $(\vcd(F)+2w-2,\vcd(F)+w-1, w)$ and $(\vcd(F)+2w-2,\vcd(F)+w, w)$.
\end{proof}

\begin{corollary}
\label{corollary:KQfields}
Suppose $F$ is a field of $\Char(F)\neq 2$ and $\vcd(F)<\infty$.
The permanent cycle $\widetilde{\tau}^{4}$ induces an 8-fold periodicity isomorphism 
\begin{equation*}
{\KQ_{p,w}(F;\Z/2)}
\cong
{\KQ_{p+8,w}(F;\Z/2)}
\end{equation*}
for all $p-2w\geq\vcd(F)-1$.
\end{corollary}
\begin{proof}
Multiplication by $\widetilde{\tau}^4$ commutes with the differentials, so the periodicity in \aref{theorem:KQ2-period} carries over to the $E^{\infty}$-page.
We conclude by reference to \aref{theorem:best-conv}.
\end{proof}
 
\begin{example}
\label{ex:KQ2-C}
Let $\overline{F}$ be an algebraically closed field, or more generally a quadratically closed field, of $\Char(\overline{F})\neq 2$.
Then $h^{*,*}=\FF_{2}[\tau]$ and \aref{thm:KQ2-E2} implies isomorphisms for $2w \leq p$
\[
E^{\infty}_{p,q,w}(\KQ/2)
\cong
\begin{cases}
  h^{0, q - w} & 2q - p = 0, 2 \text{ and } p \equiv 0,2 \bmod 8 \\
  h^{0, q - w} & 2q - p = 1 \text{ and } p \equiv 1,3 \bmod 8 \\
  h^{0, q - w} & 2q - p = 2 \text{ and } p \equiv 2,4 \bmod 8 \\
  0 & \text{otherwise}.
\end{cases}
\]
For $2w > p$ we have isomorphisms 
\[
E^{\infty}_{p,q,w}(\KQ/2)
\cong
\begin{cases}
  h^{0, 0} & q = w \text{ and } p - w \equiv 0,1 \bmod 4 \\
  0 & \text{otherwise}.
\end{cases}
\]
Over $\C$,
this gives the $8$-periodicity $\Z/2$, $\Z/2$, $\Z/4$, $\Z/2$, $\Z/2$, $0$, $0$, $0$ for the mod 2 $K$-groups of the real numbers, 
see e.g., 
\cite[Theorem 4.9]{MR772065}.
\end{example}

\begin{example}
\label{ex:KQ2-R}
If $F$ is a real closed field, 
then $h^{*,*}=\FF_{2}[\tau,\rho]$. 
In the following table we use the notation in \aref{theorem:convergence2} and determine the filtration quotients for the group $\KQ_{8k+2,0}(F;\Z/2)$.
\begin{center}

\begin{tabularx}{\textwidth}{>{$}l <{$} | >{$}l <{$}| >{$}l <{$}| >{$}l <{$}| >{$}l <{$} }
\hline 
n \geq 0 & \KQ_{n,0}(F;\Z/2) &  \KQ_{n+2,1}(F;\Z/2) &  \KQ_{n+4,2}(F;\Z/2) &  \KQ_{n+6,3}(F;\Z/2)\\
\hline 
8k & h^{1,4k+1}\bullet h^{0,4k} & h^{0,4k+1} & h^{0,4k+1} & 0\\
8k+1 & h^{1,4k+2}\bullet h^{0,4k+1} & h^{1,4k+2} \bullet h^{0,4k+1} & 0 & h^{1,4k+1}\\
8k+2 & h^{0,4k+1}, h^{1,4k+2} \oplus h^{0,4k+2} , h^{2,4k+3} & h^{0,4k+2} & h^{2,4k+2} & h^{1,4k+2}\\
8k+3 & h^{1,4k+3}\bullet h^{0,4k+2} & h^{3,4k+3} & 0 & h^{2,4k+3}\bullet h^{1,4k+2}\\
8k+4 & h^{4,4k+4}\bullet h^{0,4k+3} & 0 & h^{2,4k+3} & h^{1,4k+3}\\
8k+5 & 0 & 0 & 0 & h^{1,4k+4}\\
8k+6 & 0 & 0 & h^{0,4k+4} & 0\\
8k+7 & 0 & h^{0,4k+4} & 0 & 0\\
\hline 
\end{tabularx}

\end{center}
\end{example}

We are ready to prove \aref{theorem:mod2hermitiankgroupsintroduction} stated in the introduction.

\begin{theorem}
\label{theorem:mod2hermitiankgroups}
The mod $2$ hermitian $K$-groups of $\OO_{F,\mathcal{S}}$ are computed up to extensions by the following filtrations of length $l$. 
\begin{center}

\begin{tabularx}{\textwidth}{>{$}l <{$}|>{$}l <{$}|>{$}l <{$}  }
\hline n \geq 0 & l & \KQ_{n,{0}}(\mathcal{O}_{F,S}; \Z/2)\\
\hline 8k & 3 & \f_{0}/\f_{1} = h^{0, 4k} , \f_{1}/\f_{2} = \ker(\rho_{2, 4k + 1}) \oplus h^{1, 4k + 1} , \f_{2} = h^{2, 4k + 2}/\rho\\
8k + 1 & 2 & \f_{0}/\f_{1} = \ker(\rho_{1, 4k + 1}) \oplus h^{0, 4k + 1} , \f_{1} = \ker(\rho_{2, 4k + 2}) \oplus h^{1, 4k + 2}\\
8k + 2 & 3 & \f_{0}/\f_{1} = h^{0, 4k + 1} , \f_{1}/\f_{2} = h^{1, 4k + 2} \oplus h^{0, 4k + 2} , \f_{2} = h^{2, 4k + 3}\\
8k + 3 & 2 & \f_{0}/\f_{1} = h^{0, 4k + 2} , \f_{1} = h^{1, 4k + 3}\\
8k + 4 & 2 & \f_{0}/\f_{1} = h^{0, 4k + 3} , \f_{1} = h^{4, 4k + 4}\\
8k + 5 & 0 & 0\\
8k + 6 & 2 & \f_{0}/\f_{1} = \ker(\rho_{2, 4k + 4}) , \f_{1} = h^{4, 4k + 5}/\rho^3\\
8k + 7 & 2 & \f_{0}/\f_{1} = \ker(\rho^2_{1, 4k + 4}) , \f_{1} = \ker(\rho_{2, 4k + 5})\\
\hline 
\end{tabularx}

\vskip1ex

\begin{tabularx}{\textwidth}{>{$}l <{$}|>{$}l <{$}|>{$}l <{$}  }
\hline n \geq 0 & l & \KQ_{n + 2,{1}}(\mathcal{O}_{F,S}; \Z/2)\\
\hline 8k & 3 & \f_{0}/\f_{1} = \ker(\rho_{0, 4k}) , \f_{1}/\f_{2} = \ker(\rho^2_{1, 4k + 1}) \oplus h^{0, 4k + 1} , \f_{2} = \ker(\rho_{2, 4k + 2})\\
8k + 1 & 2 & \f_{0}/\f_{1} = h^{0, 4k + 1} , \f_{1} = h^{1, 4k + 2}\\
8k + 2 & 1 & \f_{0} = h^{0, 4k + 2}\\
8k + 3 & 1 & \f_{0} = h^{3, 4k + 3}\\
8k + 4 & 2 & \f_{0}/\f_{1} = \ker(\rho_{2, 4k + 3}) , \f_{1} = h^{3, 4k + 4}/\rho^2\\
8k + 5 & 3 & \f_{0}/\f_{1} = \ker(\rho^2_{1, 4k + 3}) , \f_{1}/\f_{2} = h^{3, 4k + 4}/\rho^3 \oplus \ker(\rho_{2, 4k + 4}) , \f_{2} = h^{4, 4k + 5}/\rho^3\\
8k + 6 & 3 & \f_{0}/\f_{1} = \ker(\rho^3_{0, 4k + 3}) , \f_{1}/\f_{2} = \ker(\rho_{2, 4k + 4}) \oplus \ker(\rho^2_{1, 4k + 4}) , \f_{2} = h^{3, 4k + 5}/\rho^2 \oplus \ker(\rho_{2, 4k + 5})\\
8k + 7 & 3 & \f_{0}/\f_{1} = \ker(\rho^2_{1, 4k + 4}) \oplus h^{0, 4k + 4} , \f_{1}/\f_{2} = \ker(\rho_{2, 4k + 5}) \oplus h^{1, 4k + 5}/\rho , \f_{2} = h^{3, 4k + 6}/\rho^2\\
\hline 
\end{tabularx}

\vskip1ex

\begin{tabularx}{\textwidth}{>{$}l <{$}|>{$}l <{$}|>{$}l <{$}  }
\hline n \geq 0 & l & \KQ_{n + 4,{2}}(\mathcal{O}_{F,S}; \Z/2)\\
\hline 8k & 1 & \f_{0} = h^{0, 4k + 1}\\
8k + 1 & 0 & 0\\
8k + 2 & 2 & \f_{0}/\f_{1} = h^{2, 4k + 2} , \f_{1} = h^{3, 4k + 3}/\rho^2\\
8k + 3 & 3 & \f_{0}/\f_{1} = \ker(\rho^2_{1, 4k + 2}) , \f_{1}/\f_{2} = h^{2, 4k + 3}/\rho^2 , \f_{2} = h^{3, 4k + 4}/\rho^2\\
8k + 4 & 4 & \f_{0}/\f_{1} = \ker(\rho^2_{0, 4k + 2}) , \f_{1}/\f_{2} = h^{2, 4k + 3} \oplus \ker(\rho^2_{1, 4k + 3}) , \f_{2}/\f_{3} = h^{3, 4k + 4}/\rho^3 \oplus h^{2, 4k + 4}/\rho , \f_{3} = h^{4, 4k + 5}/\rho^3\\
8k + 5 & 3 & \f_{0}/\f_{1} = \ker(\rho_{1, 4k + 3}) \oplus \ker(\rho^3_{0, 4k + 3}) , \f_{1}/\f_{2} = h^{2, 4k + 4}/\rho^2 \oplus \ker(\rho^2_{1, 4k + 4}) , \f_{2} = h^{3, 4k + 5}/\rho^2\\
8k + 6 & 4 & \f_{0}/\f_{1} = \ker(\rho^2_{0, 4k + 3}) , \f_{1}/\f_{2} = \ker(\rho^2_{1, 4k + 4}) \oplus h^{0, 4k + 4} , \f_{2}/\f_{3} = h^{2, 4k + 5}/\rho^2 , \f_{3} = h^{6, 4k + 6}/\rho^5\\
8k + 7 & 2 & \f_{0}/\f_{1} = \ker(\rho^2_{0, 4k + 4}) , \f_{1} = \ker(\rho^2_{1, 4k + 5})\\
\hline 
\end{tabularx}

\vskip1ex

\begin{tabularx}{\textwidth}{>{$}l <{$}|>{$}l <{$}|>{$}l <{$}  }
\hline n \geq 0 & l & \KQ_{n + 6,{3}}(\mathcal{O}_{F,S}; \Z/2)\\
\hline 8k & 1 & \f_{0} = \ker(\rho_{2, 4k + 1})\\
8k + 1 & 2 & \f_{0}/\f_{1} = h^{1, 4k + 1} , \f_{1} = h^{2, 4k + 2}/\rho^2\\
8k + 2 & 3 & \f_{0}/\f_{1} = \ker(\rho^2_{0, 4k + 1}) , \f_{1}/\f_{2} = \ker(\rho_{2, 4k + 2}) \oplus h^{1, 4k + 2} , \f_{2} = h^{2, 4k + 3}/\rho^2\\
8k + 3 & 3 & \f_{0}/\f_{1} = h^{1, 4k + 2} \oplus \ker(\rho^2_{0, 4k + 2}) , \f_{1}/\f_{2} = h^{2, 4k + 3} \oplus h^{1, 4k + 3}/\rho , \f_{2} = h^{3, 4k + 4}/\rho^3\\
8k + 4 & 3 & \f_{0}/\f_{1} = \ker(\rho_{0, 4k + 2}) , \f_{1}/\f_{2} = h^{1, 4k + 3} \oplus \ker(\rho^3_{0, 4k + 3}) , \f_{2} = h^{2, 4k + 4}/\rho^2\\
8k + 5 & 3 & \f_{0}/\f_{1} = \ker(\rho^2_{0, 4k + 3}) , \f_{1}/\f_{2} = h^{1, 4k + 4} , \f_{2} = h^{5, 4k + 5}/\rho^5\\
8k + 6 & 1 & \f_{0} = \ker(\rho^3_{0, 4k + 4})\\
8k + 7 & 0 & 0\\
\hline 
\end{tabularx}

\end{center}
\end{theorem}
\begin{proof} 
We have $\vcd(F)=\vcd(\OO_{F,\mathcal{S}})=2$.
The proofs of \aref{thm:KQ2-E2} and \aref{theorem:KQ2-period} apply to $\OO_{F,\mathcal{S}}$  since there are no nontrivial differentials exiting or entering $h^{2,1}\in E^{1}_{0,1,0}(\KQ/2)$.
Lemma \ref{lem:H2R-surj} shows the naturally induced map $h^{p,q}(\OO_{F,\mathcal{S}})\to \bigoplus^{r_{1}}h^{p,q}(\R)$ is surjective for $q \geq 2$.
For degree reasons $E^{\infty}(\KQ/2)=E^{2}(\KQ/2)$.
The $8$-periodicity follows as in the proof of \aref{corollary:KQfields}.
\end{proof}

Our next aim is to compute the slice $\dd^{1}$-differentials for $\KW/2^{n}$ and $\KQ/2^{n}$ when $n\geq 2$.
By \eqref{equation:wittheoryslices} the slices of $\KW/2^{n}$ are given by 
\begin{equation}
\label{equation:KW/2nslices}
\s_{q}(\KW/2^{n}) 
\simeq 
\bigvee_{j} \Sigma^{q+j,q}\MZ/2,
\end{equation}
while \eqref{equation:hermitianktheoryslices} identifies the slices of $\KQ/2^{n}$ as
\begin{equation}
\label{equation:KQ2n-slices}
\s_{q}(\KQ/2^{n}) \simeq 
\begin{cases}
\Sigma^{2q,q}\MZ/2^{n} \vee \bigvee_{j < q} \Sigma^{q+j,q}\MZ/2 & q \text{ even } \\
\bigvee_{j \leq q} \Sigma^{q+j,q}\MZ/2 & q \text{ odd. }
\end{cases}
\end{equation}

The $\MZ$-module structure on the wedge product decomposition of $\s_{q}(\E/2^{n})$ is not unique.
For our calculational purposes we may make the exact same choices as in \aref{convention:kq-to-kt-slices}.

\begin{theorem}
\label{thm:KW2n-diff}
The restriction of the slice $\dd^{1}$-differential to the summand $\Sigma^{q + j,q}\MZ/2$ of $\s_{q}(\KW/2^{n})$ in \eqref{equation:KW/2nslices} is given by 
\begin{equation}
\dd^{1}(\KW/2^{n})(q,j)
= 
\begin{cases}
(\Sq^{3}\Sq^{1},0,\Sq^{2},0,0) & j\equiv 0\bmod 4 \\
(\Sq^{3}\Sq^{1},0,\Sq^{2},0,0) & j\equiv 1\bmod 4 \\
(\Sq^{3}\Sq^{1},0,\Sq^{2}+\rho\Sq^{1},0,\tau) & j\equiv 2\bmod 4 \\
(\Sq^{3}\Sq^{1},0,\Sq^{2}+\rho\Sq^{1},0,\tau) & j\equiv 3\bmod 4. 
\end{cases}
\label{equation:KW2n-diff1}
\end{equation}
\end{theorem}
\begin{proof}
According to \cite[Theorem 6.3]{slices} the corresponding slice $\dd^{1}$-differential of $\KW$ is given by 
\[
\dd^{1}(\KW)(q,j)
= 
\begin{cases}
(\Sq^3\Sq^{1},\Sq^{2},0) & j \equiv 0 \bmod 4 \\
(\Sq^3\Sq^{1},\Sq^{2}+\rho\Sq^{1},\tau) & j \equiv 2 \bmod 4.
\end{cases}
\]

When $j$ is even it follows that $\dd^{1}(\KW)(q,j)=\dd^{1}(\KW/2)(q,j)$.

When $j$ is odd, $\pr\circ \dd^{1}(\KW)(q,j) = \pr\circ \dd^{1}(\KW/2)(q,j)$, where $\pr$ is the projection of $\Sigma^{1,0}\s_{q+1}(\KW/2)$ to the odd summands.
Hence, by \aref{lem:steenrod-alg} and the vanishing $\dd^{1}\circ\dd^{1}=0$, 
we have
\[
\dd^{1}(\KW/2^{n})(q,j)
=
\begin{cases}
(\Sq^3\Sq^{1},0,\Sq^{2},0,0) & j \equiv 0 \bmod 4 \\
(\Sq^{3}\Sq^{1},a(\Sq^{2}\Sq^{1}+\Sq^{3}),\Sq^{2},\phi+a\tau\Sq^{1},0) & j\equiv 1\bmod 4 \\
(\Sq^3\Sq^{1},0,\Sq^{2}+\rho\Sq^{1},0,\tau) & j \equiv 2 \bmod 4 \\
(\Sq^{3}\Sq^{1},a^{\prime}(\Sq^{2}\Sq^{1}+\Sq^{3}),\Sq^{2}+\rho\Sq^{1},
\phi + a \rho +a\tau\Sq^{1},\tau) & j\equiv 3\bmod 4,
\end{cases}
\]
for some $a, a' \in h^{0,0}$ and $\phi\in h^{1,1}$.
Consider the commutative diagram for $q$ even
\[
\begin{tikzcd}
\s_{q}(\KQ/2^{n}) \ar[r]\ar[d, "\dd^{1}(\KQ/2)(q)"] & \s_{q}(\KW/2^{n})\ar[d, "\dd^{1}(\KW/2)(q)"] \\
\Sigma^{1,0}\s_{q+1}(\KQ/2^{n}) \ar[r] & \Sigma^{1,0}\s_{q+1}(\KW/2^{n}).
\end{tikzcd}
\]
The top summand $\Sigma^{2q,q}\MZ/2^{n}$ maps by $(\partial^{2^{n}}_{2}, \pr^{2^{n}}_{2})$ and hence trivially to the summand $\Sigma^{2q+4,q+1}\MZ/2$ of $\Sigma^{1,0}\s_{q+1}(\KW/2^{n})$, 
i.e.,
\[
0 = 
\begin{cases}
a(\Sq^{2}\Sq^{1} + \Sq^3)\partial^{2^{n}}_{2} + \Sq^3\Sq^{1}\pr^{2^{n}}_{2} & q \equiv 0\bmod 4 \\
a^{\prime}(\Sq^{2}\partial^{2^{n}}_{2} + \Sq^3)\Sq^{1} + \Sq^3\Sq^{1}\pr^{2^{n}}_{2} & q \equiv 2 \bmod 4.
\end{cases}
\]
This implies $a = a^{\prime} = 0$.
Next we show $\phi = 0$.
For $q'=q-w$ the $E^2$-page of the slice spectral sequence for $\KW/2^{n}$ over a finite field or a completion of $\Q$ takes the form
\begin{equation}
\label{equation:witt2n-E2}
E^2_{p,q,w}
\cong 
\begin{cases}
h^{q',q'}/\phi & p\equiv w\bmod 4 \\
\ker(\phi_{q',q'}) & p\equiv w+1\bmod 4 \\
0 & \text{otherwise}.
\end{cases}
\end{equation}
\aref{tbl:wittQ} and \aref{tbl:motQ} show this filtration is too small to produce the mod $2^{n}$ Witt groups if $\phi \neq 0$.
\end{proof}

In \eqref{equation:witt2n-E2}, 
$h^{q',q'}$ identifies with the quotient of $h^{q',q'} \oplus h^{q'-4,q'} \oplus h^{q'-8,q'} \oplus \dots$ by elements $(x_{1}, x_{2},\dots)$, 
where $x_{i}=\Sq^3\Sq^1 x_{i+1}$ for all $i\geq 1$.
To understand the product structure we write
\begin{align}
B &= 
\begin{pmatrix}
\Sq^3\Sq^1 & 0 & \dots \\
\tau & \Sq^3\Sq^1 & 0 & \ddots\\
0 & \tau & \Sq^3\Sq^1 & \ddots \\
\vdots & \ddots & \ddots & \ddots \\
\end{pmatrix}
\end{align}
 
\begin{align}
\hat{H}^{p,q}_n &= \frac{H_n^{p,q}\oplus h^{p-4,q}\oplus h^{p-8,q}\oplus\dots}
    {\begin{pmatrix}\delta\Sq^2\Sq^1 & 0 \\
        \tau & B \\
        0 & 
    \end{pmatrix}(h^{p-4,q-1}\oplus h^{p-8,q-1}\oplus\dots)}
\end{align}

\begin{align}
\hat{h}^{p,q} &= \frac{h^{p,q}\oplus h^{p-4,q}\oplus h^{p-8,q}\oplus\dots}
    {B (h^{p-4,q-1}\oplus h^{p-8,q-1}\oplus\dots)}.
\label{equation:hath2}
\end{align}

With this notation we can describe the $E^2$-page of the slice spectral sequence for $\KW/2^{n}$ over any field of characteristic unequal to $2$ and rings of $\mathcal{S}$-integers in number fields.

\begin{theorem}
\label{thm:KW2-E2n}
For $q'=q-w$ the $E^{2}$-page of the slice spectral sequence for $\KW/2^{n}$ over a field $F$ of characteristic unequal to $2$ is given by 
\[
E^{2}_{p,q,w}(\KW/2^{n})
\cong 
\begin{cases}
\hat{h}^{q',q'} & p\equiv w\bmod 4 \\
\hat{h}^{q',q'} & p\equiv w+1\bmod 4 \\
0 & \text{otherwise}.
\end{cases}
\]
The same identifications hold over the ring of $\mathcal{S}$-integers in a number field with the exceptions
\[
E^{2}_{p,w+2,w}(\KW/2^{n}) 
\cong \begin{cases}
  \hat{h}^{2,2}/\tau & p - w \equiv 0, 1 \bmod 4 \\
  0 & p - w \equiv 2, 3 \bmod 4,
\end{cases}
\]
\[
E^{2}_{p,w+1,w}(\KW/2^{n}) 
\cong \begin{cases}
  \hat{h}^{1,1}\oplus h^{2,1} & p - w \equiv 0 \bmod 4 \\
  \hat{h}^{1,1} & p - w \equiv 1 \bmod 4 \\
  0 & p - w \equiv 2 \bmod 4 \\
  h^{2,1} & p - w \equiv 3 \bmod 4.
\end{cases}
\]
\end{theorem}

\begin{remark}
\label{rmk:KQ2n-diffs}
We note that $E^2\neq E^{\infty}$ for $\KW/2^{n}$ over $\R$ and $\Q_{2}$.
Over $\R$ there is a family of differentials on the $E^n$-page, 
see \aref{thm:KW2nR}, 
and over $\Q_{2}$ there is a nontrivial $d^2$-differential.
\end{remark}

\begin{theorem}\label{thm:diff-kq-mod2n}
\label{thm:KQ2n-diff}
For $l=1$ or $n$, 
the restriction of the slice $\dd^{1}$-differential to the summand $\Sigma^{q+j,q}\MZ/2^l$ of $\s_{q}(\KQ/2^{n})$ in \eqref{equation:KQ2n-slices} is given by
\begin{align*}
\dd^{1}(\KQ/2^{n})(q,j) 
& =  
\begin{cases}
(\Sq^{3}\Sq^{1},0,\Sq^{2},0,0) & q-1>j\equiv 0\bmod 4 \\
(\Sq^{3}\Sq^{1},0,\Sq^{2},0,0) & q-1>j\equiv 1\bmod 4 \\
(\Sq^{3}\Sq^{1},0,\Sq^{2}+\rho\Sq^{1},0,\tau) & q-1>j\equiv 2\bmod 4 \\
(\Sq^{3}\Sq^{1},0,\Sq^{2}+\rho\Sq^{1},0,\tau) & q-1>j\equiv 3\bmod 4, 
\end{cases} \\
\dd^{1}(\KQ/2^{n})(q,q-1)  
& =  
\begin{cases}
(\partial^{2}_{2^{n}}\Sq^{2}\Sq^{1},0,\Sq^{2},0,0) & q-1\equiv 0\bmod 4 \\
(\Sq^{3}\Sq^{1},0,\Sq^{2},0,0) & q-1\equiv 1\bmod 4 \\
(\partial^{2}_{2^{n}}\Sq^{2}\Sq^{1},0,\Sq^{2}+\rho\Sq^{1},0,\tau) & q-1\equiv 2\bmod 4 \\
(\Sq^{3}\Sq^{1},0,\Sq^{2}+\rho\Sq^{1},0,\tau) & q-1\equiv 3\bmod 4, 
\end{cases} \\
\dd^{1}(\KQ/2^{n})(q,q) & =  
\begin{cases}
(0,\Sq^{2}\partial^{2^{n}}_{2},\Sq^{2}\circ\pr^{2^{n}}_{2},0,0) & q\equiv 0\bmod 4 \\
(0,\inc^2_{2^{n}} \circ \Sq^{2}\Sq^{1},\Sq^{2},0,0) & q\equiv 1\bmod 4 \\
(0,\Sq^{2}\partial^{2^{n}}_{2},\Sq^{2}\circ\pr^{2^{n}}_{2},\tau\partial^{2^{n}}_{2},\tau\circ\pr^{2^{n}}_{2}) & q\equiv 2\bmod 4 \\
(0,\inc^2_{2^{n}} \circ \Sq^{2}\Sq^{1},\Sq^{2}+\rho\Sq^{1},0,\tau) & q\equiv 3\bmod 4. 
\end{cases} 
\end{align*}
The $i$th component of $\dd^{1}(\KQ/2^{n})(q,j)$ is a map $\Sigma^{q+j,q}\MZ/2^l\to\Sigma^{q+j+i,q+1}\MZ/2^{l'}$, 
$l'= 1$ or $n$. 
\end{theorem}
\begin{proof}
This follows from \aref{thm:KW2n-diff}.
\end{proof}

\begin{theorem}
\label{thm:KW2nR}
The only nontrivial $d^{i}$-differentials for $i\geq 2$ in the slice spectral sequence for $\KW/2^{n}$ over the real numbers $\R$ are $d^{n}\colon E^n_{p,q,w}\to E^{n}_{p-1,q+n,w}$ for $p - w \equiv 1 \bmod 4$.
The $E^{\infty}$-page is given by 
\[
E^{\infty}_{p,q,w}(\KW/2^{n})
\cong
\begin{cases}
h^{q',q'} & p - w \equiv 0 \bmod 4, q' = q-w < n \\
0 & \text{otherwise}.
\end{cases}
\]
\end{theorem}
\begin{proof}
Recall the abutment is given by 
\begin{equation}
\label{equation:KW2nabutment}
\KW_{p,w}(\R; \Z/2^{n}) 
\cong
\begin{cases}
\Z/2^{n} &  p - w \equiv 0 \bmod 4 \\
0 & \text{otherwise}.
\end{cases}
\end{equation}
If $E^2_{4k+1+w,w,w} = h^{0,0}$ does not support any differentials then $\KW_{p,w}(\R; \Z/2^{n})\neq 0$ for $p-w \equiv 1 \bmod 4$, 
a contradiction.
By \eqref{equation:h-mult} multiplication by $\rho \in h^{1,1} = \pi_{0,0}\s_{1}(\One)$ 
--- generator in the polynomial algebra $h^{\ast,\ast}$ ---
induces a map $E^r_{p,q,w} \to E^r_{p,q+1,w}$ that commutes with the differentials.
Thus, 
by $\rho$-linearity, 
if $E^r_{4k+1+w,w,w}$ supports a nontrivial $d^r$-differential then so does $E^r_{4k+1+w,w+q,w}$ for every $q \geq 0$.
If $r \neq n$ the terms on the $E^{\infty}$-page cannot produce the groups in \eqref{equation:KW2nabutment} by a cardinality count.
\end{proof}

\begin{remark}
\label{rmk:<-1>}
Since $\langle -1\rangle - \langle 1 \rangle = -2\in W(F)$,
\aref{lem:<-1>} implies $\rho$ maps to $-2\in\pi_{*,*}(\KW)$.
Hence we recover $\KW_{*,*}(\R; \Z/2^{n})$ from the associated graded (all the extensions are nontrivial).
\end{remark}

In the next result we let $a = 2q-p$, $q'=q-w$, $\overline{q}=(q \bmod 4) \in \{0,1,2,3 \}$, and set 
\begin{align}
R_{p,q,w} &= 
\begin{cases}
\hat{h}^{a-3-\overline{q}} & p - q - w \equiv 1 - \overline{q} \bmod 4 \\
\hat{h}^{a-4-\overline{q}} & p - q - w \equiv   - \overline{q} \bmod 4 \\
0 & \text{otherwise},
\end{cases}
\nonumber \\
A_{p,q,w} &= 
\begin{cases}
(\inc^2_{2^{n}}\Sq^2\Sq^1, \Sq^2 + \rho\Sq^1, \tau) \hat{h}^{a-3 - \overline{q},q'-1} \to H_n^{a - \overline{q},q'} \oplus \dots & p - q - w \equiv 1 - \overline{q} \bmod 4, \overline{q} = 0 \\
(\partial\Sq^2\Sq^1, \Sq^2 + \rho\Sq^1, \tau) \hat{h}^{a-4 - \overline{q},q'-1} \to H_n^{a+1 - \overline{q},q'} \oplus \dots & p - q - w \equiv  - \overline{q} \bmod 4, \overline{q} = 0  \\
(\Sq^3\Sq^1, \Sq^2 + \rho\Sq^1, \tau) \hat{h}^{a-3 - \overline{q},q'-1} \to h^{a+1 - \overline{q},q'} \oplus \dots & p - q - w \equiv 1 - \overline{q} \bmod 4, \overline{q} \neq 0  \\
(\Sq^3\Sq^1, \Sq^2 + \rho\Sq^1, \tau) \hat{h}^{a-4 - \overline{q},q'-1} \to h^{a - \overline{q},q'} \oplus \dots & p - q - w \equiv  - \overline{q} \bmod 4, \overline{q} \neq 0  \\
0 & \text{otherwise}.
\end{cases} \nonumber
\end{align}

\begin{remark}
$R_{p,q,w}$ and $\im A_{p,q,w}$ are used to identify $h^{a-\bar{q},q'}$ with $\rho^{4k}\tau^{-4k}h^{a-\bar{q}-4k,q'}$. 
This records the multiplicative structure on the $E^2$-page, 
which is important for determining higher differentials and extensions, 
cf.~\aref{thm:OFS-E8}. 
\end{remark}

\begin{theorem}
\label{thm:E2-KQ2n}
Over fields $F$ of $\Char(F)\neq 2$ and rings of $\mathcal{S}$-integers in number fields we identify the term $E^2_{p,q,w}(\KQ/2^{n})$ as follows.

For $q+w\leq p$ there is a direct sum decomposition 
$$
E^2_{p,q,w}
\cong
(\widetilde{E}^{2}_{p,q,w}
\oplus 
R_{p,q,w})/\im A_{p,q,w},
$$
where $\widetilde{E}^{2}_{p,q,w}$ is the first homology group of the following complexes:
\begin{align}
\label{equation:KQ2nE0}
\widetilde{E}^{2}_{p,q\equiv 0\bmod 4,w} &\cong
H_{1}(0 \to H_n^{a,q'} \xrightarrow{
\begin{pmatrix}
\Sq^2\partial^{2^{n}}_{2} \\
\Sq^2\pr
\end{pmatrix}
} h^{a+3, q'+1}\oplus h^{a+2,q'+1}), 
\\
\widetilde{E}^{2}_{p,q\equiv1\bmod 4,w} &\cong
H_{1}(H_n^{a-3,q'-1} \xrightarrow{
\begin{pmatrix}
\Sq^2\partial^{2^{n}}_{2} \\
\Sq^2\pr
\end{pmatrix}
}
h^{a,q'}\oplus {h}^{a-1,q'} \xrightarrow{
\begin{pmatrix}
\Sq^2 & 0 \\
0 & \Sq^2
\end{pmatrix}
}
h^{a+2,q'+1} \oplus h^{a+3,q'+1}), 
\\
\widetilde{E}^{2}_{p,q\equiv2 \bmod 4,w} &\cong
H_{1}(h^{a-3,q'-1}\oplus h^{a-4,q'-1} \xrightarrow{
\begin{pmatrix}
\inc^2_{2^{n}}\Sq^2\Sq^1 & \partial^2_{2^{n}}\Sq^2\Sq^1 \\
\Sq^2 & 0 \\
0 & \Sq^2
\end{pmatrix}
}
\\ &
H_n^{a,q'}\oplus h^{a-1,q'} \oplus {h}^{a-2,q'} \xrightarrow{
\begin{pmatrix}
\tau\partial^{2^{n}}_{2} & \Sq^2 & 0 \\
\tau\pr^{2^{n}}_{2} & 0 & \Sq^2
\end{pmatrix}
} h^{a+3, q+1}\oplus h^{a+2,q+1}), 
\\
\widetilde{E}^{2}_{p,q\equiv 3\bmod 4,w} & 
\cong 
\label{equation:KQ2nE3}
H_{1}(H_n^{a-3,q'-1}\oplus h^{a-4,q'-1}\oplus h^{a-5,q'-1} \xrightarrow{
\begin{pmatrix}
\tau\partial^{2^{n}}_{2} & \Sq^2 & 0\\
\tau\pr^{2^{n}}_{2} & 0 & \Sq^2
\end{pmatrix}
}
h^{a-2,q'}\oplus h^{a-3,q'} \to 0). 
\end{align}

For $q+w>p$ the canonical map $\KQ/2^{n} \to \KW/2^{n}$ induces an isomorphism of $E^2$-pages for $q + w > p$ over fields and for $q + w - 1 > p$ over rings of $\mathcal{S}$-integers. 
The $E^{2}$-page of $\KW/2^{n}$ is given in \aref{thm:KW2-E2n}.
\end{theorem}
\begin{proof}
This follows by inspection of the differentials similarly to the proof of \aref{thm:KQ2-E2}.
To give the gist of the argument we discuss a few special cases.

When $q \equiv 0 \bmod 4$ the group $E^2_{p,q,w}$ is the homology of the complex:
\begin{align*}
h^{a-3,q'-1}\oplus h^{a-4,q'-1} \oplus h^{a-5,q'-1} \oplus \dots \\
\xrightarrow  
{\tiny
\begin{pmatrix}
\inc^{2}_{2^{n}}\Sq^2\Sq^1 & \partial^{2}_{2^{n}}\Sq^2\Sq^1 & 0 & \dots \\
\Sq^2 + \rho\Sq^1 & 0 & \Sq^3\Sq^1 & 0  & \dots \\
0 & \Sq^2 + \rho\Sq^1 & 0 & \Sq^3\Sq^1 & 0 & \dots \\
\tau & 0 & \Sq^2 & 0 & \Sq^3\Sq^1 & 0 & \dots \\
0 & \tau & 0 & \Sq^2 & 0 & \Sq^3\Sq^1 & 0 & \dots \\
0 & 0 & 0 & 0 & \Sq^2 + \rho\Sq^1 & 0 & \Sq^3\Sq^1 & 0 & \dots \\
0 & 0 & 0 & 0 & 0 & \Sq^2+\rho\Sq^1 & 0 & \Sq^3\Sq^1 & 0 & \dots \\
0 & 0 & 0 & 0 & \tau & 0 & \Sq^2 & 0 & \Sq^3\Sq^1 & 0 & \dots \\
\vdots & \vdots & \ddots & \ddots & \ddots & \ddots & \ddots & \ddots & \ddots & \ddots
\end{pmatrix}
} \\
H_n^{a,q'} \oplus h^{a-1,q'} \oplus h^{a-2,q'} \oplus \dots  \\
\xrightarrow{\tiny
\begin{pmatrix}
\Sq^2\partial^{2^{n}}_{2} & \Sq^3 \Sq^1 & 0 & \dots \\
\Sq^2\pr^{2^{n}}_{2} & 0 & \Sq^3\Sq^1 & 0  & \dots \\
0 & \Sq^2 + \rho\Sq^1 & 0 & \Sq^3\Sq^1 & 0 & \dots \\
0 & 0 & \Sq^2 + \rho\Sq^1 & 0 & \Sq^3\Sq^1 & 0 & \dots \\
0 & \tau & 0 & \Sq^2 & 0 & \Sq^3\Sq^1 & 0 & \dots \\
0 & 0 & \tau & 0 & \Sq^2 & 0 & \Sq^3\Sq^1 & 0 & \dots \\
0 & 0 & 0 & 0 & 0 & \Sq^2 + \rho\Sq^1 & 0 & \Sq^3\Sq^1 & 0 & \dots \\
\vdots & \vdots & \ddots & \ddots & \ddots & \ddots & \ddots & \ddots & \ddots & \ddots
\end{pmatrix}
} \\
h^{a+3,q'+1} \oplus h^{a+2,q'+1} \oplus h^{a+1,q'+1} \oplus \dots 
\end{align*}
In the latter matrix, 
a $\Sq^2$ to the right of $\tau$ and a $\Sq^3\Sq^1$ above $\tau$ cannot both act nontrivally, 
i.e., 
$\Sq^2(\tau^{p+2})$ and $\Sq^3\Sq^1(\tau^{p})$ are never nontrivial simultaneously,
cf.~\aref{tbl:Sq-action}. 
We find the kernel is given by 
\begin{align}
\label{equation:q=0kerKQ2n}
\ker(\begin{pmatrix}
\Sq^2\partial^{2^{n}}_{2}  \\
\Sq^2\pr^{2^{n}}_{2}
\end{pmatrix} : & \, H_n^{a,q'} \to h^{a+3,q'+1}\oplus h^{a+2,q'+1}) \\
&\oplus(\tau^{-1}\Sq^2,1)(h^{a-3,q'}\oplus h^{a-7,q'}\oplus\dots)  \\
&\oplus(\tau^{-1}\Sq^2,1)(h^{a-4,q'}\oplus h^{a-9,q'}\oplus\dots), 
\end{align}
and the image by
\begin{align}
\label{equation:q=0imKQ2n}
\begin{pmatrix}
\inc^{2}_{2^{n}}\Sq^2\Sq^1 & \partial^{2}_{2^{n}}\Sq^2\Sq^1 & 0 &  \\
\Sq^2 + \rho\Sq^1 & 0 & \Sq^3\Sq^1 & 0\\
0 & \Sq^2+\rho\Sq^1 & 0 & \Sq^3\Sq^1 \\
\tau & 0 & \Sq^2 & 0 \\
0 & \tau & 0 & \Sq^2 
\end{pmatrix}(&h^{a-3,q'-1}\oplus h^{a-4,q'-1}\oplus h^{a-5,q'-1}\oplus h^{a-6,q'-1}) \\
& + (\Sq^3\Sq^1, \Sq^2 + \rho\Sq^1, \tau)(h^{a-7,q'-1} \oplus h^{a-11,q'-1} \oplus \dots)  \\
& + (\Sq^3\Sq^1, \Sq^2 + \rho\Sq^1, \tau)(h^{a-8,q'-1} \oplus h^{a-12,q'-1} \oplus \dots). 
\end{align}
Here $\Sq^2\Sq^1$ above $\tau$ and $\Sq^2$ to the left of $\tau$ cannot both act nontrivially simultaneously, 
hence the image contains $\im A_{p,q,w}$.
Here $\Sq^3\Sq^1$ acts (non)trivially on $h^{a-j- 4k,q'-1}$ depending on $p-q-w \bmod 4$.
The last terms in \eqref{equation:q=0kerKQ2n} either cancel the corresponding terms in \eqref{equation:q=0imKQ2n} or give rise to $\hat{h}^{a-j,q'}$, $j=3-\bar{q}, 4-\bar{q}$ 
(recall $\Sq^3\Sq^1(\tau^{q}) \neq 0$ if and only if $q \equiv 3 \bmod 4$, cf.~\aref{tbl:Sq-action}).
This produces $R_{p,q,w}$ and $\im A_{p,q,w}$ connects it to the first term of the kernel.
In this way we arrive at the complex in \eqref{equation:KQ2nE0}.

When $q \equiv 3 \bmod 4$ the group $E^2_{p,q,w}$ is the homology of the complex:
\begin{align*}
H_n^{a-3,q'-1}\oplus h^{a-4,q'-1} \oplus h^{a-5,q'-1} \oplus \dots \\
\xrightarrow  
{\tiny
\begin{pmatrix}
\Sq^2\partial^{2^{n}}_{2} & \Sq^3 \Sq^1 & 0 & \dots \\
\Sq^2\pr^{2^{n}}_{2} & 0 & \Sq^3\Sq^1 & 0  & \dots \\
\tau\partial^{2^{n}}_{2} & \Sq^2 & 0 & \Sq^3\Sq^1 & 0 & \dots \\
\tau\pr^{2^{n}}_{2} & 0 & \Sq^2 & 0 & \Sq^3\Sq^1 & 0 & \dots \\
0 & 0 & 0 & \Sq^2 + \rho\Sq^1 & 0 & \Sq^3\Sq^1 & 0 & \dots \\
0 & 0 & 0 & 0 & \Sq^2 + \rho\Sq^1 & 0 & \Sq^3\Sq^1 & 0 & \dots \\
0 & 0 & 0 & \tau & 0 & \Sq^2 & 0 & \Sq^3\Sq^1 & 0 & \dots \\
0 & 0 & 0 & 0 & \tau & 0 & \Sq^2 & 0 & \Sq^3\Sq^1 & 0 & \dots \\
\vdots & \vdots & \ddots & \ddots & \ddots & \ddots & \ddots & \ddots & \ddots & \ddots
\end{pmatrix}
} \\
h^{a,q'} \oplus h^{a-1,q'} \oplus h^{a-2,q'} \oplus \dots  \\
\xrightarrow{\tiny
\begin{pmatrix}
\inc^{2}_{2^{n}}\Sq^2\Sq^1 & \partial^{2}_{2^{n}}\Sq^2 \Sq^1 & 0 & \dots \\
\Sq^2 + \rho\Sq^1 & 0 & \Sq^3\Sq^1 & 0  & \dots \\
0 & \Sq^2 + \rho\Sq^1 & 0 & \Sq^3\Sq^1 & 0 & \dots \\
\tau & 0 & \Sq^2 & 0 & \Sq^3\Sq^1 & 0 & \dots \\
0 & \tau & 0 & \Sq^2 & 0 & \Sq^3\Sq^1 & 0 & \dots \\
0 & 0 & 0 & 0 & \Sq^2 + \rho\Sq^1 & 0 & \Sq^3\Sq^1 & 0 & \dots \\
0 & 0 & 0 & 0 & 0 & \Sq^2 + \rho\Sq^1 & 0 & \Sq^3\Sq^1 & 0 & \dots \\
\vdots & \vdots & \ddots & \ddots & \ddots & \ddots & \ddots & \ddots & \ddots & \ddots
\end{pmatrix}
} \\
H_n^{a+3,q'+1} \oplus h^{a+2,q'+1} \oplus h^{a+1,q'+1} \oplus \dots 
\end{align*}
Here a $\Sq^2$ to the right of $\tau$ and a $\Sq^3\Sq^1$ above $\tau$ cannot both act nontrivally;
$\Sq^2(\tau^{p+2})$ and $\Sq^3\Sq^1(\tau^{p})$ are never nontrivial simultaneously, 
cf.~\aref{tbl:Sq-action}.
We find the kernel is given by 
\begin{align}
\label{equation:q=3kerKQ2n}
\ker(\begin{pmatrix}
\inc^{2}_{2^{n}} \Sq^2\Sq^1 & \partial^{2}_{2^{n}}\Sq^2\Sq^1 & 0 & 0 \\
\tau & 0 & \Sq^2 & 0 \\
0 & \tau & 0 & \Sq^2
\end{pmatrix}: & \, h^{a,q'}\oplus h^{a-1,q'}
\oplus h^{a-2,q'}\oplus h^{a-3,q'} \\
&\to H_n^{a+3,q'+1}\oplus h^{a,q'+1}\oplus h^{a-1,q'+1})  \\
&\oplus(\tau^{-1}\Sq^2,1)(h^{a-6,q'}\oplus h^{a-10,q'}\oplus\dots)  \\
&\oplus(\tau^{-1}\Sq^2,1)(h^{a-7,q'}\oplus h^{a-11,q'}\oplus\dots), 
\end{align}
and the image by
\begin{align}
\label{equation:q=3imKQ2n}
\begin{pmatrix}
\Sq^2\partial^{2^{n}}_{2} & \Sq^3\Sq^1 & 0 \\
\Sq^2\pr^{2^{n}}_{2} & 0 & \Sq^3\Sq^1 \\
\tau\partial^{2^{n}}_{2} & \Sq^2 & 0 \\
\tau\pr^{2^{n}}_{2} & 0 & \Sq^2
\end{pmatrix}(&H_n^{a-3,q'-1}\oplus h^{a-4,q'-1}\oplus h^{a-5,q'-1}) \\
& + (\Sq^3\Sq^1, \Sq^2 + \rho\Sq^1, \tau)(h^{a-6,q'-1} \oplus h^{a-10,q'-1} \oplus \dots)  \\
& + (\Sq^3\Sq^1, \Sq^2 + \rho\Sq^1, \tau)(h^{a-7,q'-1} \oplus h^{a-11,q'-1} \oplus \dots). 
\end{align}
The action of $\Sq^3\Sq^1$ on $h^{a-j- 4k,q'-1}$ depends on $p-q-w \bmod 4$.
The last terms in \eqref{equation:q=3kerKQ2n} and \eqref{equation:q=3imKQ2n} either cancel or give rise to $\hat{h}^{a-j,q'}$, 
$j = 3-\bar{q}, 4-\bar{q}$  (recall $\Sq^3\Sq^1(\tau^{q}) \neq 0$ if and only if $q \equiv 3 \bmod 4$, cf.~\aref{tbl:Sq-action}).
This produces $R_{p,q,w}$ and $\im A_{p,q,w}$ connects it to the first term of the kernel.
Hence $E^2_{p,q,w}$ is the homology of the complex
\begin{align*}
&H_n^{a-3,q'-1}\oplus h^{a-4,q'-1}\oplus h^{a-5,q'-1} \\
&\xrightarrow{
\begin{pmatrix}
\Sq^2\partial^{2^{n}}_{2} & \Sq^3\Sq^1 & 0 \\
\Sq^2\pr^{2^{n}}_{2} & 0 & \Sq^3\Sq^1 \\
\tau\partial^{2^{n}}_{2} & \Sq^2 & 0 \\
\tau\pr^{2^{n}}_{2^{n}} & 0 & \Sq^2
\end{pmatrix}
}
\\ 
&h^{a,q'}\oplus h^{a-1,q'}\oplus
h^{a-2,q'}\oplus h^{a-3,q'} \oplus R_{p,q,w} \\
& \xrightarrow{
\begin{pmatrix}
\inc^{2}_{2^{n}} \Sq^2\Sq^1 & \partial^{2}_{2^{n}}\Sq^2\Sq^1 & 0 & 0 \\
\tau & 0 & \Sq^2 & 0 \\
0 & \tau & 0 & \Sq^2
\end{pmatrix}
}\\
&H_n^{a+3, q'+1}\oplus h^{a,q'+1}\oplus h^{a-2,q'+1}
\end{align*}
modulo $\im A_{p,q,w}$.
Here $\Sq^2\Sq^1$ and $\Sq^2$ are connected via $\tau$ in the latter matrix and they are never nontrivial simultaneously.
This yields the complex in \eqref{equation:KQ2nE3}.
Similar arguments apply in the remaining cases.
\end{proof}

Next we specialize the computation in \aref{thm:E2-KQ2n} to the real numbers $\R$ and rings of $\mathcal{S}$-integers $\OFS$ in a number field $F$.
We determine the higher differentials by comparison with $\KW/2^{n}$ over $\R$.
A combination of \aref{lem:coh-R-mod2n} and \aref{thm:E2-KQ2n} yields the following description of the $E^2$-page.

\begin{corollary}
\label{cor:E2-F}
Over a number field $F$ and its ring of $\mathcal{S}$-integers $\OFS$ the $E^2$-page of the slice spectral sequence for $\KQ/2^{n}$ is given as follows.
Let $\epsilon$ be $0$ for a number field $F$ and $1$ for $\OFS$.
When $a = 2q - p < 0$ we have 
\[
E^{2}_{p,q,w}(\KQ/2^{n}) = 0.
\]
When $a = 2q-p\geq 6$ and $q + w - \epsilon \leq p$ we have
\[
E^2_{p,q,w}(\KQ/2^{n}) = \begin{cases}
  \hat{H}_n^{a,q'} & p - w \equiv 0, 1\bmod 4 \text{ and } q \equiv 0 \bmod 4 \\
  \hat{h}^{a - (a - q' \bmod 4),q'} & p - w \equiv 0, 1\bmod 4 \text{ and } q \not\equiv 0 \bmod 4 \\
  0 & \text{otherwise}.
\end{cases}
\]
When $q + w - \epsilon > p$ we have
\[
E^{2}_{p,q,w}(\KQ/2^{n})
\cong 
\begin{cases}
\hat{h}^{q',q'} & p - w \equiv 0,1 \bmod 4 \\
0 & \text{otherwise}.
\end{cases}
\]

The same identifications hold over $\OFS$ with the exceptions
\[
E^{2}_{p,w+2,w}(\KQ/2^{n}) 
\cong \begin{cases}
  \hat{h}^{2,2}/\tau & p - w \equiv 0, 1 \bmod 4 \\
  0 & p - w \equiv 2, 3 \bmod 4,
\end{cases}
\]
\[
E^{2}_{p,w+1,w}(\KQ/2^{n}) 
\cong \begin{cases}
  \hat{h}^{1,1}\oplus h^{2,1} & p - w \equiv 0 \bmod 4 \\
  \hat{h}^{1,1} & p - w \equiv 1 \bmod 4 \\
  0 & p - w \equiv 2 \bmod 4 \\
  h^{2,1} & p - w \equiv 3 \bmod 4.
\end{cases}
\]
The remaining groups in the region $0 \leq a= 2q-p < 6$ are given in \aref{thm:E2-KQ2n}.

The $E^2$-page sorted by the congruence class mod $4$ of $w$ are given in \aref{fig:KQ2n-E2-w0}, 
\aref{fig:KQ2n-E2-w1},
\aref{fig:KQ2n-E2-w2},
and
\aref{fig:KQ2n-E2-w3}.
\end{corollary}

\begin{proof}
When $a = 2q - p \geq 6$ we have $E^2_{p,q,w}(F) \cong \bigoplus^{r_{1}}E^2_{p,q,w}(\R)$ by \aref{thm:E2-KQ2n}.
If $a - 3 \geq 3$ then $H^{a-3,q'}(F;\Z/2^{n}) \cong \bigoplus^{r_{1}} H^{a-3,q'}(\R;\Z/2^{n})$; 
that is, 
the terms in the slice spectral sequences over $F$ and $\R$ are isomorphic in this range.
Combined with \aref{lem:Hhstruct} we obtain the figures.
\end{proof}

We determine the $E^{\infty}$-page for $\KQ/2^{n}$ and $n\geq 2$ by comparison with $\KW/2^{n}$ over the reals. 
It turns out that the higher differentials for $\KQ/2^{n}$ are determined by the $d^n$-differentials for $\KW/2^{n}$ over $\R$.

\begin{theorem}
\label{thm:OFS-E8}
Let $F$ be a number field with ring of $\mathcal{S}$-integers $\OFS$ and let $n\geq 2$.
The slice spectral sequence for $\KQ/2^{n}$ over $\OFS$ has only $d^r$-differentials for $r\geq 2$ when $r=n$, 
and the $E^{\infty} = E^{n+1}$-page is obtained from the $E^2 = E^{n}$-page (see \aref{fig:KQ2n-E2-w0}, \aref{fig:KQ2n-E2-w1}, \aref{fig:KQ2n-E2-w2}, and \aref{fig:KQ2n-E2-w3}) as follows:
\begin{itemize}
\item If $E^2_{p,q,w}(\R;\KW/2^{n})$ supports a nontrivial $d^n$-differential then $E^{\infty}_{p,q,w}(\OFS;\KQ/2^{n})$ identifies with the kernel of 
$E^2_{p,q,w}(\OFS;\KQ/2^{n}) \to \bigoplus^{r_{1}}E^2_{p,q,w}(\R;\KW/2^{n})$.
\item 
If $E^2_{p,q,w}(\R;\KW/2^{n})$ is the target of a nontrivial $d^n$-differential then $E^{\infty}_{p,q,w}(\OFS;\KQ/2^{n})$ identifies with the cokernel of 
$E^2_{p-1,q-n,w}(\OFS;\KQ/2^{n}) \to \bigoplus^{r_{1}}E^2_{p-1,q-n,w}(\R;\KW/2^{n})$.
\item In all other degrees we have $E^{\infty} = E^{2}$.
\end{itemize}
The $E^{\infty}$-page of the weight $0$ slice spectral sequence for $\KQ/2^{n}$ is displayed in \aref{fig:KQ2n-E8-w0}.
In negative degrees it is isomorphic to the $E^{\infty}$-page for Witt-theory, cf.~\aref{thm:KW2nR}.
\begin{figure}[h]
\caption{$E^{\infty}_{8k+p,4k+q,0}(\KQ/2^{n}), k \geq 0$}
\begin{center}
\resizebox{\linewidth}{!}{
\begin{tikzpicture}[font=\tiny,scale=0.8]
        \draw[help lines,xstep=2.9,ystep=1.0] (-0.5,0.0) grid (23.7,10.5);
        \foreach \i in {0,...,5} {\node[label=left:$\i$] at (-.5,1.0*\i) {};}
        \foreach \i in {0,...,8} {\node[label=below:$\i$] at (2.9*\i,0-.2) {};}
        
{\node[label=left:$n+1$] at (-.5,7.0) {};}
{\node[label=left:$n+2$] at (-.5,8.0) {};}
{\node[label=left:$n+3$] at (-.5,9.0) {};}
{\node[label=left:$n+4$] at (-.5,10.0) {};}
{\draw[fill] (0.0,0.0) circle (1pt);}
{\node[above right=0pt] at (-0.145,0.0) {$H^{0}$};}
{\draw[fill] (0.0,1.0) circle (1pt);}
{\node[above right=0pt] at (-0.145,1.0) {$h^{1}$};}
{\node[above right=0pt] at (0.348,1.3) {$\oplus$};}
{\node[above right=0pt] at (0.725,1.3) {$\ker\rho^2_{2}$};}
{\draw[fill] (0.0,2.0) circle (1pt);}
{\node[above right=0pt] at (-0.145,2.0) {$h^{2}$};}
{\draw[fill] (0.0,3.0) circle (1pt);}
{\node[above right=0pt] at (-0.145,3.0) {$h^{3}$};}
{\draw[fill] (0.0,4.0) circle (1pt);}
{\node[above right=0pt] at (-0.145,4.0) {$\hat{H}^{8}$};}
{\draw[fill] (0.0,5.0) circle (1pt);}
{\node[above right=0pt] at (-0.145,5.0) {$\hat{h}^{9}$};}
{\draw[fill] (2.9,1.0) circle (1pt);}
{\node[above right=0pt] at (2.755,1.0) {$h^{1}$};}
{\draw[fill] (2.9,2.0) circle (1pt);}
{\node[above right=0pt] at (2.755,2.0) {$\begin{pmatrix}\tau\partial & 0 \\ \tau\pr & \Sq^{2} \end{pmatrix}_3 \oplus \ker\rho_{2,2}$};}
{\draw[fill] (5.8,1.0) circle (1pt);}
{\node[above right=0pt] at (5.655,1.0) {$h^{0}$};}
{\draw[fill] (5.8,2.0) circle (1pt);}
{\node[above right=0pt] at (5.655,2.0) {${\begin{pmatrix}
\tau\partial & 0 & 0\\
\tau\pr & 0 & \Sq^2\\
\end{pmatrix}}_{2}$};}
{\draw[fill] (5.8,3.0) circle (1pt);}
{\node[above right=0pt] at (5.655,3.0) {$\frac{h^{2} + h^{1}}{\tau\partial , \rho^2 , \tau\pr}$};}
{\draw[fill] (8.7,2.0) circle (1pt);}
{\node[above right=0pt] at (8.555,2.0) {${\begin{pmatrix}
\tau\partial & \Sq^2\\
\tau\pr & 0\\
\end{pmatrix}}_{1}$};}
{\draw[fill] (8.7,3.0) circle (1pt);}
{\node[above right=0pt] at (8.555,3.0) {$\frac{h^{1} + h^{0}}{\tau\partial , \tau\pr}$};}
{\draw[fill] (11.6,2.0) circle (1pt);}
{\node[above right=0pt] at (11.455,2.0) {${\begin{pmatrix}
\tau\partial\\
\tau\pr\\
\end{pmatrix}}_{0}$};}
{\draw[fill] (11.6,3.0) circle (1pt);}
{\node[above right=0pt] at (11.455,3.0) {$h^{0}$};}
{\draw[fill] (11.6,4.0) circle (1pt);}
{\node[above right=0pt] at (11.455,4.0) {$H^{4}$};}
{\draw[fill] (11.6,5.0) circle (1pt);}
{\node[above right=0pt] at (11.455,5.0) {$\hat{h}^{5}$};}
{\draw[fill] (17.4,4.0) circle (1pt);}
{\node[above right=0pt] at (17.255,4.0) {${(\Sq^2\pr)}_{2}$};}
{\draw[fill] (20.3,4.0) circle (1pt);}
{\node[above right=0pt] at (20.155,4.0) {${(\Sq^2\partial)}_{1}$};}
{\draw[fill] (20.3,5.0) circle (1pt);}
{\node[above right=0pt] at (20.155,5.0) {$\ker\rho^2_{2}$};}
{\draw[fill] (23.2,4.0) circle (1pt);}
{\draw[fill] (23.2,5.0) circle (1pt);}
{\draw[fill] (23.2,6.0) circle (1pt);}
{\draw[fill] (23.2,7.0) circle (1pt);}
{\draw[fill] (23.2,8.0) circle (1pt);}
{\draw[fill] (23.2,9.0) circle (1pt);}
{\draw[fill] (23.2,10.0) circle (1pt);}
{\node[above right=0pt] at (0.0,6.0) {$\vdots$};}
{\draw[fill] (0.0,6.0) circle (1pt);}
{\draw[fill] (0.0,7.0) circle (1pt);}
{\node[above right=0pt] at (11.6,6.0) {$\vdots$};}
{\draw[fill] (11.6,6.0) circle (1pt);}
{\draw[fill] (11.6,7.0) circle (1pt);}
{\draw[fill] (0.0,7.0) circle (1pt);}
{\node[above right=0pt] at (0.0,7.0) {$\hat{H}^{n'}/h_1$};}
{\draw[fill] (0.0,8.0) circle (1pt);}
{\node[above right=0pt] at (0.0,8.0) {$\hat{H}^{n'}/h_2$};}
{\draw[fill] (11.6,7.0) circle (1pt);}
{\node[above right=0pt] at (11.6,7.0) {$\hat{h}^{n'-2}$};}
{\draw[fill] (11.6,8.0) circle (1pt);}
{\node[above right=0pt] at (11.6,8.0) {$\hat{h}^{n'-1}$};}
{\draw[fill] (11.6,9.0) circle (1pt);}
{\node[above right=0pt] at (11.6,9.0) {$\hat{H}^{n'}$};}
        \end{tikzpicture}

}
\end{center}
\label{fig:KQ2n-E8-w0}
\end{figure}
\end{theorem}

\begin{remark}
\label{rmk:E2n-convention}
The terms $E^{2}_{8k+p+w, 4k+q+w, w}(\KQ/2^{n})$ in the range $-1 \leq p,q \leq 8, w, k \geq 0$ are displayed in \aref{fig:KQ2n-E2-w0}, \aref{fig:KQ2n-E2-w1}, \aref{fig:KQ2n-E2-w2}, 
and \aref{fig:KQ2n-E2-w3}.
By periodicity it suffices to consider $w \leq p < w + 8$.
We use the shorthand notations
\begin{align*}
H^{p} = H^{p,q+4k}_n,\
h^{p} = h^{p,q+4k},\
\wt{h}^p = \hat{h}^{\min\{p, q+4k\}},
\pr = \pr^{2^{n}}_{2},\
\partial = \partial^{2^{n}}_{2},\
\wt{H}^p = \begin{cases}
\hat{H}^{p, q+4k}_n & p \leq q + 4k \\
\hat{h}^{p, q + 4k} & p > q + 4k.
\end{cases}
\end{align*}
Recall the groups $\hat{H}_n^{p,q}$ and $\hat{h}^{p,q}$ are defined in \eqref{equation:hath2}. 
A bracket $\begin{pmatrix} \cdot \end{pmatrix}_{p}$ denotes the kernel of some matrix 
--- where $p$ refers to top cohomological dimension $p$ in the source --- 
as in  
 \[
\begin{pmatrix}
\tau \partial & 0 & 0 \\
\tau \pr & 0 & \Sq^2  
\end{pmatrix}_{2} :=
\ker\left(
\begin{pmatrix}
\tau \partial^{2^{n}}_{2} & 0 & 0 \\
\tau \pr^{2^{n}}_{2} & 0 & \Sq^2  
\end{pmatrix} : H^{2, q + 4k}_n \oplus h^{1, q+4k} \oplus h^{0,q+4k}
\to
h^{3,q+4k+1}_n \oplus h^{2, q+ 1 + 4k}
\right).
\]
\end{remark}

\begin{proof} (\aref{thm:OFS-E8})
Consider the commutative diagram
\[
\begin{tikzcd}
E^r_{p,q,w}(\OFS;\KQ/2^{n}) \ar[r, "f^r_{p,q,w}"]\ar[d, "d^r(\OFS;\KQ/2^{n})"] & E^r_{p,q,w}(\OFS;\KW/2^{n})\ar[d, "d^r(\OFS; \KW/2^{n})"] \ar[r, "g^r_{p,q,w}"] & \bigoplus^{r_{1}}E^r_{p,q,w}(\R;\KW/2^{n})\ar[d, "d^r(\R;\KW/2^{n})"] \\
E^r_{p-1,q+r,w}(\OFS;\KQ/2^{n}) \ar[r, "f^r_{p-1,q+r,w}"] & E^r_{p-1,q+r,w}(\OFS;\KW/2^{n}) \ar[r, "g^r_{p-1,q+r,w}"] & \bigoplus^{r_{1}}E^r_{p-1,q+r,w}(\R;\KW/2^{n}).
\end{tikzcd}
\]
Inductively the maps $f^r_{p-1,q+r,w}$ and $g^r_{p-1,q+r,w}$ are isomorphisms whenever $E^r_{p-1,q+r,w}$ is the target of a nontrivial differential.
For $f^r_{p-1,q+r,w}$ this is the content of \aref{cor:E2-F}.
For $g^r_{p-1,q+r,w}$ the group $E^r_{p-1,q+r,w}(\OFS;\E/2^{n})$
is isomorphic to $h^{b,q'}(\OFS)$ for some $b \geq 3$,
and $h^{b,q'}(\OFS;\Z/2^{n}) \to \bigoplus^{r_{1}} h^{b,q'}(\R)$ is an isomorphism (see \aref{lem:number}).
Thus $d^r(\OFS;\KQ/2^{n})$ is determined by $d^r(\R;\KW/2^{n})$ and the composite $g^r_{p,q,w}\circ f^r_{p,q,w}$.
Since $d^r(\R;\KW/2^{n})$ is nontrivial only when $r = n$ the same holds true for $d^r(\OFS;\KQ/2^{n})$, 
and the spectral sequence collapses at its $E^{n+1}$-page.
\end{proof}

In the following we give an integral calculation of the hermitian $K$-groups of $\OO_{F,\mathcal{S}}$.
Recall that $\KQ_{\ast,\ast}(\OFS)$ is a finitely generated abelian group when $\mathcal{S}$ is finite \cite[Proposition 3.13]{BKO}.
It follows that the group $\KQ_{\ast,\ast}(\OFS; \Z/\ell^n)$ is finite for all primes $\ell$ (when $\ell=2$ this also follows from the $E^{\infty}$-page in \aref{thm:OFS-E8}).
Consider the commutative diagram of cofiber sequences
\[
\begin{tikzcd}
\KQ \ar[r, "2^{n+1}"]\ar[d, "2"] & \KQ \ar[r]\ar[d, "\id"] & \KQ/2^{n+1} \ar[r]\ar[d] & \Sigma^{1,0}\KQ \ar[d, "2"] \\
\KQ \ar[r, "2^{n}"] & \KQ \ar[r] & \KQ/2^{n} \ar[r] & \Sigma^{1,0}\KQ.
\end{tikzcd}
\]
Here $\KQ/2^{n+1} \to \KQ/2^{n}$ induces an inverse system of the slice filtrations of $\KQ_{p,w}(\OFS;\Z/2^{n})$, and an inverse system of the $E^{\infty}$-pages.
Since the groups $\KQ_{p,w}(\OFS;\Z/2^{n})$ are finitely generated, the $\lim^1$-term in the Milnor exact sequence vanishes and we obtain
\[
\KQ_{p,w}(\OFS;\Z_{2})
\cong
\lim_{n} \KQ_{p,w}(\OFS;\Z/2^{n}).
\]
Hence the limit of the inverse system of filtrations becomes an exhaustive, Hausdorff, and complete filtration of $\KQ_{p,w}(\OFS;\Z_{2})$.
Moreover, 
each filtration quotient is the inverse limit of the corresponding $E^{\infty}$-terms.

\begin{theorem}
\label{theorem:integralKQ}
When $\{2,\infty\}\subset\mathcal{S}$ the $2$-adic hermitian $K$-groups $\KQ_{\ast,\ast}(\OFS;\Z_{2})$ are determined up to filtration quotients in \aref{table:KQ2adicgroupsOF}.

\begin{table}
\begin{tabular}{l|l|l}
\hline $n \geq 0$&$l$&$\KQ_{n,4k}(\OFS;\Z_2)$
\\\hline 
$8k$&4&$\f_{0}/\f_{1} = H^{0,4k}, \f_{1}/\f_{2} = h^{1,4k+1}, \f_{2}/\f_{3} = h^{2,4k+2}, \f_{3}=\Z_2^{r_1}$\\
$8k + 1$&2&$\f_{0}/\f_{1} = h^{0,4k+1}, \f_{1}=\ker(\pr_{3,4k+2})\oplus h^{1,4k+2}$\\
$8k + 2$&2&$\f_{0}/\f_{1} = \ker(\tau\pr, \Sq^2)_{2,4k+2}, \f_{1}=h^{1,4k+3}/\tau\pr$\\
$8k + 3$&2&$\f_{0}/\f_{1} = \ker(\pr_{1,4k+2}), \f_{1}=h^{0,4k+3}$\\
$8k + 4$&1&$\f_{0}=\Z_2^{r_1}$\\
$8k + 5$&0&$0$\\
$8k + 6$&1&$\f_{0}=\ker(\Sq^2\pr)_{2,4k+4}$\\
$8k + 7$&2&$\f_{0}/\f_{1} = \ker(\Sq^2\partial)_{1,4k+4}, \f_{1}=\ker(\rho_{2,4k+5}^2)$
\\\hline 
$n \geq 0$&$l$&$\KQ_{n,4k + 1}(\OFS;\Z_2)$
\\\hline 
$8k$&2&$\f_{0}/\f_{1} = \ker(\rho_{1,4k+4}^2), \f_{1}=\ker(\rho_{2,4k+5}^2)$\\
$8k + 1$&4&$\f_{0}/\f_{1} = h^{0,4k}, \f_{1}/\f_{2} = \ker(\pr_{3,4k+1})\oplus h^{1,4k+1}, \f_{2}/\f_{3} = h^{2,4k+2}/\tau\pr, \f_{3}=\Z_2^{r_1}$\\
$8k + 2$&2&$\f_{0}/\f_{1} = \ker(\pr_{2,4k+1})\oplus h^{0,4k+1}, \f_{1}=h^{1,4k+2}/\tau\pr$\\
$8k + 3$&2&$\f_{0}/\f_{1} = \ker(\pr_{1,4k+1}), \f_{1}=h^{0,4k+2}$\\
$8k + 4$&0&$0$\\
$8k + 5$&1&$\f_{0}=\Z_2^{r_1}$\\
$8k + 6$&1&$\f_{0}=\ker(\pr_{2,4k+3})$\\
$8k + 7$&2&$\f_{0}/\f_{1} = \ker(\Sq^2\pr_{1,4k+3}), \f_{1}=\ker(\rho_{2,4k+4})$
\\\hline 
$n \geq 0$&$l$&$\KQ_{n,4k + 2}(\OFS;\Z_2)$
\\\hline 
$8k$&2&$\f_{0}/\f_{1} = \ker(\rho_{1,4k+3}^2), \f_{1}=h^{2,4k+4}/\rho^2$\\
$8k + 1$&2&$\f_{0}/\f_{1} = \ker(\rho_{0,4k+3}^2), \f_{1}=\ker(\rho_{1,4k+4}^2)$\\
$8k + 2$&3&$\f_{0}/\f_{1} = \ker(\tau\pr_{2,4k})\oplus h^{0,4k}, \f_{1}/\f_{2} = h^{1,4k+1}/\tau\pr, \f_{2}=\Z_2^{r_1}$\\
$8k + 3$&2&$\f_{0}/\f_{1} = \ker(\pr_{1,4k}), \f_{1}=h^{0,4k+1}/\tau\pr$\\
$8k + 4$&1&$\f_{0}=\ker(\pr_{0,4k})$\\
$8k + 5$&0&$0$\\
$8k + 6$&2&$\f_{0}/\f_{1} = H^{2,4k+2}, \f_{1}=\Z_2^{r_1}$\\
$8k + 7$&2&$\f_{0}/\f_{1} = H^{1,4k+2}, \f_{1}=h^{2,4k+3}$
\\\hline 
$n \geq 0$&$l$&$\KQ_{n,4k + 3}(\OFS;\Z_2)$
\\\hline 
$8k$&2&$\f_{0}/\f_{1} = h^{1,4k+1}, \f_{1}=h^{2,4k+3}/\rho^2$\\
$8k + 1$&3&$\f_{0}/\f_{1} = \ker(\rho_{0,4k+2}^2), \f_{1}/\f_{2} = \ker(\pr_{3,4k+3}), \f_{2}=h^{2,4k+4}/\tau\pr$\\
$8k + 2$&2&$\f_{0}/\f_{1} = \ker(\tau\pr, \Sq^2)_{2,4k+3}, \f_{1}=h^{1,4k+4}/\tau\pr$\\
$8k + 3$&3&$\f_{0}/\f_{1} = \ker(\pr_{1,4k-1}), \f_{1}/\f_{2} = h^{0,4k}, \f_{2}=\Z_2^{r_1}$\\
$8k + 4$&0&$0$\\
$8k + 5$&0&$0$\\
$8k + 6$&1&$\f_{0}=\ker(\Sq^2\partial)_{2,4k+1}$\\
$8k + 7$&3&$\f_{0}/\f_{1} = H^{1,4k+1}, \f_{1}/\f_{2} = h^{2,4k+2}, \f_{2}=\Z_2^r$
\\\hline 
\end{tabular}
\caption{The $2$-adic hermitian $K$-groups of $\OO_{F,\mathcal{S}}$}
\label{table:KQ2adicgroupsOF}
\end{table}

When $\{\ell,\infty\}\subset\mathcal{S}$ for an odd prime $\ell$ and $p - w \geq 0$ there are isomorphisms 
\begin{equation}
\label{equation:KQoddprimes}
\KQ_{p, w}(\OFS;\Z_{\ell}) 
\cong
\begin{cases}
  H^{0, \lfloor(p-w)/2\rfloor}(\OFS; \Z_{\ell}) & p-w \equiv 0 \bmod 4 \\
  0 & p-w \equiv 1 \bmod 4 \\
  H^{2, 2+\lfloor(p-w)/2\rfloor}(\OFS; \Z_{\ell}) & p-w \equiv 2 \bmod 4 \\
  H^{1, 2+\lfloor(p-w)/2\rfloor}(\OFS; \Z_{\ell}) & p-w \equiv 3 \bmod 4.
\end{cases}
\end{equation}
\end{theorem}
\begin{proof}
In the decomposition
\begin{equation}
\label{equation:description}
\s_{q}(\KQ/2^{n}) 
\simeq
\Sigma^{q,q}(\MZ/2^l \vee \bigvee_{j < q} \Sigma^{j,0}\MZ/2),
\end{equation}
where $l = n$ if $q$ is even, and $l = 1$ if $q$ is odd,
the slice summand $\Sigma^{j,0}\MZ/2$ maps by the identity to the next stage of the inverse system when $j$ is even, and by $2=0$ when $j$ is odd, if $j < q$
(there is no $\Sq^1$ from an odd slice summand to an even slice summand in the inverse system,
cf.~\aref{convention:kq-to-kt-slices}).
Thus to identify the inverse limits it suffices to consider the summands $h^{a,b}$ of $E^2_{p,q,w}(\KQ/2^{n})$ when $a-b+p+w=i\equiv 0\bmod 2$.

Since $\rho \in \pi_{0,0}\s_{1}(\One)$ on the $E^{\infty}$-page of the slice spectral sequence for the sphere $\One$ maps to $-2 \in \KW(F)$ by \aref{lem:<-1>},
we can use $\rho$ to detect non-split extensions in Witt-theory and hence also in hermitian $K$-theory.
Similarly we can detect non-split extension in Witt-theory and hermitian $K$-theory over $\OFS$ with the help of \aref{lem:<-1>2}.
Using \aref{thm:OFS-E8} we are ready to show the calculations in \aref{table:KQ2adicgroupsOF}.
We give proofs in some special cases.

From \aref{fig:KQ2n-E8-w0} we read off the filtration quotients 
\[
H^{0,4k}, h^{1,4k+1} \oplus \ker (\rho^2_{2,4k+2}), h^{2,4k+2}, h^{3,4k+3}, \hat{H}^{8,4k+4}, \dots
\]
for $\KQ_{8k,0}(\OFS;\Z/2^{n})$. 
Since $\ker (\rho^2_{2,4k+2})$ vanishes in the inverse limit we are left with the groups 
\[
H^{0,4k}, h^{1,4k+1}, h^{2,4k+2}, h^{3,4k+3}, \hat{H}^{8,4k+4}, \dots,
\]
connected via $\rho$-multiplications.
Hence the filtration takes the form 
\[
H^{0,4k}, h^{1,4k+1}, h^{2,4k+2}, \Z_{2}^{r_{1}}.
\]
If $k>0$ the filtration identifies with $\ker (\rho^2_{1,4k+1}), \ker (\rho_{2,4k+2}), \Z_{2}^{r_{1}}$ (see \aref{thm:hOFS}).

The filtration quotients of $\KQ_{8k+1,0}(\OFS;\Z_{2})$ are given by 
\[
h^{1,4k+1}, \ker(\pr_{3,4k+2})\oplus h^{1,4k+2}.
\]
Indeed, these are the only summands surviving in the inverse limit of $\{\KQ_{8k+1,0}(\OFS;\Z/2^{n})\}_{n\geq 1}$.

The filtration quotients for $\KQ_{8k+7,2}(\OFS;\Z_{2})$ are given by
\[
\ker(\pr_{1,4k+2}), h^{2,4k+3}/\Sq^2\pr.
\]
By \aref{thm:hOFS} and \aref{lem:Hhstruct}(3) these terms identify with $H^{2,4k+2}$ and $h^{2,4k+3}$, respectively.

When $\ell$ is an odd prime, 
the slices of $\KQ/\ell^n$ are given by 
\[
\s_{q}(\KQ/\ell^n) 
\simeq
\begin{cases}
\Sigma^{2q,q}\MZ/\ell^n & q \equiv 0 \bmod 2 \\
0 & q \equiv 1 \bmod 2.
\end{cases}
\]
Since $\cd_\ell(\OFS) \leq 2$ the slice spectral sequence for $\KQ/\ell^n$ collapses at its $E^{1}$-page, 
and we are done.
\end{proof}

\section{Special values of Dedekind $\zeta$-functions}
\label{section:svDzf}
In this section we relate special values of $\zeta$-functions of totally real abelian number fields to hermitian $K$-groups of rings of integers.
\begin{theorem}
\label{theorem:nice}
For $k\geq 0$ and $F$ a totally real abelian number field with ring of $2$-integers $\OF[\frac{1}{2}]$, 
the Dedekind $\zeta$-function of $F$ takes the values
\begin{align*}
\zeta_F(-1 - 4k) 
= 
\frac{\# H^{2,4k+2}(\OF[\frac{1}{2}];\Z_{2})}{\# H^{1,4k+2}(\OF[\frac{1}{2}];\Z_{2})} 
&=
\frac{\# h^{2,4k+3}}{\# h^{1,4k+3}} \frac{\#\KQ_{8k+2,0}(\OF[\frac{1}{2}];\Z_{2})}{\#\KQ_{8k+3,0}(\OF[\frac{1}{2}];\Z_{2})} \\
&= 
2^{r_{1}}\# h^{2,4k+3}  \frac{\# \KQ_{8k+6,2}(\OF[\frac{1}{2}];\Z_{2})_{\text{tor}} }{\# \KQ_{7 + 8k,2}(\OF[\frac{1}{2}];\Z_{2}) } \\
\zeta_F(-3 - 4k) 
= 
\frac{\# H^{2,4k+4}(\OF[\frac{1}{2}];\Z_{2})}{\# H^{1,4k+4}(\OF[\frac{1}{2}];\Z_{2})} 
&= 2^{r_{1}}\#\ker (\rho^2_{2,4k+4})  \frac{\# \KQ_{8k+6,0}(\OF[\frac{1}{2}];\Z_{2}) }{\# \KQ_{7 + 8k,0}(\OF[\frac{1}{2}];\Z_{2}) } \\
&= 2^{r_{1}+1}\frac{\# h^{2,4k+3}}{\# h^{1,4k+1} } \frac{\# \KQ_{8k+2,2}(\OF[\frac{1}{2}];\Z_{2})_{\text{tor}}}{\# \KQ_{8k+3,2}(\OF[\frac{1}{2}];\Z_{2})}
\end{align*}
up to odd multiples.
\end{theorem}
\begin{proof}
Throughout the proof we write $\KQ_{p,w}$ for $\KQ_{p,w}(\OF[\frac{1}{2}];\Z_{2})$ and $H^{p,w}$ for $H^{p,w}(\OF[\frac{1}{2}];\Z_{2})$.
Based on Wiles's proof of the main conjecture in Iwasawa theory \cite{Wiles}, Kolster \cite[Theorem A.1]{Rognes-Weibel} showed 
\begin{equation}
\label{equation:zetaetalecohomology}
\zeta_F(1 - 2k) 
= 
2^{r_{1}}
\frac{\# H^2_\et(\OF[\frac{1}{2}];\Z_{2}(2k) )}{\# H^1_\et(\OF[\frac{1}{2}];\Z_{2}(2k) )}
\end{equation}
up to odd multiples.
We relate the right hand side of \eqref{equation:zetaetalecohomology} to the hermitian $K$-groups $\KQ_{p,w}$.

By \aref{theorem:integralKQ} the filtration quotients of $\KQ_{8k+2,0}$ are given by 
\[
\ker(\tau \pr, \Sq^2)_{2,2+4k}, h^{1,3+4k}/\tau\pr
\]
and the filtration quotients of $\KQ_{8k+3, 0}$ are given by 
\[
\ker(\pr_{1,2+4k}), h^{0,3+4k}.
\]
We note there are exact sequences
\[
0 \to \ker (\tau\pr^{2^{n}}_{2}\ \Sq^2)_{2,4k+2} \to H_n^{2,4k+2} \oplus h^{0,4k+2}
\to  h^{2,4k+3} \to 0,
\]
\[
0 \to \ker(\tau\pr^{2^{n}}_{2})_{1,4k+2} \to H_n^{1,4k+2} \to h^{1,4k+3} \to h^{1,4k+3}/\tau\pr^{2^{n}}_{2} \to 0,
\]
and
\[
H^{0,2+4k} \xrightarrow{\tau \pr} h^{0,3+4k} \to 0.
\]
Here $\tau\pr^{2^{n}}_{2}$ is surjective on $H_n^{2,4k+2}$, cf.~\aref{lem:Hhstruct}.
These exact sequences imply
\[
\frac{\# h^{1,4k+3} }{\# H^{1,4k+2}}
=
\frac{\# h^{1,4k+3}/\tau\pr^{2^{n}}_{2} }{\# \ker(\tau\pr^{2^{n}}_{2})_{1,4k+2} },
\]
and
\[
\# \ker(\tau\pr^{2^{n}}_{2} \Sq^2)_{2,4k+2} = \frac{2\# H^{2,4k+2} }{\# h^{2,4k+3} }.
\]
In conclusion, 
we find the equalities
\[
\frac{\# \KQ_{8k+2,0}}{\# \KQ_{8k+3,0}}
=
\frac{\frac{2\# H^{2,4k+2} }{\# h^{2,4k+3} } \# h^{1,4k+3}/\tau\pr^{2^{n}}_{2} }
{2 \# \ker(\tau\pr^{2^{n}}_{2})_{1,4k+2} }
=
\frac{\# h^{1,4k+3}  \# H^{2,4k+2} }
{\# h^{2,4k+3} \# H^{1,4k+2} }.
\]

By \aref{theorem:integralKQ} we have
\[
\KQ_{8k+6,0} \cong (\Sq^2\pr^{2^{\infty}}_{2})_{2, 8k+6},
\]
and there is a short exact sequence
\[
0 \to (\ker \rho_{2,4k+5})
\to
\KQ_{8k+7,0}
\to
\ker(\Sq^2\pr^{2^{\infty}}_{2})_{1,4k+4}
\to 0.
\]
Moreover, there is an exact sequence
\[
0 \to \ker(\Sq^2\pr^{2^{\infty}}_{2}) \to H^{2,4k+4} \to h^{4,4k+5} \to 0.
\]
Hence we find
\[
\frac{\# \KQ_{8k+6,0}}{\# \KQ_{8k+7,0}}
=
\frac{\frac{\# H^{2,4k+4} }{\# h^{4,4k+5}}}{\# \ker(\rho_{2,4k+5}) \# H^{1,4+4k} }.
\]

By \aref{theorem:integralKQ} the filtration quotients of $\KQ_{8k+6,2}$ are given by 
\[
H_n^{2,4k+2}, \Z_{2}^{r_{1}} %
\]
Here $\Z_{2}^{r_{1}}$ arises from the tower $\hat{h}^{3,3}, \hat{h}^{4,4}, \dots$.
There is a $\rho$-multiplication between $H_n^{2,4k+2}$ and $\Z_{2}^{r_{1}}$.
The map from $\KQ \to \KW$ induces an isomorphism on the filtration quotients with the exception of $H_n^{2,4k+2} \to h^{2,4k+2}$.
Since $\rho = -2$ in the Witt ring, 
we find 
\[
\# (\KQ_{8k+6,2})_{\tor}  = \# H^{2,4k+2}/\# h^{3,4k+3}.
\]
For $\KQ_{8k+7,2}$ we have the short exact sequence
\[
0
\to
h^{2,4k+2}
\to
\KQ_{8k+7,2}
\to
H_n^{1,4k+2}
\to 
0.
\]
Hence we find
\[
\frac{\# (\KQ_{8k+6,2})_{\tor} }{\#\KQ_{8k+7,2}}
=
\frac{1}{\# h^{3,4k+3}\# h^{2,4k+2} }\frac{\# H^{2,4k+2}  }{\# H^{1,4k+2}}.
\]

The filtration quotients of $(\KQ_{8k+2,2})_\tor$ are given by 
\[
(\tau \pr^{2^{n}}_{2})_{2,4k} \oplus h^{0,4k}, \ker(h^{1,4k+1}/\tau\pr^{2^{\infty}}_{2} \to (h^{5,4k+1}\oplus h^{1,4k+1})/(\Sq^3\Sq^1, \tau)h^{1,4k+4})
\]
obtained as the inverse limit of
\[
(\tau \pr^{2^{n}}_{2})_{2,4k} \oplus h^{0,4k}, h^{1,4k+1}/\tau\pr^{2^{n}}_{2}, \hat{H}_n^{6,4k+2}, \dots.
\]
By comparing with $\KW$ we see the groups $h^{0,4k}$, $h^{1,4k+1}/\tau\pr^{2^{n}}_{2}$, $\hat{H}_n^{6,4k+2}, \dots$ are connected by $2$-multiplications, hence are nontorsion.

For $\KQ_{8k+3,2}$ we have the short exact sequence
\[
0
\to
h^{0,4k+1}
\to
\KQ_{8k+3,2}
\to
\ker(\pr^{2^{\infty}}_{2})_{1,4k}
\to 0.
\]
The exact sequences
\[
0 \to (\tau \pr^{2^{n}}_{2})_{1,4k} \to H_n^{1,4k} \xrightarrow{\tau \pr^{2^{n}}_{2}} h^{1,4k+1} \to h^{1,4k+1}/\tau\pr^{2^{n}}_{2}
\]
and
\begin{align*}
0 \to \ker(h^{1,4k+1}/\tau\pr^{2^{n}}_{2} &\to (h^{5,4k+1}\oplus h^{1,4k+1})/(\Sq^3\Sq^1, \tau)h^{1,4k+4}) \\
&\to h^{1,4k+1}/\tau\pr^{2^{n}}_{2} \to (h^{5,4k+1}\oplus h^{1,4k+1})/(\Sq^3\Sq^1, \tau)h^{1,4k+4} \cong h^{5,4k+1},
\end{align*}
conspire to produce the equalities 
\[
\frac{\# h^{1,4k+1}/\tau\pr^{2^{n}}_{2}}{\ker(\tau \pr^{2^{n}}_{2})_{1,4k}}
= \frac{\# h^{1,4k+1}}{\# H^{1,4k} }
\]
and
\[
\# \ker(h^{1,4k+1}/\tau\pr^{2^{n}}_{2} \to (h^{5,4k+1}\oplus h^{1,4k+1})/(\Sq^3\Sq^1, \tau)h^{1,4k+4}) 
=
\frac{\# h^{1,4k+1}/\tau\pr^{2^{n}}_{2} }{\# h^{5,4k+1} }.
\]
Hence we obtain
\begin{align*}
\frac{\#(\KQ_{8k+2,2})_\tor }{\# \KQ_{8k+3,2}}
&=
\frac{\# (\tau \pr^{2^{n}}_{2})_{2,4k} \# h^{0,4k} \frac{\# h^{1,4k+1}/\tau\pr^{2^{n}}_{2} }{\# h^{5,4k+1} }}{\# (\tau \pr^{2^{n}}_{2})_{1,4k}  \# h^{0,4k+1} } \\
&=
\frac{\# H^{2,4k} }{\# h^{2,4k+1} \# h^{1,4k} }
\frac{\# h^{1,4k+1} }{\# H^{1,4k} \# h^{5,4k+1} }.
\end{align*}
\end{proof}

\begin{remark}
The motivic cohomology groups appearing in \aref{theorem:nice} are given explicitly in terms of $r_{1}$, $r_{2}$, $\sls$, $\ts$, and $\ts^+$ in \aref{lem:struct}.
\end{remark}
\begin{remark}
\aref{theorem:nice} is a vast generalization of \cite[Theorem 5.9]{BKSO} for totally real 2-regular number fields (in which case $\sls=1$, and $r_{2}=\ts=\ts^+=0$).
The grading convention in \cite{BKSO} is such that
\[\
\GW^{[w]}_{p}(\OF[\frac{1}{2}]) 
= 
\KQ_{p-2w,-w}(\OF[\frac{1}{2}]).
\]
\end{remark}

\section{The multiplicative structure on $\s_{*}(\KQ/2^{n})$ for $n>1$}
\label{section:multstrKQ2n}
In this section we determine the multiplicative structure on the slices of $\KQ/2^{n}$ for $n>1$.
The answer shows $\tau^4$ acts periodically.
Our main input is the calculation of the graded slices $\s_{*}(\KQ)$ in \cite{mult}.
In outline, 
we first determine the multiplication on the top summands, 
and then transfer this information to all other summands via multiplying by the Hopf map $\eta$.

\begin{lemma}
\label{lem:sqhopf}
Let $n \geq 1$.
\begin{enumerate}
\item The forgetful map $f : \KQ/2^{n} \to \KGL/2^{n}$ induces the following map on slices
\[
\s_{q}(f) 
= 
\begin{cases}
  (\id, 0, \dots) & q \equiv 0 \bmod 2 \\
  (\inc^{2}_{2^{n}}, \partial^{2}_{2^{n}}, 0, \dots) & q \equiv 1 \bmod 2.
\end{cases}
\]
\item The Hopf map $\eta : \Sigma^{1,1}\KQ/2^{n} \to \KQ/2^{n}$ induces the following map on the top slice summands
\[
\s_{q}(\eta)_{\vert} 
= 
\begin{cases}
  \Sigma^{2q-1,q}\MZ/2 \xrightarrow{\begin{pmatrix}0 \\ \id\end{pmatrix}}
  \Sigma^{2q, q}\MZ/2^{n} \vee   \Sigma^{2q-1, q}\MZ/2 & q \equiv 0 \bmod 2\\
  \Sigma^{2q-1,q}\MZ/2^{n} \xrightarrow{\begin{pmatrix}\partial^{2^{n}}_{2} \\ \pr^{2^{n}}_{2}\end{pmatrix}}
  \Sigma^{2q, q}\MZ/2 \vee   \Sigma^{2q-1, q}\MZ/2 & q \equiv 1 \bmod 2.
\end{cases}
\]
\aref{tbl:Sq-action} shows $\inc^{2}_{2^{n}}$ is the unique nontrivial map $\MZ/2 \to \Sigma^{1,0}\MZ/2^{n}$,
and $\partial^{2}_{2^{n}}$ is the unique nontrivial map $\MZ/2 \to \Sigma^{1,0}\MZ/2^{n}$.
\end{enumerate}
\end{lemma}
\begin{proof}
When $q$ is odd this follows from \cite[Corollary 4.13]{slices} (the top summand of $\s_{q}(\Sigma^{1,1}\KQ/2^{n})$ is even).
  
When $q$ is even we compare the top summands in the commutative diagram
\[
\begin{tikzcd}
\s_{q}(\KQ/2^{n}) \ar[r, "\dd^1"] \ar[d, "\s_{q}(f)"] & \Sigma^{1,0}\s_{q+1}(\KQ/2^{n}) \ar[d, "\Sigma^{1,0}\s_{q+1}(f)"] \\
\s_{q}(\KGL/2^{n}) \ar[r, "\dd^1"]\ar[d] & \Sigma^{1,0}\s_{q+1}(\KGL/2^{n})\ar[d] \\
\s_{q}(\KGL/2) \ar[r, "\dd^1=\QQ_{1}"] & \Sigma^{1,0}\s_{q+1}(\KGL/2)
\end{tikzcd}
\]
obtained from 
\eqref{equation:KGL-d1} and \aref{thm:KQ2n-diff}.
This forces the formula for $\s_{q}(f)$.

The second claim follows from the first via the Wood cofiber sequence \eqref{equation:wood}.
\end{proof}

\begin{lemma}
\label{lem:sqetam}
On slices the iterated Hopf map $\eta^m\colon\Sigma^{m,m}\KQ/2^{n} \to \KQ/2^{n}$, $m\geq 1$, 
restricts to the top summand $\Sigma^{2q-m,q}\MZ/2^l$ of $\s_{q}(\Sigma^{m,m}\KQ/2^{n})$ as
\[
\s_{q}(\eta^m)_{\vert}
=
\begin{cases}
  \Sigma^{2q-m,q}\MZ/2^{n} \xrightarrow{\begin{pmatrix}\partial^{2^{n}}_{2} \\ \pr^{2^{n}}_{2}\end{pmatrix}}
  \Sigma^{2q-m+1, q}\MZ/2 \vee   \Sigma^{2q-m, q}\MZ/2 & q - m \equiv 0 \bmod 2 \\
  \Sigma^{2q-m,q}\MZ/2 \xrightarrow{\begin{pmatrix}0 \\ \id\end{pmatrix}}
  \Sigma^{2q-m+1, q}\MZ/2^l \vee   \Sigma^{2q-m, q}\MZ/2 & q - m \equiv 1 \bmod 2.\\
\end{cases}
\]
Here $l=n$ if $m=1$, and $l=1$ otherwise.
\end{lemma}
\begin{proof}
For $q\equiv m\bmod 4$ this follows since the top summand $\Sigma^{2q-m,q}\MZ/2$ is even.
For $q - m \equiv 1\bmod 4$ we compare with the $\dd^1$-differential.
\aref{lem:sqhopf} shows the claim when $m = 1$.
The following diagram commutes since the $\dd^1$-differential vanishes on $\eta$ (cf.~\cite{April1})
\begin{equation}
\label{equation:vanishesonetacommutativity}
\begin{tikzcd}
\s_{q}(\Sigma^{m,m}\KQ/2^{n}) \ar[r, "\dd^1"] \ar[d, "\s_{q}(\eta)"] & \Sigma^{1,0}\s_{q+1}(\Sigma^{m,m}\KQ/2^{n}) \ar[d, "\Sigma^{1,0}\s_{q+1}(\eta)"] \\
\s_{q}(\KQ/2^{n}) \ar[r, "\dd^1"] & \Sigma^{1,0}\s_{q+1}(\KQ/2^{n}).
\end{tikzcd}
\end{equation}
Nontriviality of $\Sigma^{2q-m,q}\MZ/2 \to \Sigma^{2q-m+1, q}\MZ/2$ would contradict commutativity of \eqref{equation:vanishesonetacommutativity}.
\end{proof}

\begin{lemma}
\label{lem:sqeta-all}
When $q\equiv m\bmod 4$ the restriction of $\s_{q}(\eta^m)$ to the top summand of $\s_*(\KQ/2^{n})$ is given by  
\[
\Sigma^{2q-m,q}\MZ/2^{n} \xrightarrow{\begin{pmatrix}\partial^{2^{n}}_{2} \\ \pr^{2^{n}}_{2} \end{pmatrix}} \Sigma^{2q-m+1,q}\MZ/2 \vee \Sigma^{2q-m,q}\MZ/2.
\]
When $q\not\equiv m\bmod 4$ the map $\s_{q}(\eta^m)$ restricts to the direct summands of $\s_*(\KQ/2^{n})$ as
\[
\Sigma^{2q-m-j,q}\MZ/2 \xrightarrow{\begin{pmatrix}0 \\ \id \end{pmatrix}} \Sigma^{2q-m-j+1,q}\MZ/2^l \vee \Sigma^{2q-m-j,q}\MZ/2.
\]
Here $l$ is $n$ if $m = 1, j = 0$ and $q \equiv 0\bmod 4$, and $1$ otherwise.
\end{lemma}
\begin{proof}
Multiplication by $\eta^m$ yields a surjective map from a top summand to another summand.
Our claim follows from \aref{lem:sqetam}.
\end{proof}

Since all slices are $\s_{0}(\One) \simeq \MZ$-modules it follows that $\s_*(\KQ)$ is an $\MZ$-algebra, 
and hence for our purposes we may form smash products of slices over $\MZ$.
This helps us simplify the calculations.

\begin{lemma}
Under the identification
\[
\MZ/2^{n} \wedge_{\MZ}\MZ/2 
\simeq
\Sigma^{1,0}\MZ/2 \vee \MZ/2
\]
we have
\[
\pr^{2^{n}}_{2} \wedge \MZ/2 
= 
\begin{pmatrix}
  0 & 0 \\
  0 & \id
\end{pmatrix} : \Sigma^{1,0}\MZ/2 \vee \MZ/2 \to \Sigma^{1,0}\MZ/2 \vee \MZ/2
\]
and
\[
\partial^{2^{n}}_{2} \wedge \MZ/2 
= 
\begin{pmatrix}
  0 & 0 \\
  \id & 0
\end{pmatrix} : \Sigma^{1,0}\MZ/2 \vee \MZ/2 \to \Sigma^{2,0} \MZ/2 \vee \Sigma^{1,0}\MZ/2.
\]
\end{lemma}

Recall from \cite{mult} the following description of the multiplication map on the slices of $\KQ$.
\begin{lemma}[\protect{\cite[Theorem 3.3]{mult}}]
\label{lem:KQ-mult}
The multiplication map on the $\MZ/2$-summands of $\s_*(\KQ)\wedge \s_*(\KQ)$ is given by 
\begin{align*}
\Sigma^{m,q}\MZ/2 \wedge_{\MZ} \Sigma^{m',q'}\MZ/2
& \simeq 
\Sigma^{m+m',q+q'}(\Sigma^{1,0} \MZ/2 \wedge \MZ/2) \\
&\xrightarrow{\begin{pmatrix}
\Sq^1 & 0 \\
0 & \id
\end{pmatrix}}
\Sigma^{m+m',q+q'}(\Sigma^{2,0}\MZ/2 \vee \MZ/2)
\end{align*}
and
\begin{align*}
\Sigma^{m,q}\MZ/2 \wedge_{\MZ} \Sigma^{m',q'}\MZ/2
& \simeq
\Sigma^{m+m',q+q'}(\Sigma^{1,0} \MZ/2 \wedge \MZ/2) \\
&\xrightarrow{\begin{pmatrix}
\partial^{2}_{2^{\infty}} & 0 \\
0 & \id    
\end{pmatrix}}
\Sigma^{m+m',q+q'}(\Sigma^{2,0}\MZ \vee \MZ/2).
\end{align*}
The multiplication map on the $\MZ$-summands of $\s_*(\KQ)\wedge \s_*(\KQ)$ is the identity.
\end{lemma}

\begin{lemma}
\label{lem:KQ2n-mult}
For $n>1$ the multiplication map on the even summands of $\s_*(\KQ/2^{n})\wedge_\MZ\s_*(\KQ/2^{n})$ is given as follows:
On the $\MZ/2$-summands it is given by 
\begin{align*}
\Sigma^{m,q}\MZ/2 \wedge_{\MZ} \Sigma^{m',q'}\MZ/2
& \simeq 
\Sigma^{m+m',q+q'}(\Sigma^{1,0} \MZ/2 \wedge_\MZ \MZ/2) \\
&\xrightarrow{\begin{pmatrix}
\Sq^1 & 0 \\
0 & \id    
\end{pmatrix}}
\Sigma^{m+m',q+q'}(\Sigma^{2,0}\MZ/2 \vee \MZ/2)
\end{align*}
and
\begin{align*}
\Sigma^{m,q}\MZ/2 \wedge_{\MZ} \Sigma^{m',q'}\MZ/2
& \simeq
\Sigma^{m+m',q+q'}(\Sigma^{1,0} \MZ/2 \wedge_\MZ \MZ/2) \\
&\xrightarrow{\begin{pmatrix}
\partial^{2}_{2^{n}} & 0 \\
0 & \id    
\end{pmatrix}}
\Sigma^{m+m',q+q'}(\Sigma^{2,0}\MZ/2^{n} \vee \MZ/2).
\end{align*}
On the $\MZ/2^{n} \wedge_{\MZ/2^{n}} \MZ/2$- and $\MZ/2^{n} \wedge_{\MZ/2^{n}} \MZ/2^{n}$-summands it is the identity.
\end{lemma}
\begin{proof}
This is straightforward from \aref{lem:KQ-mult}.
\end{proof}

It remains to determine the multiplication on the odd summands of $\s_*(\KQ/2^{n})\wedge_\MZ\s_*(\KQ/2^{n})$ for $n>1$.
\begin{lemma}
\label{lem:mult-top}
Let $n>1$ and consider the commutative diagram
\begin{equation}
\label{equation:s1KQKGL}
\begin{tikzcd}
  \s_{1}(\KQ/2^{n}) \wedge_\MZ \s_{1}(\KQ/2^{n}) \ar[r]\ar[d] & \s_{2}(\KQ/2^{n})\ar[d] \\
  \s_{1}(\KGL/2^{n}) \wedge_\MZ \s_{1}(\KGL/2^{n}) \ar[r] & \s_{2}(\KGL/2^{n}).
\end{tikzcd}
\end{equation}
Restricting the multiplication map 
\begin{align*}
\Sigma^{2,1}\MZ/2 \wedge_\MZ \Sigma^{2,1}\MZ/2  
\vee 
(\Sigma^{2,1}\MZ/2 \wedge_\MZ \Sigma^{1,1}\MZ/2)  
\vee 
(\Sigma^{1,1}\MZ/2 \wedge_\MZ \Sigma^{2,1}\MZ/2) 
\vee \dots  \\
\to \Sigma^{4,2}\MZ/2^{n} \vee \Sigma^{3,2} \MZ/2,
\end{align*}
to the top summands yields the trivial map
\[
\Sigma^{2,1}\MZ/2 \wedge_\MZ \Sigma^{2,1}\MZ/2 \rightarrow \Sigma^{4,2}\MZ/2^{n} \vee \Sigma^{3,2} \MZ/2, 
\]
and
\begin{align*}
\Sigma^{2,1}\MZ/2 \wedge_\MZ \Sigma^{1,1}\MZ/2
& \simeq 
\Sigma^{4,2}\MZ/2 \vee \Sigma^{3,2}\MZ/2 \\
&
\xrightarrow{
\begin{pmatrix}
\inc^{2}_{2^{n}} & 0 \\
0 & \id \\
\end{pmatrix}
}
\Sigma^{4,2}\MZ/2^{n} \vee \Sigma^{3,2} \MZ/2.
\end{align*}
Permuting the smash factors $\Sigma^{2,1}\MZ/2$ and $\Sigma^{1,1}\MZ/2$ in the source yields the same map in $\SH$.
\end{lemma}
\begin{proof}
Note that $\s_{2}(\KQ/2^{n}) \to \s_{2}(\KGL/2^{n})$ restricts to an isomorphism on the top summand $\Sigma^{4,2}\MZ/2$.
On the first summand the left vertical map $\partial^{2}_{2^{n}} \wedge \partial^{2}_{2^{n}}$ in \eqref{equation:s1KQKGL} is trivial by \aref{lem:partialpartial}.
For the second map we look at the commutative diagram
\[
\begin{tikzcd}[column sep=70pt]
\Sigma^{1,1}\s_{0}(\KQ/2^{n}) \wedge_\MZ \s_{1}(\KQ/2^{n}) \ar[r, "\s_{1}(\eta)\wedge\s_{q}(\KQ/2^{n})"]\ar[d] & \s_{1}(\KQ/2^{n}) \wedge_\MZ \s_{1}(\KQ/2^{n}) \ar[d] \\
\Sigma^{1,1}\s_{1}(\KQ/2^{n}) \ar[r] & \s_{2}(\KQ/2^{n}).
\end{tikzcd}
\]
Here the left vertical map restricts as $(0, \Sigma^{3,2}\MZ/2)$ on the summand $\Sigma^{1,1}\MZ/2^{n} \wedge_\MZ \Sigma^{2,1}\MZ/2 \simeq \Sigma^{4,2}\MZ/2 \vee \Sigma^{3,2}\MZ/2$.
\aref{lem:sqhopf} shows the horizontal maps agree with $(\partial^{2^{n}}_{2}, \pr^{2^{n}}_{2}) \wedge_\MZ \Sigma^{2,1}\MZ/2$ and $(0, \Sigma^{3,2}\MZ/2)$ 
(incidentally, this also implies $\Sigma^{2,1}\MZ/2 \wedge_\MZ \Sigma^{2,1}\MZ/2 \rightarrow \Sigma^{4,2}\MZ/2^{n} \vee \Sigma^{3,2} \MZ/2$ is trivial).
By commutativity of the diagram we have
\begin{align*}
\Sigma^{1,1}\MZ/2 \wedge_\MZ \Sigma^{2,1}\MZ/2 
& \simeq
\Sigma^{4,2}\MZ/2 \vee \Sigma^{3,2}\MZ/2 \\
&\xrightarrow{
\begin{pmatrix}
? & 0 \\
0 & \id
\end{pmatrix}
}
\Sigma^{4,2}\MZ/2^{n} \vee \Sigma^{3,2}\MZ/2.
\end{align*}
Using \eqref{equation:s1KQKGL} and \aref{lem:partialpartial} we conclude $? = \inc^{2}_{2^{n}}$.
\end{proof}

\begin{lemma}
\label{lem:partialpartial}
The map
\[
\MZ/2 \wedge_\MZ \MZ/2 \xrightarrow{\partial^{2}_{2^{n}}\wedge\partial^{2}_{2^{n}}} \Sigma^{2,0} \MZ/2^{n} \wedge_\MZ \MZ/2^{n}
\]
is trivial, 
while
\[
\MZ/2 \wedge_\MZ \MZ/2 \xrightarrow{\partial^{2}_{2^{n}}\wedge\inc^{2}_{2^{n}}} \Sigma^{1,0} \MZ/2^{n} \wedge_\MZ \MZ/2^{n}
\]
coincides with
\[
\Sigma^{1,0} \MZ/2 \vee \MZ/2 \xrightarrow{\begin{pmatrix}
0 & 0 \\
\inc^{2}_{2^{n}} & 0
\end{pmatrix}}
\Sigma^{2,0} \MZ/2^{n} \vee \Sigma^{1,0}\MZ/2^{n}.
\]
\end{lemma}

With notation inspired by \cite{mult} $\s_*(\KQ/2^{n})$ is a polynomial algebra with generators and relations
\[
\MZ/2^{n}[\sqrt{\alpha}^{\pm 1}, \eta, \gamma]/(2\eta = 0, \eta^{2} \xrightarrow{\partial^{2}_{2^{n}}} \sqrt{\alpha}, \gamma^2 = 0, \gamma\eta \xrightarrow{\inc^{2}_{2^{n}}} \sqrt{\alpha} ).
\]
Here the bidegrees of $\sqrt{\alpha}$, $\eta$, and $\gamma$ are $(4,2)$, $(1,1)$, and $(2,1)$, respectively.

\begin{theorem}
\label{thm:KQ2n-mult}
Let $n\geq 2$. For $i \geq 1$ and $j >1$ the multiplicative structure on $\s_*(\KQ/2^{n})$ is given by
\begin{align*}
\eta^{2} = \begin{pmatrix}
  \partial^{2}_{2^{n}} & 0 \\
  0 & 0 \\
  0 & \id
\end{pmatrix},  
\gamma\eta = \begin{pmatrix}
  \inc^{2}_{2^{n}} & 0 \\
  0 & \id
\end{pmatrix}, 
\gamma^{2} = 0, 
\eta^i \eta^j = \begin{pmatrix}
  \Sq^1 & 0 \\
  0 & 0 \\
  0 & \id
\end{pmatrix}, 
\gamma\eta^j = \begin{pmatrix}
  \Sq^1 & 0 \\
  0 & 0 \\
  0 & \id
\end{pmatrix}.
\end{align*}
\end{theorem}
\begin{proof}
Use \aref{lem:sqeta-all} and the multiplicative structure on the top summands given in \aref{lem:mult-top}.
\end{proof}

\section{Motivic cohomology and the Steenrod algebra}
\label{section:mcatSa}
\label{appendixB}
In this section we review the motivic cohomology groups of rings of integers in number fields and properties of the mod $2$ motivic Steenrod algebra.
Suppose $F$ is a field of $\Char(F)\neq 2$.
Recall that $H^{*,*}(-;\Z/2)$ is a contravariant functor defined on the category $\Sm_{F}$ of smooth separated schemes of finite type over $F$, 
taking values in bigraded commutative rings of characteristic $2$ \cite[\S3]{Suslin-Voevodsky:banff}.  
Let $h^{*,*}$ be short for the mod $2$ motivic cohomology ring $\bigoplus_{p,q} h^{p,q}$ of $F$.
The structure map $X\to\spec(F)$ turns $H^{*,*}(X;\Z/2)$ into a bigraded commutative $h^{*,*}$-algebra for every $X\in\Sm_{F}$.  

By change of topology there is a naturally induced map between motivic and \'etale cohomology
\begin{equation}
\label{equation:motivictoetale}
h^{p,q} \to H^{p}_{\et}(F;\mu_{2}^{\otimes q}). 
\end{equation}
The map in \eqref{equation:motivictoetale} is an isomorphism if $p\leq q$ and $q\geq 0$ by the solution of the Beilinson-Lichtenbaum conjecture \cite[\S6]{Suslin-Voevodsky:banff} 
in Voevodsky's proof of Milnor's conjecture on Galois cohomology \cite{Voevodsky:Z/2}.
Here $\tau\in h^{0,1}$ maps to $-1$ in $H^{0}_{\et}(F;\mu_{2})\cong \mu_{2}(F)\cong\{\pm 1\}$;
multiplication by this class is an isomorphism on \'etale cohomology.
It follows that $\tau^{i}\neq 0$ in $h^{0,*}$.
By \cite[Theorem 3.4]{Suslin-Voevodsky:banff} there is an isomorphism $h^{p,p}\cong k^{\Mil}_{p}:=K^{\Mil}_{p}(F)/2$ for all $p\geq 0$, 
and we conclude 
\begin{equation}
\label{equation:motivicisomorphism}
h^{*,*}
\cong 
k^{\Mil}_{*}[\tau].
\end{equation}
That is, 
$h^{p,q}=0$ if $p>q$ or if $p<0$, 
$h^{p,p}\cong k^{\Mil}_{p}$ and multiplication by $\tau\in h^{0,1}$ is an isomorphism on $h^{*,*}$.
Note that when $0\leq p \leq q$,  
every element of $h^{p,q}$ is a $\tau^{q-p}$-multiple of an element of $h^{p,p}$.
Formally inverting $\tau$ yields an isomorphism \cite[Theorem 1.1, Remark 6.3]{MR1740880}, \cite[Corollary 3.6]{HKO}
\begin{equation}
\label{equation:motivictoetaletau}
h^{*,*}[\tau^{-1}] 
\overset{\cong}{\rightarrow}
H^{\ast}_{\et}(F;\mu_{2}^{\otimes \ast}). 
\end{equation}

We write $\AAA^{\ast,\ast}$ for the mod $2$ Steenrod algebra of bistable motivic cohomology operations on $\Sm_{F}$ generated by the Steenrod squares $\Sq^{i}$ 
and multiplication by elements in $h^{*,*}$, 
see \cite{HKO}, \cite{Voevodsky:steenrod}.
Every map $\MZ/2\to\Sigma^{p,q}\MZ/2$ in the stable motivic homotopy category $\SH(F)$ induces a bistable operation of bidegree $(p,q)$;
in fact, 
every operation arises in this way. 
For degree reasons the action of any Steenrod square on $h^{p,p}$ is zero.
By \eqref{equation:motivicisomorphism} and the Cartan formula  \cite[Proposition 9.6]{Voevodsky:steenrod} it essentially remains to determine the actions on $\tau$.
Recall that $\rho$ denotes the class of $-1$ in $h^{1,1}\cong F^{\times}/2$ and $\Sq^{1}(\tau)=\rho$.
In the main body of the paper we make use of \aref{tbl:Sq-action}, which is the content of \cite[Corollary 6.2]{slices}.
\begin{table}[h]
\begin{center}
\begin{tabular}{>{$}l<{$}|>{$}l<{$}>{$}l<{$}>{$}l<{$}>{$}l<{$}}
& \tau^{4k} & \tau^{4k+1} & \tau^{4k+2} & \tau^{4k+3} \\
\hline 
\Sq^{1} & 0 & \rho\tau^{4k} & 0 & \rho\tau^{4k+2} \\
\Sq^{2} & 0 & 0 & \rho^{2}\tau^{4k+1} & \rho^{2}\tau^{4k+2} \\
\Sq^{2} + \rho\Sq^{1} & 0 & \rho^{2}\tau^{4k} & \rho^{2}\tau^{4k+1} & 0 \\
\Sq^{3} & 0 & 0 & \rho^{3}\tau^{4k} & 0 \\
\Sq^{2} \Sq^{1} & 0 & 0 & 0 & \rho^{3}\tau^{4k+1} \\
\Sq^{2} \Sq^{1} + \Sq^{3} & 0 & 0 & \rho^{3}\tau^{4k} & \rho^{3}\tau^{4k+1} \\
\Sq^{3} \Sq^{1} & 0 & 0 & 0 & \rho^{4}\tau^{4k} 
\end{tabular}
\caption{Steenrod operations acting on $\tau$-powers, $k\geq 0$.}
\label{tbl:Sq-action}
\end{center}
\end{table} 

The localization sequences for motivic cohomology \cite[\S14.4]{Levine99} and \'etale cohomology \cite[III.1.3]{MR553999} show that \eqref{equation:motivictoetale}
and \eqref{equation:motivictoetaletau} generalize to the ring of $\mathcal{S}$-integers $\OO_{F,\mathcal{S}}$.
One notable difference compared to the case of fields is the possible non-vanishing of $h^{2,1}\cong\Pic(\OO_{F,\mathcal{S}})/2$.
By the work of Spitzweck \cite{Spitzweck},
the mod $2$ Steenrod algebra $\AAA^{\ast,\ast}$ over the Dedekind domain $\OO_{F,\mathcal{S}}$ has the same structure as over fields.
In particular, 
\aref{tbl:Sq-action} remains valid over rings of $\mathcal{S}$-integers.

\begin{lemma} (\cite[Lemma 11.1]{Voevodsky:steenrod}, \cite[Theorem 1.1]{HKO}, \cite[\S11.2]{Spitzweck})
\label{lem:steenrod-alg}
In weight $0$ and $1$ the motivic Steenrod algebra $\AAA^{\ast,\ast}$ is generated by $\Sq^{1}$, $\Sq^{2}$, $\Sq^3$, $\Sq^{2}\Sq^{1}$ and $\Sq^3\Sq^{1}$ as a free left $h^{*,*}$-module. 
The nontrivial elements are:
\[
\begin{tabular}{>{$} l <{$} | >{$} l <{$}}
(p,q) & \AAA^{p,q} \\ \hline
(0,0) & h^{0,0} \\
(1,0) & h^{0,0} \{ \Sq^{1} \} \\
(0,1) & h^{0,0} \{\tau\} \\
(1,1) & h^{1,1} \oplus h^{0,0}\{ \tau \Sq^{1} \} \\
(2,1) & h^{2,1} \oplus h^{1,1}\{ \Sq^{1}\} \oplus h^{0,0}\{ \Sq^{2} \} \\
(3,1) & h^{2,1} \{\Sq^1\} \oplus h^{0,0}\{ \Sq^{2}\Sq^{1}\} \oplus h^{0,0}\{ \Sq^3 \} \\
(4,1) & h^{0,0}\{ \Sq^3\Sq^{1}\}
\end{tabular}
\]
\end{lemma}

\begin{lemma}
[\protect{\cite[Theorem A.5]{April1}}]
There are naturally induced cofiber sequences
\[
\MZ/2^{n-1} \to \MZ/2^{n} \xrightarrow{\pr^{2^{n}}_{2}} \MZ/2 \xrightarrow{\partial^{2}_{2^{n-1}}} \Sigma^{1,0}\MZ/2^{n-1}
\]
and 
\[
\MZ/2 \xrightarrow{\inc^{2}_{2^{n}}} \MZ/2^{n} \to \MZ/2^{n-1} \xrightarrow{\partial^{2^{n-1}}_{2}} \Sigma^{1,0}\MZ/2.
\]
In weight $0$ and $1$ there are the following nontrivial maps:
\[
\begin{tabular}{>{$} l <{$} | >{$} l <{$}}
(p,q) & [\MZ/2^{n}, \Sigma^{p,q}\MZ/2] \\ \hline
(0,0) & h^{0,0}\{\pr^{2^{n}}_{2} \} \\
(1,0) & h^{0,0}\{\partial^{2^{n}}_{2} \} \\
(0,1) & h^{0,0} \{\tau \pr^{2^{n}}_{2} \} \\
(1,1) & h^{1,1}\{\pr^{2^{n}}_{2} \}\oplus h^{0,1} \{\partial^{2^{n}}_{2} \} \\
(2,1) & h^{2,1}\{\pr^{2^n}_{2}\} \oplus h^{0,0}\{ \Sq^{2} \pr^{2^{n}}_{2} \} \oplus h^{1,1}\{\partial^{2^{n}}_{2}\} \\
(3,1) & h^{2,1}\{ \partial^{2^{n}}_{2} \} \oplus h^{0,0}\{ \Sq^3 \pr^{2^{n}}_{2}, \Sq^2\partial^{2^{n}}_{2} \} \\
(4,1) & h^{0,0}\{ \Sq^1\Sq^2\partial^{2^{n}}_{2} \}
\end{tabular}
\begin{tabular}{>{$} l <{$} | >{$} l <{$}}
(p,q) & [\MZ/2, \Sigma^{p,q}\MZ/2^{n}] \\ \hline
(0,0) & \inc^{2}_{2^{n}}  h^{0,0} \\
(1,0) & \partial^{2}_{2^{n}}  h^{0,0} \\
(0,1) & \inc^{2}_{2^{n}} h^{0,1} \\
(1,1) & \inc^{2}_{2^{n}} h^{1,1} \oplus \partial^{2}_{2^{n}} h^{0,0} \\
(2,1) & \inc^{2}_{2^{n}}h^{2,1} \oplus \inc^{2}_{2^{n}} \Sq^2 h^{0,0} \oplus \partial^{2}_{2^{n}} h^{1,1} \\
(3,1) & \inc^{2}_{2^{n}} \Sq^2\Sq^1 h^{0,0} \oplus \partial^{2}_{2^{n}} h^{2,1} \oplus \partial^{2}_{2^{n}} \Sq^2 h^{0,0} \\
(4,1) &  \partial^{2}_{2^{n}}\Sq^2\Sq^1 h^{0,0}
\end{tabular}
\]
We note the equalities $\pr^{2^{n}}_{2}\partial^{2}_{2^{n}} = \Sq^1 = \partial^{2^{n}}_{2} \inc^{2}_{2^{n}}$.
\end{lemma}

\begin{lemma}
\label{lem:k1-surj}
For every number field $F$ there is a naturally induced surjective map 
\[
F^\times/2 \to \bigoplus^{r_{1}} \R^\times/2.
\]
\end{lemma}
\begin{proof}
Follows from the strong approximation theorem for valuations in number fields \cite[\S15]{Cassels}.
\end{proof}

\begin{lemma}
\label{lem:kn}
For every number field $F$ the naturally induced map
\[
k_n(F) \to \bigoplus^{r_{1}} k_n(\R)
\]
is surjective for $n=1,2$, and bijective for $n\geq 3$.
\end{lemma}
\begin{proof}
Corollary \ref{lem:k1-surj} implies this for $n=1,2$ since $k_{2}(F)$ is generated by products $k_{1}(F) \otimes k_{1}(F) \to k_{2}(F)$.
The bijection for $n\geq 3$ is shown in \cite[Theorem A.2]{Milnor}.
\end{proof}

We can recast these results in terms of motivic cohomology by using the identification of Milnor $K$-groups with the diagonal part of motivic cohomology \cite[Theorem 3.4]{Suslin-Voevodsky:banff}.
\begin{lemma}
\label{lem:number}
For every number field $F$ the naturally induced map
\[
h^{p,q}(F)\to\bigoplus^{r_{1}} h^{p,q}(\R) 
\]
is injective for $p=0$, surjective for $p=1,2$, and an isomorphism for $p \geq 3$.
\end{lemma}

For $a\in\OO_{F}\smallsetminus\{0\}$ we have $\OO_{F,\mathcal{S}}:=\{x\in F\vert \lVert x \rVert_{v}\leq 1\text{ for all } v\not\in\mathcal{S}\}=\OO_{F}[\frac{1}{a}]$ 
if and only if the prime factors of $a\OO_{F}$ are precisely the primes ideals $\mathfrak{p}$ for which the corresponding place $\mathfrak{p}_{v}\in\mathcal{S}\smallsetminus\mathcal{S}_{\infty}$.
Note that $\OO_{F,\mathcal{S}}^{\times}=\{x\in F\vert \lVert x \rVert_{v}=1\text{ for all } v\not\in\mathcal{S}\}$ is a finitely generated abelian group of rank 
$\#\mathcal{S}-1=r_{1}+r_{2}+\#(\mathcal{S}\smallsetminus\mathcal{S}_{\infty})-1$.
By the solution of the  Bloch-Kato conjecture \cite{Voevodsky:Z/l} the results on the motivic cohomology groups of $\OO_{F,\mathcal{S}}$ with $\Z_{(2)}$-coefficients in \cite{Levine99} extend to integral coefficients.

\begin{theorem}(\cite[Theorems 14.5, 14.6]{Levine99})
\label{thm:hOFS}
The integral motivic cohomology group $H^{p,q}(\OO_{F,\mathcal{S}})$ is trivial outside the range $1\leq p\leq q$ except in bidegrees $(0,0)$ and $(2,1)$.

The naturally induced map 
$H^{p,q}(\OO_{F,\mathcal{S}})\to H^{p,q}(F)$ is:
\begin{itemize}
\item bijective for $p=0$ and all $q$, and also for $p=1$ and $q\geq 2$ (14.6 (3)).
\item injective for $(p,q)=(1,1),(2,2)$ (14.6 (1)).
\item surjective for $(p,q) = (2,1)$ (the target is the trivial group).
\item injective for $p=2, q \geq 3$ (14.6 (4)), and there is a short exact localization sequence
\[
0 
\to 
H^{2,q}(\OO_{F,\mathcal{S}}) 
\to 
H^{2,q}(F) 
\to 
\bigoplus_{x\notin\mathcal{S}} H^{1,q-1}(k(x)) 
\to 
0.
\]
\item bijective for $p \geq 3$, $p\leq q$ (14.6 (3), 14.5 (2), 14.5 (3)), and 
\[
H^{p,q}(\OO_{F,\mathcal{S}})
\cong 
\begin{cases}
(\Z/2)^{r_{1}} & p \equiv q \bmod 2 \\
0 & p \not\equiv q \bmod 2.
\end{cases}
\]
\end{itemize}
\end{theorem}

\begin{corollary}
\label{corollary:tau-iso-ints}
For $0 \leq p \leq q$, multiplication by $\tau \in h^{0,1}(\OO_{F,\mathcal{S}})$ induces an isomorphism
\[
\tau: h^{p,q}(\OO_{F,\mathcal{S}}) \overset{\cong}{\to} h^{p,q+1}(\OO_{F,\mathcal{S}}).
\]
\end{corollary}
\begin{proof}
This follows from \eqref{equation:motivictoetale}, 
finiteness of $h^{p,q}(\OFS)$ for all $p,q\in\Z$,
and the exact localization sequence \cite[\S14.4]{Levine99} commutative diagram
\[
\begin{tikzcd}
\bigoplus_{x \not\in\mathcal S} h^{p-2,q-1}(k(x)) \ar[r]\ar[d, "\tau"] & h^{p,q}(\OFS)\ar[r]\ar[d, "\tau"] & h^{p,q}(F)\ar[d, "\tau"] \\
\bigoplus_{x \not\in\mathcal S} h^{p-2,q}(k(x)) \ar[r] & h^{p,q+1}(\OFS)\ar[r] & h^{p,q+1}(F).
\end{tikzcd}
\]
Alternatively, 
apply the Beilinson-Lichtenbaum conjecture for Dedekind domains \cite[Theorem 1.2 (2)]{MR2103541}.
\end{proof}

\begin{lemma}
\label{lem:pic-tau}
Over $\OFS$ the map $\tau\colon h^{2,1} \to h^{2,2}$ is injective with image contained in $\ker(\rho_{2,2})$.
\end{lemma}
\begin{proof}
As in \cite[\S6]{Suslin-Voevodsky:banff} the change of topology adjunction between the Nisnevich and {\'e}tale sites over $\OFS$ (see also \cite[p.~9]{Spitzweck}, \cite[Theorem 1.2.4]{MR2103541}) 
yields a commutative diagram with vertical product maps
\begin{equation}
\label{equation:taudiagram}
\begin{tikzcd}[column sep=10pt]
h^{2,1}\tensor h^{0,1}
\ar[r]\ar[d] & H^{2}_{\et}(\OFS;\mu_{2}(1)) \tensor H^{0}_{\et}(\OFS;\mu_{2}(1)) \ar[d, "\cong"]\\
h^{2,2}
\ar[r, "\cong"] & H^{2}_{\et}(\OFS;\mu_{2}(2)).
\end{tikzcd}
\end{equation}
The lower horizontal and right vertical maps in \eqref{equation:taudiagram} are isomorphisms (use the localization sequence for the horizontal map).
Injectivity of the top horizontal map in \eqref{equation:taudiagram} follows from the commutative diagram of exact coefficient sequences obtained from the change of topology adjunction
\[
\begin{tikzcd}
H^{2,1}/2 \ar[r, "\cong"]\ar[d, "\cong"] & H^{2}_{\et}(\OFS;\Z(1))/2 \ar[d, hook] \\
h^{2,1} \ar[r] & H^{2}_{\et}(\OFS;\Z/2(1)) \cong H^{2}_{\et}(\OFS;\mu_{2}(2)).
\end{tikzcd}
\]
Here the maps exiting the top left corner are isomorphisms (use $H^{3,1}=0$ and the quasi-isomorphism $\Z(1) \simeq \Gm[-1]$ in \cite[Theorem 7.10]{Spitzweck}).
This shows $\tau: h^{2,1} \to h^{2,2}$ is injective.
Its image is contained in $\ker(\rho_{2,2})$ because $\rho h^{2,1} = 0$.
\end{proof}

When $q\geq 2$ the positive motivic cohomology groups $h^{p,q}_{+}$ of $\OO_{F,\mathcal{S}}$ fit into an exact sequence
\begin{equation}
\label{equation:pmc}
0
\to 
h^{0,q}
\overset{\psi_{0,q}}{\to}
(\Z/2)^{r_{1}+r_{2}}
\to 
h^{1,q}_{+}
\to 
h^{1,q}
\overset{\psi_{1,q}}{\to}
(\Z/2)^{r_{1}}
\to 
h^{2,q}_{+}
\to 
h^{2,q}
\overset{\psi_{2,q}}{\to}
(\Z/2)^{r_{1}}
\to
0.
\end{equation}
Here we consider the canonically induced map 
\[
\psi_{p,q}
\colon
h^{p,q}(\OO_{F,\mathcal{S}})
\to 
\bigoplus^{r_{1}} h^{p,q}(\R)
\oplus
\bigoplus^{r_{2}} h^{p,q}(\C).
\]
We refer to \cite{CKPS} for the definition of positive \'etale cohomology groups; %
see \cite[(9)]{MR1928650} for the indexing in \eqref{equation:pmc}.
The exactness of \eqref{equation:pmc} follows by identifying \'etale and motivic cohomology groups in the given range, 
see \cite[Theorem 14.5]{Levine99}.
Recall the subgroup of totally positive units of $\OO_{F,\mathcal{S}}$ is defined by 
\[
\OO_{F,\mathcal{S}}^{\times,+}
:=
\{\alpha\in \OO_{F,\mathcal{S}}^{\times}\,\vert\, \sigma(\alpha)>0, \, \forall \sigma\colon F\to \R\}.
\]
The narrow Picard group of $\OO_{F,\mathcal{S}}$ is defined by the exact sequence
\begin{equation*}
0
\to
\OO_{F,\mathcal{S}}^{\times,+}
\to
F^{\times,+}
\to
\Div(\OO_{F,\mathcal{S}})
\to
\Pic_{+}(\OO_{F,\mathcal{S}})
\to
0.
\end{equation*}
Here $\Div(\OO_{F,\mathcal{S}})$ denotes the group of divisors.
By \cite[Lemma 7.6(a)]{Rognes-Weibel} there is an exact sequence
\begin{equation}
\label{equation:sesPic}
0
\to
\im (h^{1,q}_{+}\to h^{1,q})
\to
h^{1,q}
\overset{\psi_{1,q}}{\to}
(\Z/2)^{r_{1}}
\to 
\Pic_{+}(\OO_{F,\mathcal{S}})
\to
\Pic(\OO_{F,\mathcal{S}})
\to
0.
\end{equation}

Using \eqref{equation:pmc} we note the following result.
\begin{lemma}
\label{lem:H2R-surj}
The naturally induced map 
\[
\psi_{2,q}
\colon
h^{2,q}(\OO_{F,\mathcal{S}})
\to 
\bigoplus^{r_{1}} h^{2,q}(\R)
\]
is surjective for $q\geq 2$.
\end{lemma}

\begin{proposition}(\cite[Theorem14.5 (3), (4), (5)]{Levine99}, \cite[Propositions 6.12, 6.13]{Rognes-Weibel})
\label{lem:RW}
Suppose $\mathcal{S}$ is a set of places of $F$ containing the archimedean and dyadic ones.
\begin{enumerate}
\item $H^{0,q}(\OFS)=\Z$ if $q=0$ and trivial if $q\neq 0$.
\item $H^{1,q}(\OFS)=\Z^{d_{q}}\oplus\Z/w_{q}(F)$,
where $d_{q}=r_{2}$ if $q\geq 2$ is even and $d_{q}=r_{1}+r_{2}$ if $q>1$ is odd.
\item $H^{2,q}(\OFS)$ is a torsion group for all $q\in\Z$.  
If $\# \mathcal{S}<\infty$ then $H^{2,q}(\OFS)$ is finite and the $2$-rank $\rk_{2}H^{2,q}(\OFS;\Z_{2})/2$ equals $r_{1}+\sls+\ts-1$ if $q$ is even and $\sls+\ts-1$ if $q$ is odd.
\end{enumerate}
\end{proposition}

\begin{corollary}
\label{lem:struct}
For the ring $\OFS$ of $\mathcal{S}$-integers in a number field $F$ we have
\begin{equation*}
\# h^{1,1} = 2^{r_{1} + r_{2} + \sls + \ts^+}, 
\# h^{2,2} = 2^{r_{1} + \sls + \ts - 1}, 
\# h^{q,q} = 2^{r_{1}} \text{ if } q \geq 3,
\# \ker(\rho_{1,1}^{2}) = 2^{r_{2} + \sls + \ts^+},  
\end{equation*}
\begin{equation*}
\# \ker(\rho_{2,2}^{2}) = \# \ker(\rho_{2,2}) = 2^{\sls + \ts - 1}, 
\# h^{2,2}/\rho^{2}       = \# h^{2,2}/2 \text{ if } r_{1}>0.
\end{equation*}
\end{corollary}

\begin{lemma}
\label{lem:rhoh11}
For the ring of $\mathcal{S}$-integers $\OFS$ in a number field $F$ the multiplication map $\rho^{q-1}\colon h^{1,1}\to h^{q,q}$ is surjective when $q \geq 3$.
\end{lemma}
\begin{proof}
By \aref{theorem:milnor-ofs} there is a commutative diagram
\[
\begin{tikzcd}
I(\OFS)/I_{2}(\OFS) \ar[r, "\cong"]\ar[d, hook] & h^{1,1}(\OFS) \ar[d, hook]\\
I(F)/I^2(F) \ar[r, "\cong"]\ar[d, "(\langle 1 \rangle - \langle -1\rangle)^{q-1}"] & h^{1,1}(F) \ar[d, "\rho^{q-1}"] \\
I^{q}(F)/I^{q+1}(F) \cong I^{q}(\OFS)/I^{q+1}(\OFS) \ar[r, "\cong"] & h^{q,q}(F) \cong h^{q,q}(\OFS).
\end{tikzcd}
\]
The left vertical composite is surjective.
Indeed, by \cite[Corollary IV.4.5]{MilnorHusemoller} the image of $W(\OFS) \cap I^2(F)$ by the signature map is $4\Z^{r_{1}}$, 
hence $\sigma(I(\OFS)) \supset 4\Z^{r_{1}}$, 
so multiplication by the element $\langle 1 \rangle - \langle -1 \rangle$ corresponding to $\rho$ induces $I^3(\OFS) = I^3(F) \cong 8\Z^{r_{1}}$.
\end{proof}

\begin{lemma}
\label{lem:Hhstruct}
The following holds over the ring of $\mathcal{S}$-integers $\OFS$ in any number field.
\begin{enumerate}
\item
$\pr^{2^{n}}_{2}\colon H_n^{2,q} \to h^{2,q}$ is surjective if $q \equiv 0 \bmod 2$, and it has cokernel of rank $2^{r_{1}}$ if $q \equiv 1 \bmod 2$.
\item
$\partial^{2^{n}}_{2}\colon H_n^{2,q} \to h^{3,q}$ is surjective if $q \equiv 1 \bmod 2$, and trivial if $q \equiv 0 \bmod 2$.
\item 
$H_n^{1,q} \xrightarrow{\pr^{2^{n}}_{2}} h^{1,q} \xrightarrow{\Sq^2} h^{3,q+1}$ is trivial if $q \equiv 0, 1, 2 \bmod 4$.
\item
$H_n^{1,q} \xrightarrow{\partial^{2^{n}}_{2}} h^{2,q} \xrightarrow{\Sq^2} h^{4,q+1}$ is trivial if $q \equiv 1, 2, 3 \bmod 4$, and surjective if $q \equiv 0 \bmod 4$.
\end{enumerate}
\end{lemma}
\begin{proof}
In the proof we make use of the motivic squaring operation $\Sq^2$ and \aref{tbl:Sq-action}.
(1) and (2) follow from the commutative diagrams with exact rows
\[
\begin{tikzcd}
H_n^{2,q}(\OFS) \ar[r, "\pr^{2^{n}}_{2}"] & h^{2,q}(\OFS) \ar[r, "\partial^{2}_{2^{n-1}}"] & H^{3,q}_{n-1}(\OFS) \ar[r]\ar[d, "\cong"]  & H^{3,q}_n(\OFS)\ar[d, "\cong"]  \\
& & \bigoplus^{r_{1}} H^{3,q}_{n-1}(\R) \ar[r, "\cong"]&  \bigoplus^{r_{1}} H^{3,q}_{n}(\R),
\end{tikzcd}
\]
\[
\begin{tikzcd}
H^{2,q}_n(\OFS) \ar[r, "\partial^{2^{n}}_{2}"] & h^{3,q}(\OFS) \ar[r, "\inc^{2}_{2^{n+1}}"] & H^{3,q}_{n+1}(\OFS) \ar[r]\ar[d, "\cong"]  & H^{3,q}_n(\OFS)\ar[d, "\cong"]  \\
& & \bigoplus^{r_{1}} H^{3,q}_{n+1}(\R) \ar[r] &  \bigoplus^{r_{1}} H^{3,q}_{n}(\R).
\end{tikzcd}
\]
The lower horizontal map is trivial if $q \equiv 0\bmod 2$ and an isomorphism if $q \equiv 1\bmod 2$ (see \aref{tbl:Sq-action}).

(3) The Steenrod square $\Sq^2\colon h^{3,q+1}\to h^{5,q+2}$ is an isomorphism if $q \equiv 0, 1\bmod 4$ (see \aref{tbl:Sq-action}).
Thus the composite 
$$
H_n^{1,q} \xrightarrow{\pr^{2^{n}}_{2}} h^{1,q} \xrightarrow{\Sq^2} h^{3,q+1}
$$ 
is trivial if $q \equiv 0, 1\bmod 4$, 
since $\Sq^2\Sq^2\pr^{2^{n}}_{2} = \tau\Sq^3\Sq^1\pr^{2^{n}}_{2} = 0$.
If $q \equiv 2\bmod 4$, $\Sq^2$ acts trivially on $h^{1,q}$ (see \aref{tbl:Sq-action}).

(4) The Steenrod square $\Sq^2\colon h^{4,q+1}\to h^{6,q+2}$ is an isomorphism if $q \equiv 1, 2\bmod 4$ (see \aref{tbl:Sq-action}).
Thus the composite 
$$
H_n^{1,q} \xrightarrow{\partial^{2^{n}}_{2}} h^{2,q} \xrightarrow{\Sq^2} h^{4,q+1}
$$
is trivial if $q \equiv 1, 2\bmod 4$, 
since $\Sq^2\Sq^2\partial^{2^{n}}_{2} = \tau\Sq^3\Sq^1\partial^{2^{n}}_{2} = 0$.
If $q \equiv 3\bmod 4$, $\Sq^2$ acts trivially on $h^{2,q}$ (see \aref{tbl:Sq-action}).
Finally, 
$\Sq^2\partial^{2^{n}}_{2}\inc^{2}_{2^{n}} = \Sq^2\Sq^1 : h^{1,q} \to h^{4,q+1}$ is surjective if $q \equiv 0\bmod 4$ (see \aref{tbl:Sq-action}).
\end{proof}

\begin{lemma}[\protect{\cite[p.~528]{HO}}]
\label{lem:KQ2}
The second hermitian $K$-group of any field $F$ is given by
\[
\KQ_{2,0}(F) 
\cong 
K_{2}(F) \oplus h^{0,1}.
\]
\end{lemma}

\begin{example}
Examples of mod $2$ motivic cohomology rings used in the main body of the paper.
\begin{center}
\begin{tabular}{>{$} l <{$} | >{$} l <{$}}
F & h^{*,*} \\ \hline
\R & \Z/2[\tau,\rho] \\
\Q_{2} & \Z/2[\tau,u,\rho,\pi]/(\pi^{2},u^{2},\rho u,\rho \pi,u \pi + \rho^{2}) \\
\Q_{\ell} & \Z/2[\tau,u,\pi]/(u^{2},\pi^{2}),\ell \equiv 1 \bmod 4 \\
\Q_{\ell} & \Z/2[\tau,\rho,\pi]/(\rho^{2},\pi(\rho + \pi)),\ell \equiv 3 \bmod 4 \\
\FF_{\ell} & \Z/2[\tau,u]/(u^{2}),\ell \equiv 1 \bmod 4 \\
\FF_{\ell} & \Z/2[\tau,\rho]/(\rho^{2}),\ell \equiv 3 \bmod 4
\end{tabular}
\captionof{table}
{
Here $u \in h^{1,1}$ is the class of a nonsquare in $F^{\times}$ not equal to $\rho$ or $\pi$,
and $\pi \in h^{1,1}$ is the class of $\ell \in\Q_{\ell}^{\times}$
}
\label{tbl:mot}
\label{tbl:motQ}
\end{center}
\end{example}

\section{Tables and figures}
\label{sec:tables}

\begin{figure}[h]
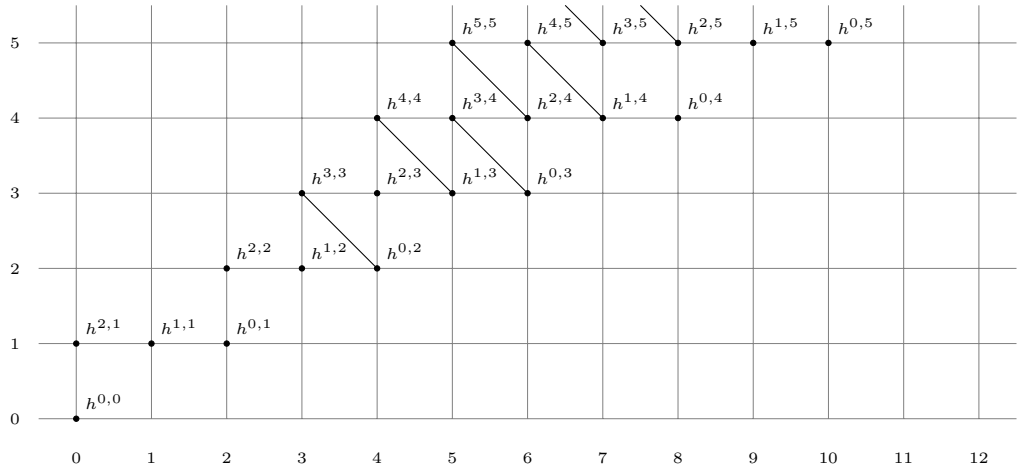

  \begin{center}
    
\begin{center}%
        \end{center}

\caption{The $E^{1}$-page of the weight $0$ slice spectral sequence for $\KGL/2$}
\label{fig:KGL2-E1}
\end{center}
\end{figure}

\begin{table}[h]
\begin{center}

\begin{tabular}{c| >{$}c <{$}| >{$}c <{$}| >{$}c <{$}| >{$}c <{$}| >{$}c <{$}| >{$}c <{$}}
\backslashbox{$p - w$}{$q - p + w$} & 0 & 1 & 2 & 3\\
\hline 0 & \begin{pmatrix}0\\
0\\
0\\
0\\
0\end{pmatrix} & \begin{pmatrix}\rho^{3} \tau^{-2} & 0\\
\rho^{2} \tau^{-1} & 0\\
0 & 0\\
0 & 0\\
0 & 0\end{pmatrix} & \begin{pmatrix}0 & \rho^{4} \tau^{-3} & 0\\
\rho^{2} \tau^{-1} & \rho^{3} \tau^{-2} & 0\\
0 & \rho^{2} \tau^{-1} & 0\\
\tau & 0 & 0\\
0 & 0 & 0\\
0 & 0 & 0\end{pmatrix} & \begin{pmatrix}0 & 0 & 0 & 0\\
\rho^{2} \tau^{-1} & 0 & \rho^{4} \tau^{-3} & 0\\
\rho & \rho^{2} \tau^{-1} & \rho^{3} \tau^{-2} & 0\\
\tau & 0 & \rho^{2} \tau^{-1} & 0\\
0 & \tau & 0 & 0\\
0 & 0 & 0 & 0\\
0 & 0 & 0 & 0\end{pmatrix}\\
1 & \begin{pmatrix}0\\
0\\
\rho\\
0\\
0\end{pmatrix} & \begin{pmatrix}\rho^{3} \tau^{-2} & 0\\
0 & 0\\
\rho & 0\\
\tau & \rho\\
0 & 0\end{pmatrix} & \begin{pmatrix}\rho^{3} \tau^{-2} & \rho^{4} \tau^{-3} & 0\\
\rho^{2} \tau^{-1} & 0 & 0\\
0 & 0 & 0\\
\tau & 0 & 0\\
0 & \tau & \rho\\
0 & 0 & 0\end{pmatrix} & \begin{pmatrix}0 & 0 & 0 & 0\\
\rho^{2} \tau^{-1} & \rho^{3} \tau^{-2} & \rho^{4} \tau^{-3} & 0\\
0 & \rho^{2} \tau^{-1} & 0 & 0\\
0 & 0 & 0 & 0\\
0 & \tau & 0 & 0\\
0 & 0 & \tau & \rho\\
0 & 0 & 0 & 0\end{pmatrix}\\
2 & \begin{pmatrix}0\\
0\\
0\\
\tau\\
0\end{pmatrix} & \begin{pmatrix}\rho^{3} \tau^{-2} & 0\\
0 & 0\\
\rho & 0\\
\tau & 0\\
0 & \tau\end{pmatrix} & \begin{pmatrix}0 & \rho^{4} \tau^{-3} & 0\\
\rho^{2} \tau^{-1} & \rho^{3} \tau^{-2} & 0\\
0 & 0 & 0\\
0 & \rho & 0\\
0 & \tau & 0\\
0 & 0 & \tau\end{pmatrix} & \begin{pmatrix}0 & 0 & 0 & 0\\
0 & 0 & \rho^{4} \tau^{-3} & 0\\
0 & \rho^{2} \tau^{-1} & \rho^{3} \tau^{-2} & 0\\
0 & 0 & 0 & 0\\
0 & 0 & \rho & 0\\
0 & 0 & \tau & 0\\
0 & 0 & 0 & \tau\end{pmatrix}\\
\end{tabular}

\end{center}
\caption{Matrix $M^{1}_{p,q,w}$ used in \aref{thm:KQ2-E2}}
\label{tbl:KQ2-diffs}
\end{table}

\pagebreak

\begin{landscape}
\begin{table}[h]
\begin{center}
\resizebox{\linewidth}{!}{
%

}
\end{center}
\caption{Kernel and image of $M^1_{p,q}$ and $M^2_{p,q}$ used in \aref{thm:KQ2-E2}}
\label{tbl:KQ2-kerim}
\end{table}
\end{landscape}

\begin{figure}
  
\begin{center}%
        
\end{center}

  \caption{$E^2_{p,q,0}(\KQ/2)$}
  \label{fig:KQ2-E2}
\end{figure}

\begin{figure}
  
\begin{center}%
        
\end{center}

  \caption{$E^2_{p,q,1}(\KQ/2)$}
  \label{fig:KQE2_{1}}
\end{figure}
\begin{figure}
  
\begin{center}%
        
\end{center}

  \caption{$E^2_{p,q,2}(\KQ/2)$}
  \label{fig:KQE2_{2}}
\end{figure}
\begin{figure}
  
\begin{center}%
        
\end{center}

  \caption{$E^2_{p,q,3}(\KQ/2)$}
  \label{fig:KQE2_3}
\end{figure}

\clearpage
\begin{landscape}
See \aref{rmk:E2n-convention} for the conventions used in the following figures.

\begin{figure}[h]
\caption{$E^2_{8k+p,4k+q,0}(\KQ/2^{n})$}
\begin{center}
\resizebox{\linewidth}{!}{
%
 H^{0,q'}}\]

\newpage
\begin{figure}[h]
\caption{$E^2_{8k+p+1,4k+q+1,1}(\KQ/2^{n})$}
\begin{center}
\resizebox{\linewidth}{!}{
%
 H^{0,q'}}\]

\newpage
\begin{figure}[h]
\caption{$E^2_{8k+p+2,4k+q+2,2}(\KQ/2^{n})$}
\begin{center}
\resizebox{\linewidth}{!}{
%
 H^{0,q'}}\]

\newpage
\begin{figure}[h]
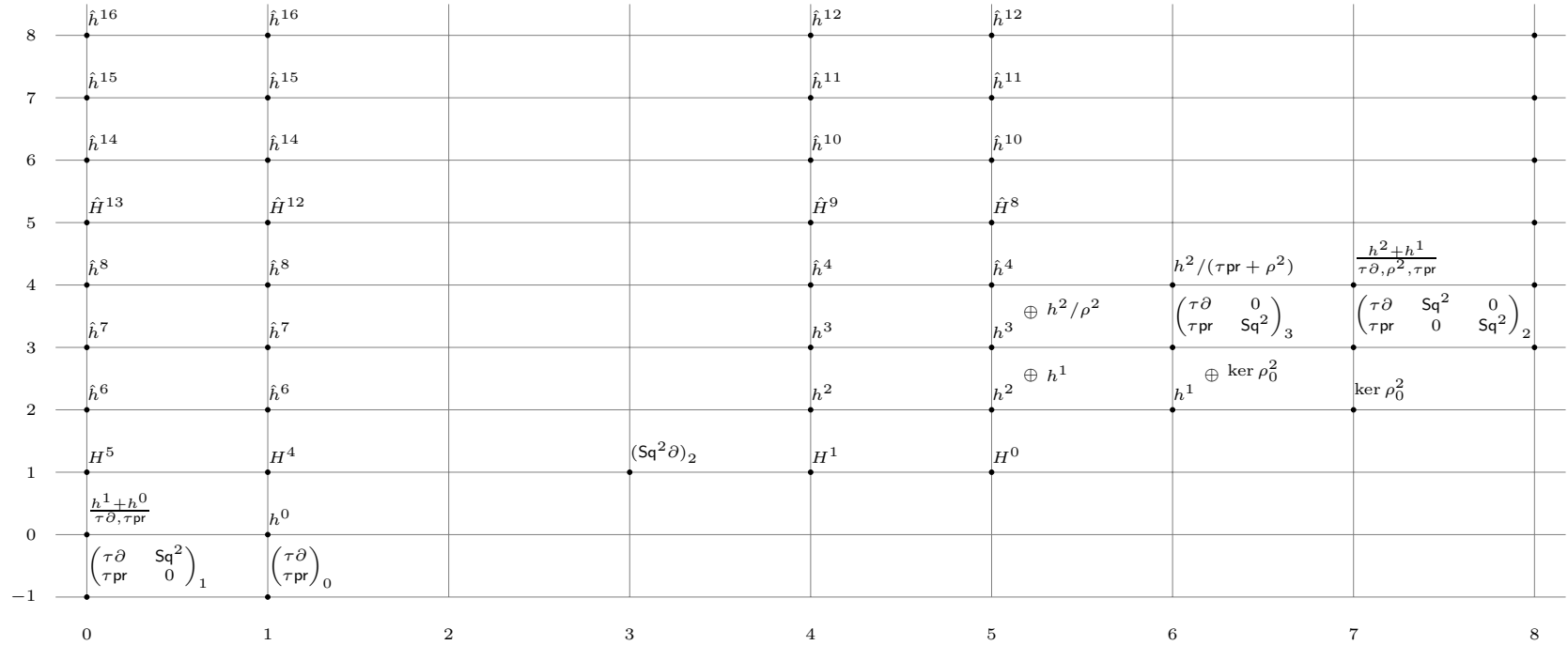
\caption{$E^2_{8k+p+3,4k+q+3,3}(\KQ/2^{n})$}
\begin{center}
\resizebox{\linewidth}{!}{
%
 H^{1,q'} + \rho^2 h^{0,q'}}\]

\end{landscape}

\bibliographystyle{plain}
\bibliography{jop}
\end{document}